\documentclass[DIV=14,letterpaper]{scrartcl}

\date{}

\usepackage{tikz}
\usetikzlibrary{calc}
\usetikzlibrary{matrix}
\usepackage{amsmath,amssymb,amsfonts,amsthm}
\usepackage{algorithm,algorithmic}
\usepackage[pdftex]{hyperref}
\usepackage{graphicx}
\usepackage{subfigure}
\usepackage{mdframed,mathabx}
\usepackage{framed}
\usepackage{multirow}
\usepackage{multicol}
\usepackage{lineno}
\allowdisplaybreaks

\newtheorem{theorem}{Theorem}[section]
\newtheorem{pro}{Proposition}

\theoremstyle{remark}

\title{Approximating and preconditioning the stiffness matrix in the GoFD approximation of the fractional {Laplacian}}

\author{Weizhang Huang\thanks{Department of Mathematics, University of Kansas, Lawrence, Kansas, U.S.A.
({\em whuang@ku.edu}).}
\and
Jinye Shen\thanks{School of Mathematics, Southwestern University of Finance and Economics, Chengdu, Sichuan 611130, China ({\em jyshen@swufe.edu.cn}).}
}

\begin{document}
\vskip 1cm
\maketitle

\begin{abstract}
In the finite difference approximation of the fractional Laplacian the stiffness matrix is typically dense and
needs to be approximated numerically. The effect of the accuracy in approximating the stiffness matrix
on the accuracy in the whole computation is analyzed and shown to be significant. Four such approximations
are discussed. While they are shown to work well with the recently developed
grid-over finite difference method (GoFD) for the numerical solution of boundary value problems of the fractional
Laplacian, they differ in accuracy, economics to compute, performance of preconditioning,
and asymptotic decay away from the diagonal line. In addition, two preconditioners based on
sparse and circulant matrices are discussed for the iterative solution of linear systems associated
with the stiffness matrix. Numerical results in two and three dimensions are presented.
\end{abstract}



\section{Introduction}

We are concerned with the finite difference (FD) solution of the boundary value problem (BVP) of the fractional Laplacian,
\begin{equation}
\label{BVP-1}
\begin{cases}
(-\Delta)^{s} u  = f, & \text{ in } \Omega
\\
u = 0, & \text{ in } \Omega^c
\end{cases}
\end{equation}
where $(-\Delta)^{s}$ is the fractional Laplacian with the fractional order $s \in (0,1)$, $\Omega$ is a bounded
domain in $\mathbb{R}^d$ ($d \ge 1$), $\Omega^c \equiv \mathbb{R}^d\setminus \Omega$ is the complement of $\Omega$,
and $f$ is a given function.
The fractional Laplacian can be expressed in the singular integral form as
\begin{equation}
\label{FL-1}
(-\Delta)^s u(\vec{x})= \frac{2^{2s}s \Gamma(s+\frac{d}{2})}{\pi^{\frac d 2}\Gamma (1-s)} \text{ p.v. } \int_{\mathbb{R}^d}
\frac{u(\vec{x})-u(\vec{y})}{|\vec{x}-\vec{y}|^{d+2s}}d \vec{y},
\end{equation}
or in terms of the Fourier transform as
\begin{equation}
\label{FL-F}
(-\Delta)^s u = \mathcal{F}^{-1}( |\vec{\xi}|^{2s} \mathcal{F}(u)),
\end{equation}
where p.v. stands for the Cauchy principal value, $\Gamma(\cdot)$ is the gamma function, and  $\mathcal{F}$ and $\mathcal{F}^{-1}$
denote the Fourier and inverse Fourier transforms, respectively.
When $\Omega$ is a simple domain such as a rectangle or a cube, BVP (\ref{BVP-1}) can be solved
using finite differences on a uniform grid (see, e.g., \cite{Hao2021,HS-2024-GoFD}).
When $\Omega$ is an arbitrary bounded domain (including a simple domain), BVP (\ref{BVP-1}) can be
solved using the recently developed grid-overlay finite difference (GoFD) method
with a simplicial mesh \cite{HS-2024-GoFD} or a point cloud \cite{SSH-2024-meshfreeGoFD}.
Generally speaking, the stiffness matrix in FD approximations is dense and needs to be approximated numerically.
The effect of the accuracy in the approximation on the accuracy in the numerical solution of BVP (\ref{BVP-1})
has not been studied in the past.
A main objective of the present work is to study this important issue for the numerical approximation of the fractional Laplacian.
It will be shown that the effect is actually
significant.  As a result, it is necessary to develop accurate and reasonably economic approximations for the stiffness matrix.
We will study four approximations. The first two are based on the fast Fourier transform (FFT) with uniform
and non-uniform sampling points. The third one is the spectral approximation of Zhou and Zhang \cite{Zhang2023}.
We will discuss a new and fast implementation of this approximation and derive the asymptotic decay rate of its entries
away from the diagonal line. The last one is a modification of the spectral approximation.
Properties of these approximations are summarized in Table~\ref{table:approximate-T-3}.
In addition, we will study two preconditioners based on sparse and circulant matrices for the iterative
solution of linear systems associated with the stiffness matrix.

For the purpose of numerical verification and demonstration, we consider an example of BVP (\ref{BVP-1})  with
the analytical exact solution (cf. \cite[Theorem 3]{Dyda2017}),
\begin{equation}
\label{main-example}
\Omega = B_1(0),\quad  f = 1,\quad u =  \frac{\Gamma(\frac{d}{2})}{2^{2 s} \Gamma(1+s) \Gamma(\frac{d}{2}+s)} (1-|\vec{x}|^2)_{+}^s ,
\end{equation}
where $B_1(0)$ is the unit ball centered at the origin. Numerical results will be given in two and three dimensions.

The fractional Laplacian is a fundamental non-local operator in the modeling of anomalous dynamics; see, for example,
\cite{Antil-2022,Huang2016,Lischke-2020} and references therein.
A number of numerical methods have been developed, including
FD methods
\cite{DuNing2019,Duo-2018,Duo-2015,Hao2021,Huang2014,Huang2016,Ying-2020,Pang-2012,Sunjing2021,Wang2012,Zhang-Jiwei-2016},
finite element methods
\cite{Acosta2017,Acosta201701,Ainsworth-2017,Ainsworth-2018,Bonito2022,Bonito2019,Faustmann2022,Tian2013},
spectral methods \cite{Song2017,LiHuiyuan2022,Zhang2023},
discontinuous Galerkin methods \cite{Du2019},
meshfree methods \cite{Burkardt-2021,Pang-2015}, neural network method \cite{Hao2023}, and
sinc-based methods \cite{Antil-2021}.
Loosely speaking, methods such as FD methods are constructed on uniform grids and have the advantage
of efficient matrix-vector multiplication via FFT but do not work for complex domains and with mesh adaptation.
On the other hand, methods such as finite element methods
can work for arbitrary bounded domains and with mesh adaptation
but suffer from slowness of stiffness matrix assembling and matrix-vector multiplication because the stiffness matrix is dense.
Few methods can do both.
A sparse approximation to the stiffness matrix of finite element methods and an efficient multigrid implementation
have been proposed by Ainsworth and Glusa \cite{Ainsworth-2017,Ainsworth-2018}.
The GoFD method of \cite{HS-2024-GoFD,SSH-2024-meshfreeGoFD} is an FD method that uses FFT for fast computation
and works for arbitrary bounded domains with unstructured meshes or point clouds.
Anisotropic nonlocal diffusion models and inhomogeneous fractional Dirichlet problem are also studied in \cite{Elia2022,Sunjing2022}.

An outline of the present paper is as follows. The GoFD method and uniform-grid FD approximation of the fractional Laplacian
are described in Sections~\ref{sec:GoFD} and \ref{SEC:AFD}, respectively.
The effect of the accuracy in the approximation of the stiffness matrix on the accuracy in the numerical solution of BVP (\ref{BVP-1})
is analyzed in Section~\ref{SEC:error-analysis-4-T} while four approximations to the stiffness matrix are discussed in
Section~\ref{SEC:approximating-T}.
Section~\ref{SEC:preconditioning} is devoted to the study of two preconditioners, a sparse preconditioner and a circulant preconditioner,
for iteratively solving linear systems associated with the stiffness matrix.
Numerical results in two and three dimensions are presented in Section~\ref{SEC:numerics}. Finally, conclusions are drawn in Section~\ref{SEC:conclusions}.

\section{The grid-overlay FD method}
\label{sec:GoFD}

To describe the GoFD method \cite{HS-2024-GoFD} for the numerical solution of BVP (\ref{BVP-1}),
we assume that an unstructured simplicial mesh $\mathcal{T}_h$ is given for $\Omega$.
Denote its vertices by $\vec{x}_1,\, ...,\, \vec{x}_{N_{vi}},\, ..., \, \vec{x}_{N_v}$, where $N_{v}$
and $N_{vi}$ are the numbers of the vertices and interior vertices,  respectively.
Here, the vertices are arranged so that the interior vertices are listed before the boundary ones.
We also assume that a $d$-dimensional cube, $\Omega_{\text{FD}} \equiv (-R_{\text{FD}},R_{\text{FD}})^d$,
has been chosen to contain $\Omega$.
For a given positive integer $N_{\text{FD}}$, let  $\mathcal{T}_{\text{FD}}$ be an uniform grid
for $\Omega_{\text{FD}}$ with $2N_{\text{FD}}+1$ nodes in each axial direction.
The spacing of $\mathcal{T}_{\text{FD}}$ is given by
\begin{equation}
\label{h-1}
h_{\text{FD}} = \frac{R_{\text{FD}}}{N_{\text{FD}}}.
\end{equation}
Denote the vertices of $\mathcal{T}_{\text{FD}}$ by $\vec{x}_k^{\text{FD}}$, $k = 1, ..., N_{v}^{\text{FD}}$.
Then, a uniform-grid FD approximation, denoted by $h_{\text{FD}}^{-2 s} A_{\text{FD}}$,
can be developed for the fractional Laplacian on $\mathcal{T}_{\text{FD}}$; cf. Section~\ref{SEC:AFD}.
It is known \cite{HS-2024-GoFD} that
$A_{\text{FD}}$ is symmetric and positive definite.
The GoFD approximation of BVP \eqref{BVP-1} is then defined as
\begin{equation}
\label{GoFD-1}
A_h \vec{u}_h = \vec{f}_h ,
\quad A_h \equiv \frac{1}{h_{\text{FD}}^{2 s}} D_h^{-1} (I_{h}^{\text{FD}})^T A_{\text{FD}} I_{h}^{\text{FD}},
\end{equation}
where $\vec{u}_h = (u_1, ..., u_{N_{vi}})^T$ is an FD approximation of $u$,
$\vec{f}_{h} = (f(\vec{x}_1), ..., f(\vec{x}_{N_{vi}}))^T$,
$I_h^{\text{FD}}$ is a transfer matrix from the mesh $\mathcal{T}_h$ to the uniform grid $\mathcal{T}_{\text{FD}}$, and
$D_h$ is the diagonal matrix formed by the column sums of $I_h^{\text{FD}}$. The invertibility of $D_h$ will be addressed
in Theorem~\ref{thm:Ah-1} below. Rewrite $A_h$ as
\[
A_h =  \frac{1}{h_{\text{FD}}^{2 s}} D_h^{-\frac{1}{2}} \cdot (I_{h}^{\text{FD}} D_h^{-\frac{1}{2}})^T A_{\text{FD}}
( I_{h}^{\text{FD}} D_h^{-\frac{1}{2}}) \cdot D_h^{\frac{1}{2}} .
\]
Thus, $A_h$ is similar to the symmetric and positive semi-definite matrix
\[
 \frac{1}{h_{\text{FD}}^{2 s}}  (I_{h}^{\text{FD}} D_h^{-\frac{1}{2}})^T A_{\text{FD}}
( I_{h}^{\text{FD}} D_h^{-\frac{1}{2}}),
\]
which can be shown to be positive definite when $I_{h}^{\text{FD}}$ is of full column rank.
Notice that  $D_h^{-1}$ is included in (\ref{GoFD-1}) to ensure that
the row sums of $D_h^{-1} (I_{h}^{\text{FD}})^T$ be equal to one. As a consequence,
$D_h^{-1} (I_{h}^{\text{FD}})^T$ represents a data transfer from $\mathcal{T}_{\text{FD}}$
to $\mathcal{T}_h$ and preserves constant functions.

We now consider a special choice of $I_h^{\text{FD}}$ as linear interpolation.
For any function $u(\vec{x})$, we can express its piecewise linear interpolant as
\begin{equation}
\label{interp-1}
I_h u (\vec{x}) = \sum_{j=1}^{N_{v}} u_j \phi_j(\vec{x}),
\end{equation}
where $u_j = u(\vec{x}_j)$ and $\phi_j$ is the Lagrange-type linear basis function associated with vertex $\vec{x}_j$.
(The basis function $\phi_j$ is considered to be zero outside the domain $\Omega$, i.e., $\phi_j |_{\Omega^c} = 0$.)
Then,
\[
I_h u (\vec{x}_k^{\text{FD}}) = \sum_{j=1}^{N_v} u_j \phi_j(\vec{x}_k^{\text{FD}}) , \quad k = 1, ..., N_v^{\text{FD}}
\]
which gives rise to
\begin{equation}
\label{IhFD-1}
(I_h^{\text{FD}})_{k,j} = \phi_j(\vec{x}_k^{\text{FD}}), \quad k = 1, ..., N_v^{\text{FD}}, \; j = 1, ..., N_v .
\end{equation}

\begin{theorem}[\cite{HS-2024-GoFD}]
\label{thm:Ah-1}
Assume that $N_{\text{FD}}$ is taken sufficiently large such that
\begin{equation}
h_{\text{FD}} \le \frac{a_h}{(d+1)\sqrt{d}},
\label{hFD-1}
\end{equation}
where $a_h$ is the minimum element height of $\mathcal{T}_h$. Then, the transfer matrix $I_h^{\text{FD}}$ associated with piecewise
linear interpolation is of full column rank. As a result, $D_h$ is invertible and the GoFD stiffness matrix given in (\ref{GoFD-1})
for the fractional Laplacian is similar to a symmetric and positive definite matrix and thus invertible.
\end{theorem}

The sufficient condition (\ref{hFD-1}) for the invertibility of $A_h$ is
conservative in general. Numerical experiment suggests that a less restrictive condition, such as
\begin{equation}
\label{hFD-2}
h_{\text{FD}} \le a_h,
\end{equation}
can be used in practical computation.

The linear algebraic system (\ref{GoFD-1}) can be rewritten as
\begin{equation}
\label{GoFD-2}
(I_{h}^{\text{FD}})^T A_{\text{FD}} I_{h}^{\text{FD}} \vec{u}_h = h_{\text{FD}}^{2 s} D_h \vec{f}_h .
\end{equation}
While $I_h^{\text{FD}}$ is sparse, $A_{\text{FD}}$ is dense and needs to be approximated numerically
in multi-dimensions (cf. Section~\ref{SEC:AFD}).
Unfortunately, this has a consequence that the accuracy of approximating $A_{\text{FD}}$ can have
a significant impact on the accuracy in the FD solution of BVP (\ref{BVP-1}).
This impact will be analyzed in Section~\ref{SEC:error-analysis-4-T}.
Moreover, several approaches for approximating the stiffness matrix will be discussed in Section~\ref{SEC:approximating-T}.

It is worth pointing out that although the coefficient matrix of (\ref{GoFD-2}) is dense, its multiplication with vectors
can be carried out efficiently using FFT (cf. Section~\ref{SEC:AFD}). Therefore, it is practical to use an iterative method
for solving (\ref{GoFD-2}).
In our computation, we use the preconditioned conjugate gradient (PCG) method since the coefficient matrix is symmetric
and positive definite when the transfer matrix is of full column rank.
Preconditioning will be discussed in Section~\ref{SEC:preconditioning}.

\section{The uniform-grid FD approximation of the fractional Laplacian}
\label{SEC:AFD}

In this section we describe the FD approximation $A_{\text{FD}}$ of the fractional Laplacian
on a uniform grid $\mathcal{T}_{\text{FD}}$ through its Fourier transform formulation.
We also discuss the fast computation of the multiplication of $A_{\text{FD}}$ with vectors using FFT.
For notational simplicity, we restrict the discussion in two dimensions (2D) in this section.
The description of $A_{\text{FD}}$ in general $d$-dimensions is similar.

In terms of the Fourier transform  (\ref{FL-F}), the fractional Laplacian reads as
\begin{align}
\label{FL-2}
(-\Delta)^{s} u (x,y) = \frac{1}{(2\pi)^2} \int_{-\infty}^{\infty} \int_{-\infty}^{\infty}  \widehat{(-\Delta)^s u}(\xi,\eta)
e^{i x \xi+i y \eta} d \xi d \eta ,
\end{align}
where
\begin{align}
\label{FL-3}
\widehat{(-\Delta)^s u}(\xi,\eta) = (\xi^2+\eta^2)^{s} \hat{u}(\xi,\eta)
\end{align}
and $\hat{u}$ is the continuous Fourier transform of $u$.
Consider a 2D uniform grid  with spacing $h_{\text{FD}}$,
\[
(x_j, y_k) = (j h_{\text{FD}}, k h_{\text{FD}}),\quad j, k = 0, \pm 1, \pm 2, ... .
\]
Applying the discrete Fourier transform to the 5-point central FD approximation of the Laplacian, we obtain the
the discrete Fourier transform of an FD approximation of the fractional Laplacian as
\begin{align}
\label{FL-4}
\widecheck{(-\Delta_h)^s u} (\xi, \eta) & = \frac{1}{h_{\text{FD}}^{2 s}}
\left (4 \sin^{2}(\frac{\xi h_{\text{FD}}}{2}) + 4 \sin^{2}(\frac{\eta h_{\text{FD}}}{2})\right )^{s} \check{u}(\xi,\eta) ,
\end{align}
where $\check{u}(\xi,\eta)$ is the discrete Fourier transform of $u$ on the uniform grid, i.e.,
\[
\check{u}(\xi,\eta) = \sum_{m,n} u_{m,n} e^{- i x_m \xi - i y_n \eta} , \quad (\xi, \eta) \in (-\frac{\pi}{h_{\text{FD}}}, \frac{\pi}{h_{\text{FD}}})
\times (-\frac{\pi}{h_{\text{FD}}}, \frac{\pi}{h_{\text{FD}}}) .
\]
Then, the uniform-grid FD approximation of the fractional Laplacian is given by
\begin{align}
(-\Delta_h)^s u(x_j,y_k) & = \frac{h_{\text{FD}}^2}{ (2\pi)^2}
\int_{-\frac{\pi}{h_{\text{FD}}}}^{\frac{\pi}{h_{\text{FD}}}} \int_{-\frac{\pi}{h_{\text{FD}}}}^{\frac{\pi}{h_{\text{FD}}}}
\widecheck{(-\Delta_h)^s u} (\xi,\eta) e^{i x_j \xi + i y_k \eta} d \xi d\eta
\notag \\
& = \frac{1}{(2\pi)^2 h_{\text{FD}}^{2 s}} \sum_{m=-\infty}^{\infty} \sum_{n=-\infty}^{\infty} u_{m,n}
\int_{-\pi}^{\pi} \int_{-\pi}^{\pi}  \psi(\xi,\eta)  e^{i (j-m) \xi + i (k-n) \eta} d \xi d\eta ,
\notag
\end{align}
where
\begin{align}
\label{psi-1}
\psi(\xi,\eta) = \left (4 \sin^{2}\left (\frac{\xi}{2}\right ) + 4 \sin^{2}\left (\frac{\eta}{2}\right )\right )^{s} .
\end{align}
We can rewrite this into
\begin{align}
(-\Delta_h)^s u(x_j,y_k) & = \frac{1}{h_{\text{FD}}^{2 s}} \sum_{m=-\infty}^{\infty} \sum_{n=-\infty}^{\infty} T_{j-m,k-n} u_{m,n} ,
\label{FL-5}
\end{align}
where
\begin{align}
\label{T-1}
& T_{p,q} = \frac{1}{(2\pi)^2} \int_{-\pi}^{\pi} \int_{-\pi}^{\pi}  \psi(\xi,\eta) e^{i p \xi + i q \eta} d \xi d\eta,
\quad p, q = 0, \pm 1, \pm 2, ... .
\end{align}
Obviously, $T_{p,q}$'s are the Fourier coefficients of the function $\psi(\xi,\eta)$.
It is not difficult to see
\begin{equation}
\label{T-symmetry}
T_{-p,-q} = T_{p,q}, \quad T_{-p,q} = T_{p,q}, \quad T_{p,-q} = T_{p,q} .
\end{equation}
Except for one dimension, the matrix $T$ needs to be approximated numerically
(cf. Section~\ref{SEC:approximating-T}).

For a function $u$ vanishing in $\Omega^c$, (\ref{FL-5}) reduces to a finite double summation as
\begin{align}
\label{FL-8}
(-\Delta_h)^s u(x_j,y_k) = \frac{1}{h_{\text{FD}}^{2 s}}
\sum_{m=-N_{\text{FD}}}^{N_{\text{FD}}} \sum_{n=-N_{\text{FD}}}^{N_{\text{FD}}} T_{j-m, k-n} u_{m,n},\qquad -N_{\text{FD}} \le j, k \le N_{\text{FD}} .
\end{align}
From this, we obtain the FD approximation of the fractional Laplacian on the uniform mesh $\mathcal{T}_{\text{FD}}$
(after ignoring the factor $1/h_{\text{FD}}^{2 s}$) as
\begin{equation}
\label{AFD-1}
A_{\text{FD}} = \left ( \frac{}{} A_{(j,k),(m,n)} = T_{j-m, k-n} \right )_{(2N_{\text{FD}}+1)^2\times (2N_{\text{FD}}+1)^2}, \quad -N_{\text{FD}} \le j, k, m, n \le N_{\text{FD}} .
\end{equation}
This indicates that $A_{\text{FD}}$ is a block Toeplitz matrix with Toeplitz blocks. Moreover,
it is known \cite{HS-2024-GoFD} that $A_{\text{FD}}$ is symmetric and positive definite.
For easy reference but without causing confusion, hereafter both matrices $A_{\text{FD}}$ and $T$ will be referred to
as the stiffness matrix of the uniform-grid FD approximation.

The multiplication of $A_{\text{FD}}$ with vectors can be computed efficiently using FFT.
Assume that an approximation of $T$ has been obtained. We first compute the discrete Fourier transform of $T$. Notice that
the entries of $T$ involved in (\ref{FL-8}) are: $T_{m,n}$, $-2 N_{\text{FD}} \le m, n \le 2 N_{\text{FD}}$.
Thus, the discrete Fourier transform of $T$ is given by
\begin{align}
& \check{T}_{p,q} = \sum_{m=-2N_{\text{FD}}}^{2N_{\text{FD}}-1} \sum_{n=-2N_{\text{FD}}}^{2N_{\text{FD}}-1} T_{m,n}
e^{-\frac{i 2 \pi (m+2N_{\text{FD}})(p+2N_{\text{FD}})}{4N_{\text{FD}}}
-\frac{i 2 \pi (n+2N_{\text{FD}})(q+2N_{\text{FD}})}{4N_{\text{FD}}}},
\notag \\
& \qquad \qquad \qquad \qquad \qquad \qquad \qquad \qquad \qquad \qquad \qquad \qquad \qquad
- 2N_{\text{FD}} \le p, q \le 2 N_{\text{FD}} .
\label{T-3-0}
\end{align}
Here, $2N_{\text{FD}}$ has been added to the indices $m$, $n$, $p$, and $q$ so that the corresponding indices
have the range from $0$ to $4N_{\text{FD}}-1$. As a result, Matlab's function \textbf{fft.m} can be used
for computing (\ref{T-3-0}) directly.
Then, $T_{m,n}$'s can be expressed using the inverse discrete Fourier transform as
\[
T_{m,n} = \frac{1}{(4N_{\text{FD}})^2} \sum_{p=-2N_{\text{FD}}}^{2N_{\text{FD}}-1} \sum_{q=-2N_{\text{FD}}}^{2N_{\text{FD}}-1} \check{T}_{p,q} e^{\frac{i 2 \pi (m+2N_{\text{FD}})(p+2N_{\text{FD}})}{4N_{\text{FD}}}
+ \frac{i 2 \pi (n+2N_{\text{FD}})(q+2N_{\text{FD}})}{4N_{\text{FD}}}} .
\]
From this, we have
\begin{align*}
(A_{\text{FD}}\vec{u}_{\text{FD}})_{(j,k)}
& = \sum_{m = -N_{\text{FD}}}^{N_{\text{FD}}} \sum_{n=-N_{\text{FD}}}^{N_{\text{FD}}} T_{j-m,k-n} u_{m,n}
\notag \\
& = \frac{1}{(4N_{\text{FD}})^2} \sum_{p=-2N_{\text{FD}}}^{2N_{\text{FD}}-1} \sum_{q=-2N_{\text{FD}}}^{2N_{\text{FD}}-1} \check{T}_{p,q}
 \sum_{m = -N_{\text{FD}}}^{N_{\text{FD}}} \sum_{n=-N_{\text{FD}}}^{N_{\text{FD}}} u_{m,n}
 \notag \\
 & \qquad \qquad \qquad \qquad \qquad \qquad
 \cdot  e^{\frac{i 2 \pi (j-m+2N_{\text{FD}})(p+2N_{\text{FD}})}{4N_{\text{FD}}}
+ \frac{i 2 \pi (k-n+2N_{\text{FD}})(q+2N_{\text{FD}})}{4N_{\text{FD}}}} .
\end{align*}
This can be rewritten as
\begin{align}
& (A_{\text{FD}}\vec{u}_{\text{FD}})_{(j,k)}
= \frac{1}{(4N_{\text{FD}})^2} \sum_{p=-2N_{\text{FD}}}^{2N_{\text{FD}}-1} \sum_{q=-2N_{\text{FD}}}^{2N_{\text{FD}}-1} \check{T}_{p,q} \check{u}_{p,q}
(-1)^{p+2N_{\text{FD}} + q+2N_{\text{FD}}}
\notag \\
& \qquad \qquad \qquad \qquad \qquad \qquad \qquad \qquad \qquad \cdot
e^{\frac{i 2 \pi (p+2N_{\text{FD}})(j+2 N_{\text{FD}})}{4N_{\text{FD}}} + \frac{i 2 \pi (q+2N_{\text{FD}})(k+2 N_{\text{FD}})}{4N_{\text{FD}}}} ,
\label{AFD-2}
\end{align}
where
\begin{align}
\check{u}_{p,q}
& = \sum_{m = -2 N_{\text{FD}}}^{2 N_{\text{FD}}-1} \sum_{n=-2 N_{\text{FD}}}^{2 N_{\text{FD}}-1}
\tilde{u}_{m,n} e^{-\frac{i 2 \pi (m+2N_{\text{FD}})(p+2N_{\text{FD}})}{4N_{\text{FD}}}
-\frac{ i 2 \pi (n+2N_{\text{FD}})(q+2N_{\text{FD}})}{4N_{\text{FD}}}},
\label{AFD-3}
\\
\tilde{u}_{m,n} & = \begin{cases} u_{m,n}, & \text{ for } - N_{\text{FD}} \le m, n \le N_{\text{FD}} \\ 0, &\text{ otherwise} .
\end{cases}
\notag
\end{align}

Notice that that (\ref{T-3-0}), (\ref{AFD-2}), and (\ref{AFD-3}) can be computed using FFT.
Since the cost of FFT is $\mathcal{O}(N_{\text{FD}}^d \log (N_{\text{FD}}^d))$ flops (in $d$-dimensions),
the number of the flops needed to carry out the multiplication of $A_{\text{FD}}$ with a vector is
\begin{equation}
\label{AFD-cost}
\mathcal{O}(N_{\text{FD}}^d \log (N_{\text{FD}}^d)) ,
\end{equation}
which is almost linearly proportional to the number of nodes of the uniform grid $\mathcal{T}_{\text{FD}}$.
Comparatively, the direction computation of the multiplication of $A_{\text{FD}}$ with a vector without using FFT requires
$\mathcal{O}(N_{\text{FD}}^{2d})$ flops.

\section{Importance of approximating the stiffness matrix $T$ accurately}
\label{SEC:error-analysis-4-T}

The accuracy in approximating $T$ can have a significant impact on the accuracy of computing
$A_{\text{FD}}\vec{u}_{\text{FD}}$ (and then the accuracy in the numerical solution of BVP (\ref{BVP-1}))
due to the fact that $A_{\text{FD}}$ is a dense matrix.
To explain this, let us assume the error in approximating $T$ is $10^{-\delta}$, where $\delta$ is a positive integer.
Then the error in computing $A_{\text{FD}}\vec{u}_{\text{FD}}$ (cf. (\ref{FL-8}))
 is bounded (in $d$-dimensions) by
\begin{equation}
\label{AFD-bound-1}
\frac{1}{R_{\text{FD}}^{2 s}} (2N_{\text{FD}}+1)^d N_{\text{FD}}^{2 s} \times 10^{-\delta} \| u \|_{L^{\infty}(\Omega)}.
\end{equation}
For fixed $R_{\text{FD}}$, $\delta$, and $\| u \|_{L^{\infty}(\Omega)}$, this bound goes to infinity as $N_{\text{FD}} \to \infty$.
This implies that the error in computing $A_{\text{FD}}\vec{u}_{\text{FD}}$ and therefore, the error in the whole computation
of BVP (\ref{BVP-1}), will be dominated by the error in approximating $T$ at some point as the mesh is refined
and increase after that point if the mesh is further refined.

To see this more clearly, we take $R_{\text{FD}} = 1$ and $\| u \|_{L^{\infty}(\Omega)} = 1$ and assume that
the discretization error for BVP (\ref{BVP-1}) is $\mathcal{O}(1/N_{\text{FD}})$ (first-order) or $\mathcal{O}(1/N_{\text{FD}}^2)$ (second-order).
The bound in (\ref{AFD-bound-1}) is plotted as a function of $N_{\text{FD}}$ in Fig.~\ref{fig:Err_in_T_1}
against $1/N_{\text{FD}}$ and $1/N_{\text{FD}}^2$ for two levels of the accuracy in approximating $T$
($10^{-9}$ and $10^{-12}$) and in 2D and 3D.
From the figure we can see that there exists an equilibrium point $(N_{\text{FD}}^{e}, E^{e})$ where the bound (\ref{AFD-bound-1})
and the discretization error intersect. For example, in Fig.~\ref{fig:Err_in_T_1}(b) we have
$(N_{\text{FD}}^{e}, E^{e}) \approx (700, 1.5\times 10^{-3})$
and $(200, 3\times 10^{-5})$ for the first-order and second-order discretization error, respectively.
Generally speaking, for $N_{\text{FD}} \le N_{\text{FD}}^{e}$, the overall error is dominated by the discretization
error (regardless of the order) and decreases as $N_{\text{FD}}$ increases.
For $N_{\text{FD}} \ge N_{\text{FD}}^{e}$, on the other hand,  the overall error is dominated
by the error in approximating $T$ and increases with $N_{\text{FD}}$.
Thus, for a fixed level of accuracy in approximating $T$,
\textbf{the total error in the computed solution is expected to decrease for small $N_{\text{FD}}$
but then increase as $N_{\text{FD}}$ increases}.
The turning of the error behavior occurs at larger $N_{\text{FD}}$ and larger $E^e$ for the first order discretization error
than the second-order one.
Moreover, $N_{\text{FD}}^{e}$ becomes larger and $E^{e}$ becomes smaller when $T$ is approximated more accurately.
Finally, the figure shows that $N_{\text{FD}}^{e}$ is smaller for 3D than for 2D for the same level
of accuracy in approximating $T$. For example, in Fig.~\ref{fig:Err_in_T_1}(d) $N_{\text{FD}}^{e} \approx 180$ and $70$
for the first-order and second-order discretization error, respectively. These are compared
with $N_{\text{FD}}^{e} \approx 700$ and $200$ for 2D in Fig.~\ref{fig:Err_in_T_1}(b).

To further demonstrate the effect, we consider Example (\ref{main-example}) with GoFD with the FFT approximation of the stiffness matrix $T$
(cf. Section~\ref{SEC:T-FFT}). The accuracy in approximating $T$ in this approach is determined
by the number of nodes used in each axial direction, $M$. The solution error in $L^2$ norm is plotted as a function of $N_{\text{FD}}$
for a fixed $M$ in Fig.~\ref{fig:Err_GoFD_4_T} for 2D and 3D. Quasi-uniform meshes are used for $\Omega$ for the computation.
It is known numerically that the error of GoFD converges at the rate $\mathcal{O}(h^{\min(1, s+0.5)})$ for quasi-uniform meshes.
We can see that for small $N_{\text{FD}}$, the error decrease at this expected rate.
But after a certain $N_{\text{FD}}$, the error begins to increase.
The turning of the error behavior occurs at a larger $N_{\text{FD}}$ for a larger $M$ (corresponding to a  higher level
of accuracy in approximating $T$).
These results confirm the analysis made earlier in this section.

It should be pointed out that the effect of the error in approximating the stiffness matrix on the overall computational accuracy
is not unique to the FD/GoFD approximation discussed in this work. It can also happen with finite element and other FD approximations
of the fractional Laplacian where the stiffness matrix is dense and needs to be approximated numerically.

\begin{figure}[ht!]
\centering
\subfigure[$\delta = 9$ ($10^{-9}$), in 2D]{
\includegraphics[width=0.33\linewidth]{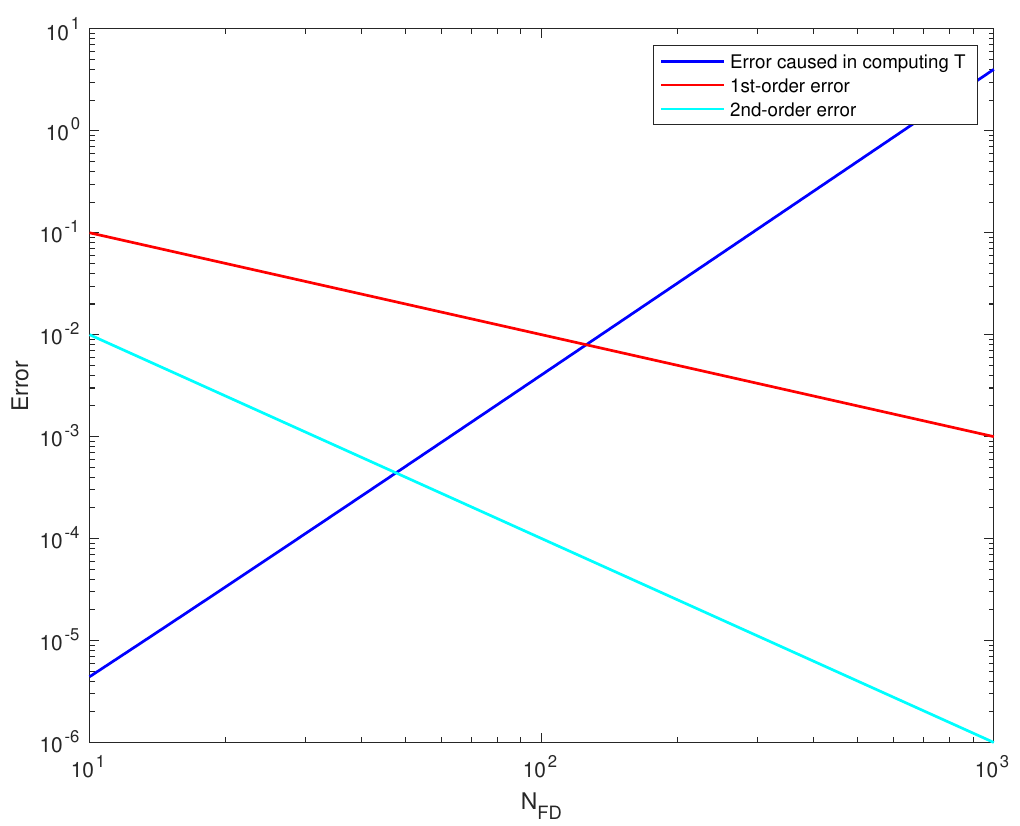}
}
\subfigure[$\delta = 12$ ($10^{-12}$), in 2D]{
\includegraphics[width=0.33\linewidth]{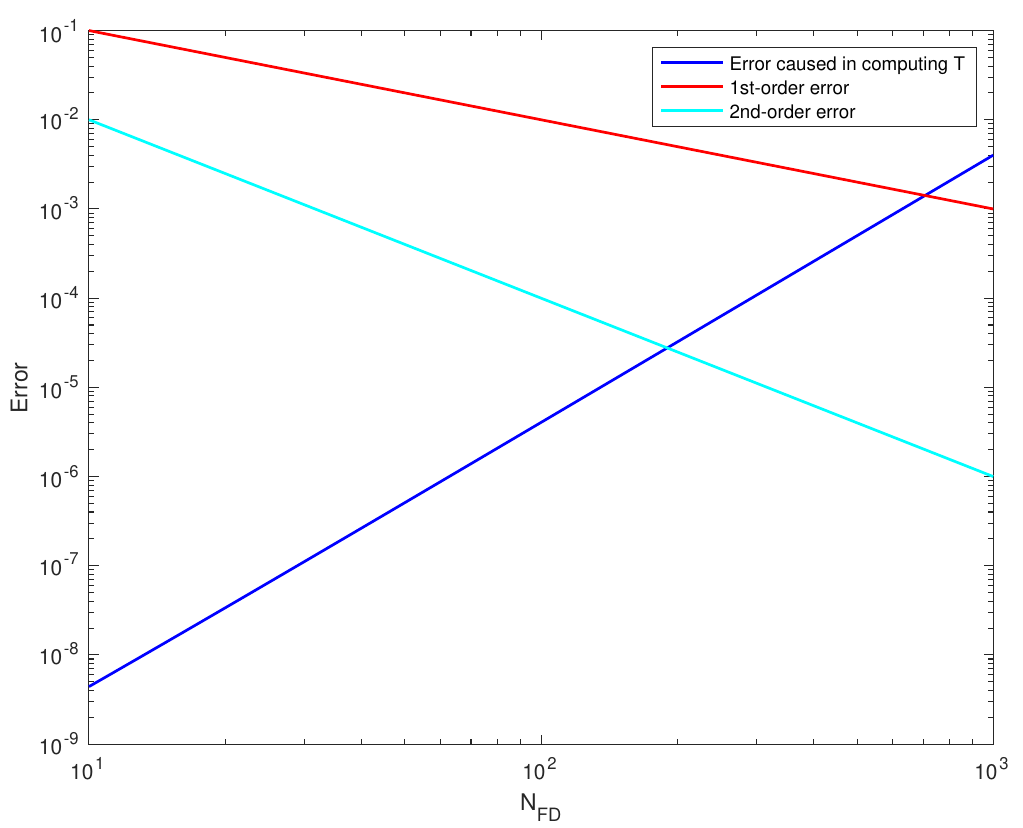}
}
\\
\subfigure[$\delta = 9$ ($10^{-9}$), in 3D]{
\includegraphics[width=0.33\linewidth]{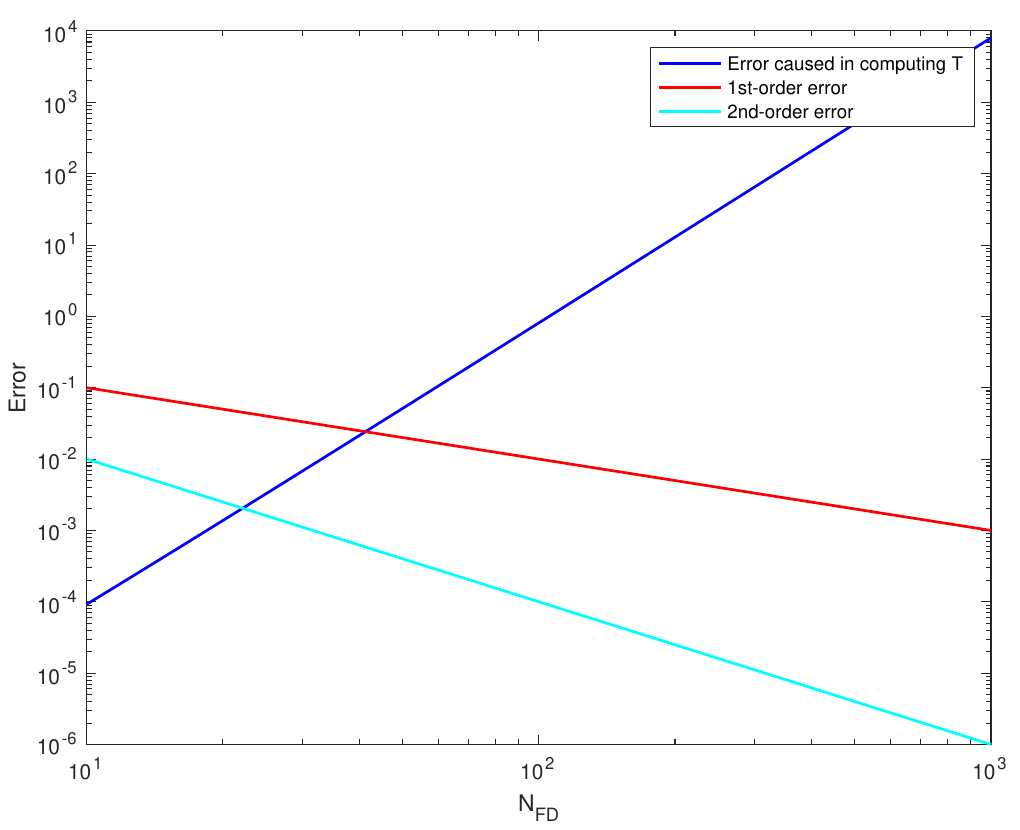}
}
\subfigure[$\delta = 12$ ($10^{-12}$), in 3D]{
\includegraphics[width=0.33\linewidth]{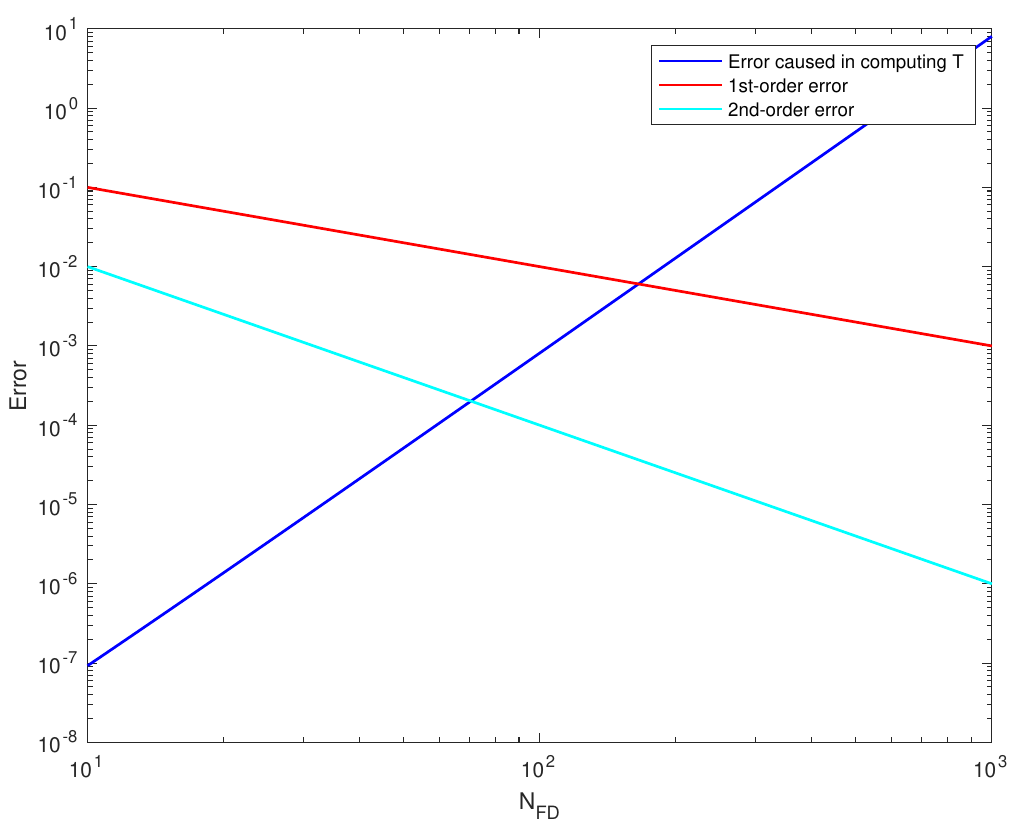}
}
\caption{The error in computing $T$ is plotted as a function of $N_{\text{FD}}$
against the discretization error.}
\label{fig:Err_in_T_1}
\end{figure}

\begin{figure}[ht!]
\centering
\subfigure[2D, $M=2^{10}$ and $2^{14}$]{
\includegraphics[width=0.33\linewidth]{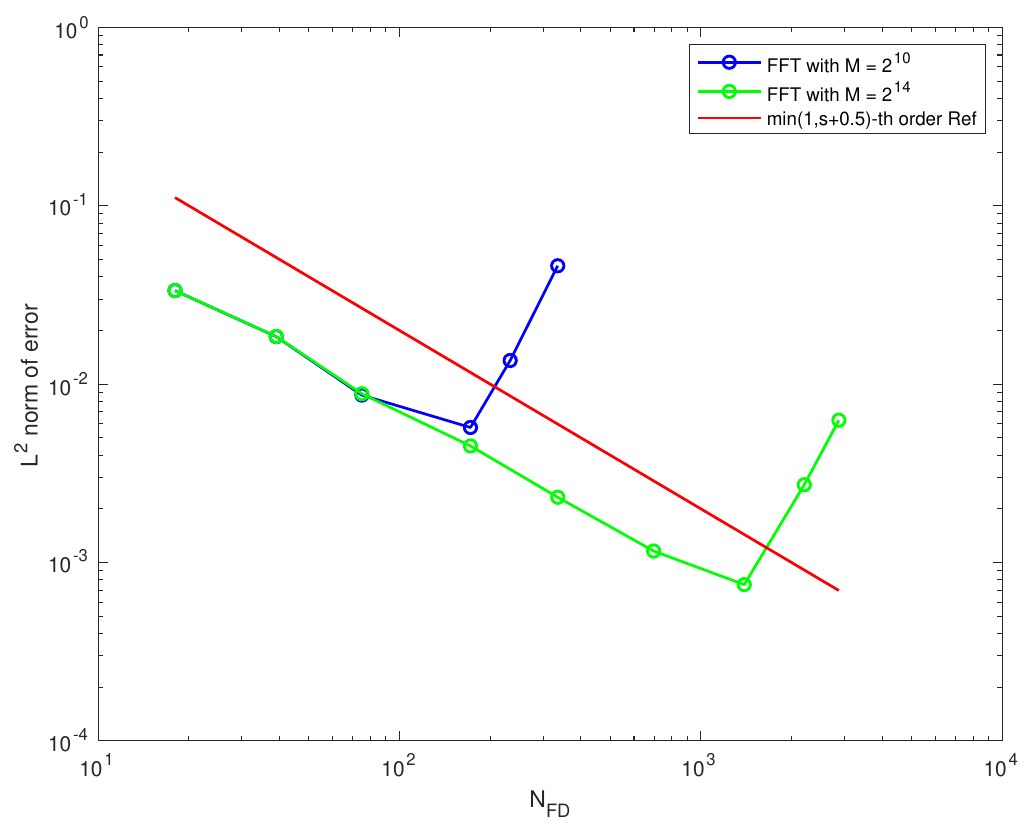}
}
\subfigure[3D, $M=2^{8}$ and $2^{10}$]{
\includegraphics[width=0.33\linewidth]{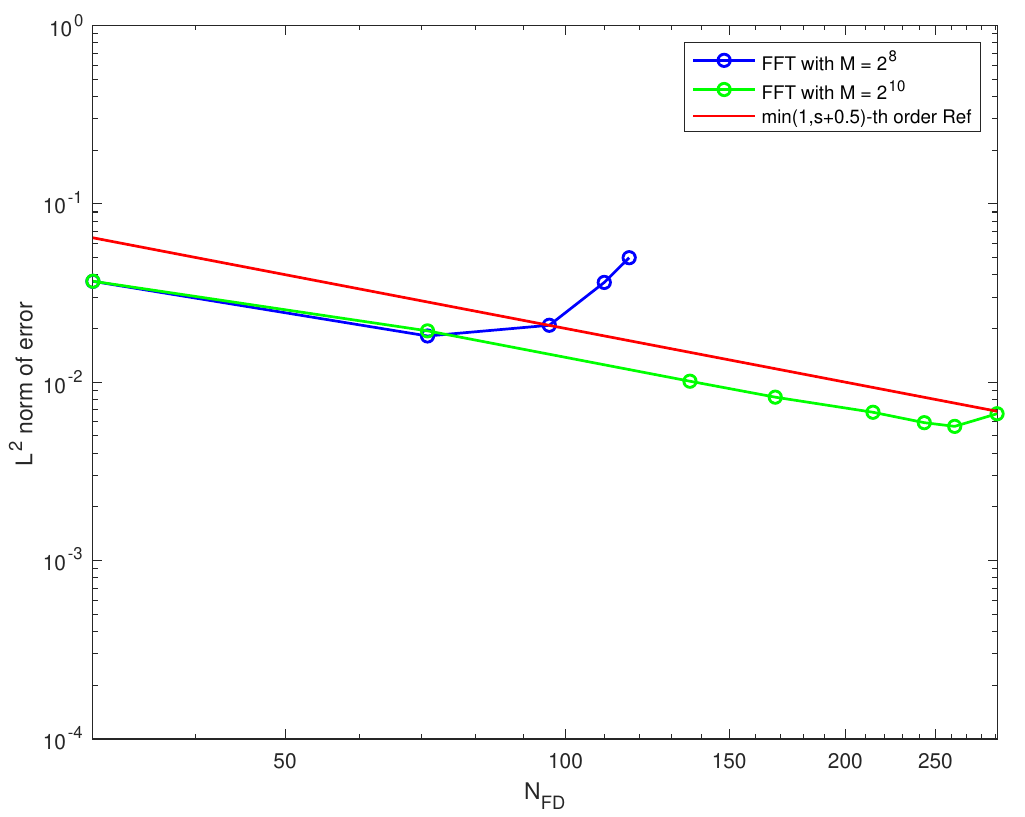}
}
\caption{Example~(\ref{main-example}) with $s = 0.5$. Convergence histories of GoFD
with the stiffness matrix $T$ being computed with FFT with various levels of accuracy (through different values of $M$).}
\label{fig:Err_GoFD_4_T}
\end{figure}

\section{Approximation of the stiffness matrix $T$}
\label{SEC:approximating-T}

In this section we describe four approaches to approximate the stiffness matrix $T$.
Recall that $T$ is given in (\ref{T-1}) in 2D. Its $d$-dimensional form is
\begin{align}
\label{T-2}
& T_{\vec{p}} = \frac{1}{(2\pi)^d} \iint_{(-\pi,\pi)^d}  \psi(\vec{\xi}) e^{i \vec{p} \cdot  \vec{\xi}}  d \vec{\xi},
\quad 0 \le \vec{p} \le 2N_{\text{FD}}
\end{align}
where
\begin{equation}
\label{psi-2}
\psi = \psi(\vec{\xi}) \equiv \left ( 4 \sum_{j=1}^d \sin^2\left (\frac{\xi_j}{2}\right ) \right )^s .
\end{equation}
In 1D, $T$ has the analytical form (see, e.g., \cite{Ortigueira2008}),
\begin{equation}
\label{T-1D}
T_p =  \frac{(-1)^p \Gamma(2 s+1)}{\Gamma(p + s + 1) \Gamma(s-p+1)}, \quad p = 0, ..., 2N_{\text{FD}} .
\end{equation}
In multi-dimensions, $T$ needs to be approximated numerically.

It is useful to recall some properties of the matrix $T$. Obviously, $T_{\vec{p}}$'s are the Fourier coefficients
of the function $\psi = \psi(\vec{\xi})$.
Moreover,  $T$ needs to be approximated only for $ 0 \le \vec{p} \le 2 N_{\text{FD}}$
(or in 2D, $ 0 \le p,q \le 2 N_{\text{FD}}$)
due to the symmetry (\ref{T-symmetry}).
Furthermore, $T_{\vec{p}}$ has the asymptotic decay
\begin{align}
\label{T-decay-0}
T_{\vec{p}} = \mathcal{O}(\frac{1}{|\vec{p}|^{d+2s}}), \quad \text{ as } |\vec{p}| \to \infty .
\end{align}
This is known analytically in 1D \cite{Ortigueira2008}. In multi-dimensions, such analytical results are not available.
Numerical verifications for $d = 1$, $2$, and $3$ can be seen in Fig.~\ref{fig:T-decay-1}.
From (\ref{AFD-1}), we can see that
the asymptotic decay of $T_{\vec{p}}$ for large $|\vec{p}|$  represents the decay of the entries of $A_{\text{FD}}$ away from the diagonal line.
Such asymptotic estimates are needed for the domain truncation and error estimation in the numerical solution of
inhomogeneous Dirichlet problems of the fractional Laplacian \cite{HS-2024-GoFD-inhom}.

\begin{figure}[ht!]
\centering
\subfigure[1D, $s=0.25$]{
\includegraphics[width=0.33\linewidth]{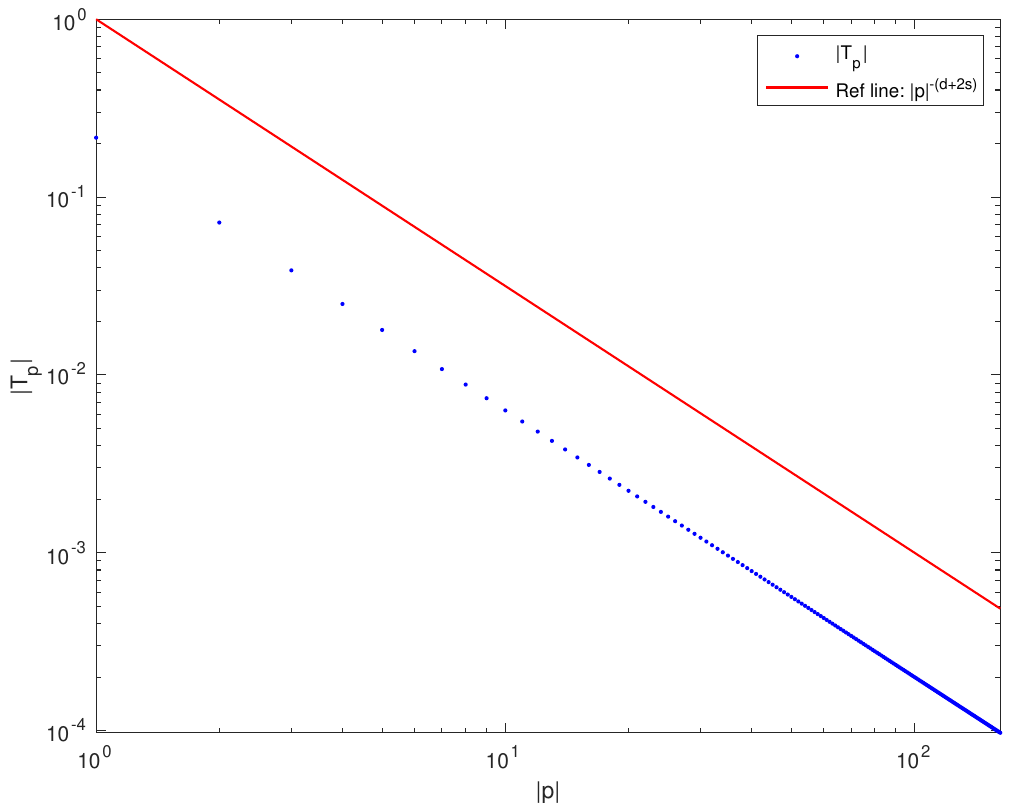}
}
\hspace{-15pt}
\subfigure[1D, $s=0.5$]{
\includegraphics[width=0.33\linewidth]{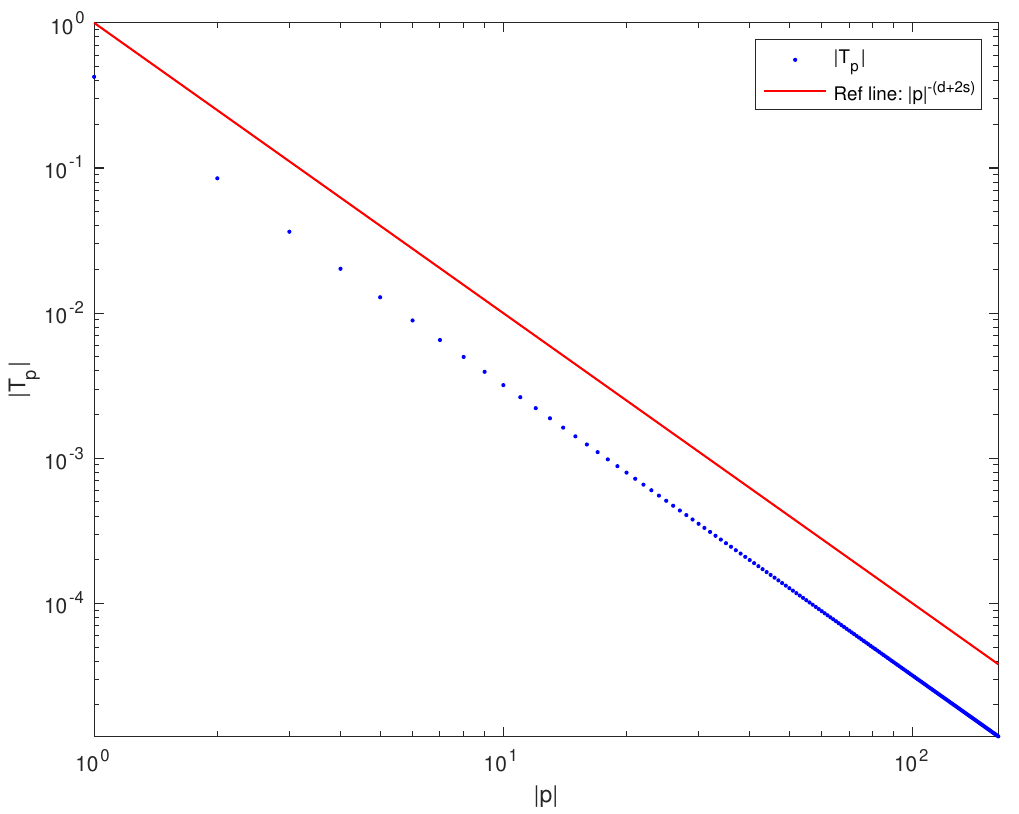}
}
\hspace{-15pt}
\subfigure[1D, $s=0.75$]{
\includegraphics[width=0.33\linewidth]{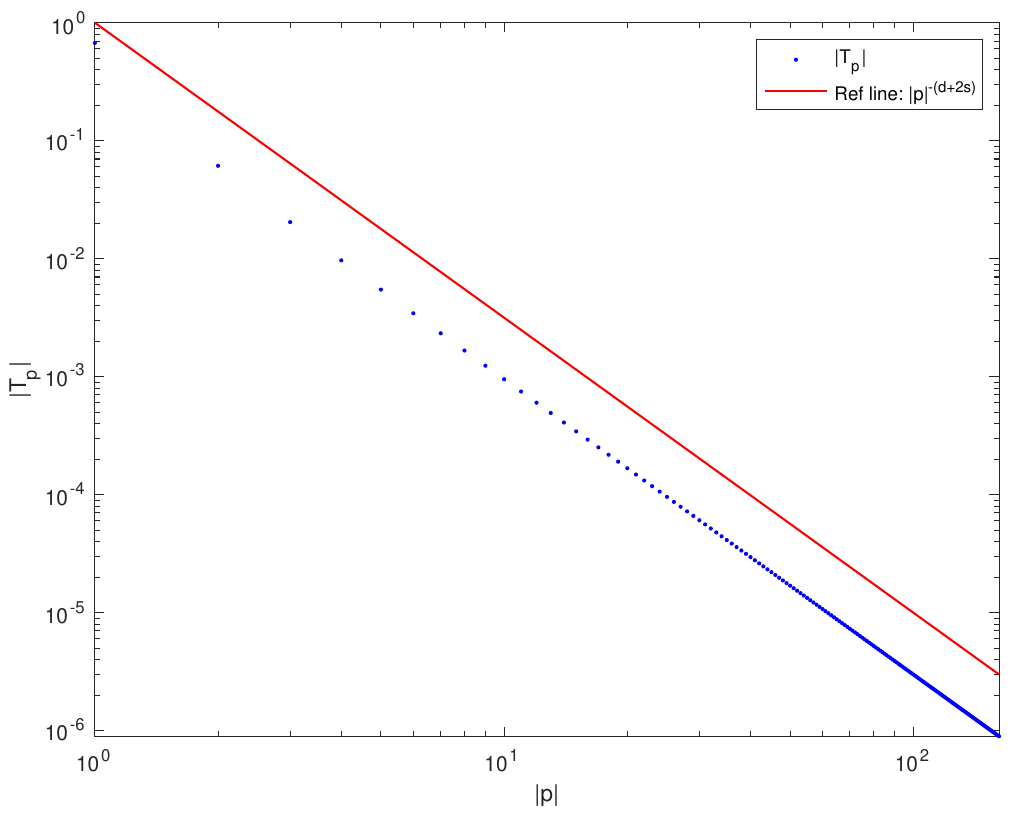}
}
\\
\subfigure[2D, $s=0.25$]{
\includegraphics[width=0.33\linewidth]{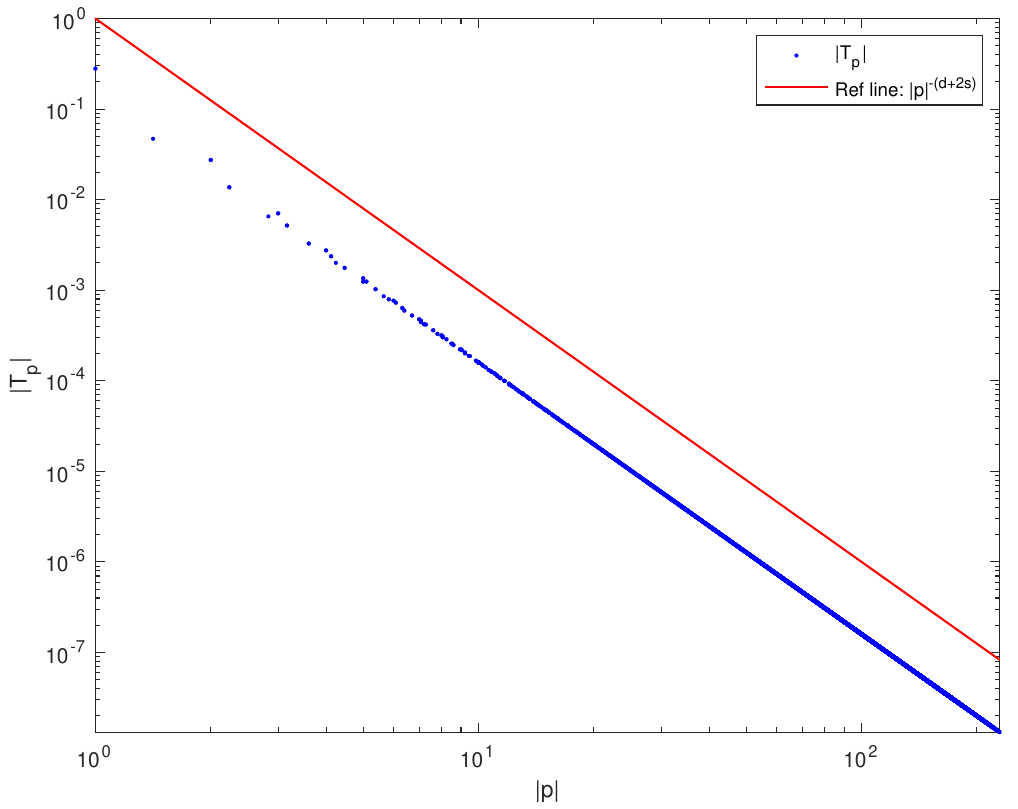}
}
\hspace{-15pt}
\subfigure[2D, $s=0.5$]{
\includegraphics[width=0.33\linewidth]{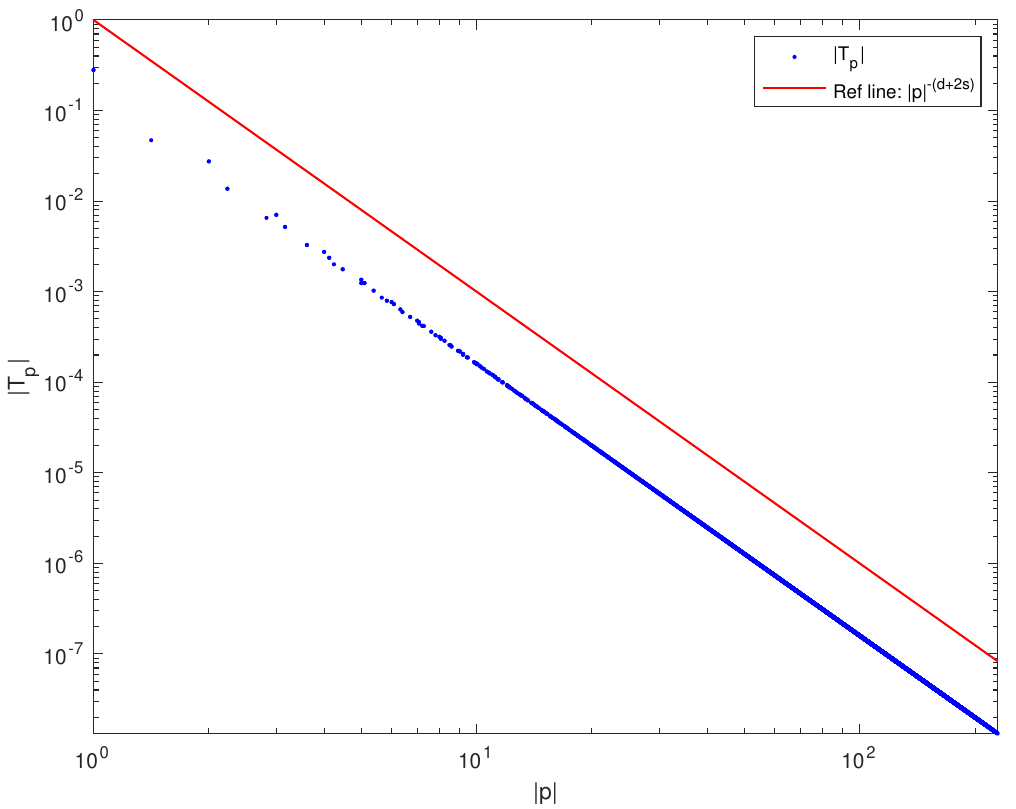}
}
\hspace{-15pt}
\subfigure[2D, $s=0.75$]{
\includegraphics[width=0.33\linewidth]{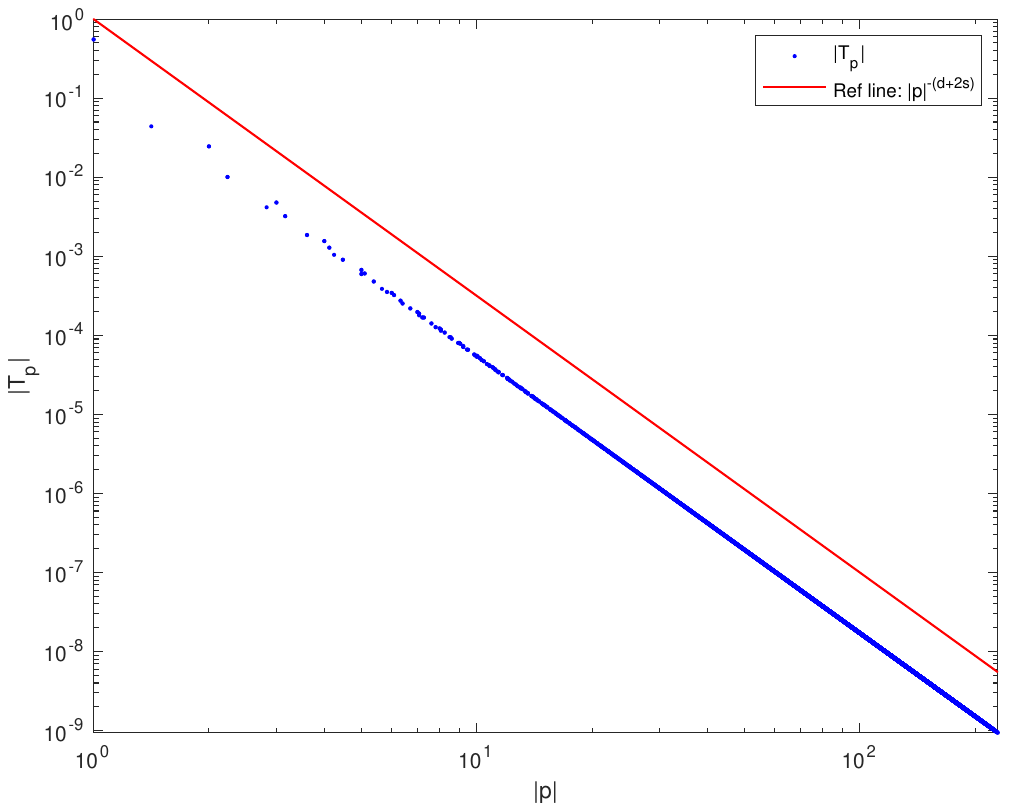}
}
\\
\subfigure[3D, $s=0.25$]{
\includegraphics[width=0.33\linewidth]{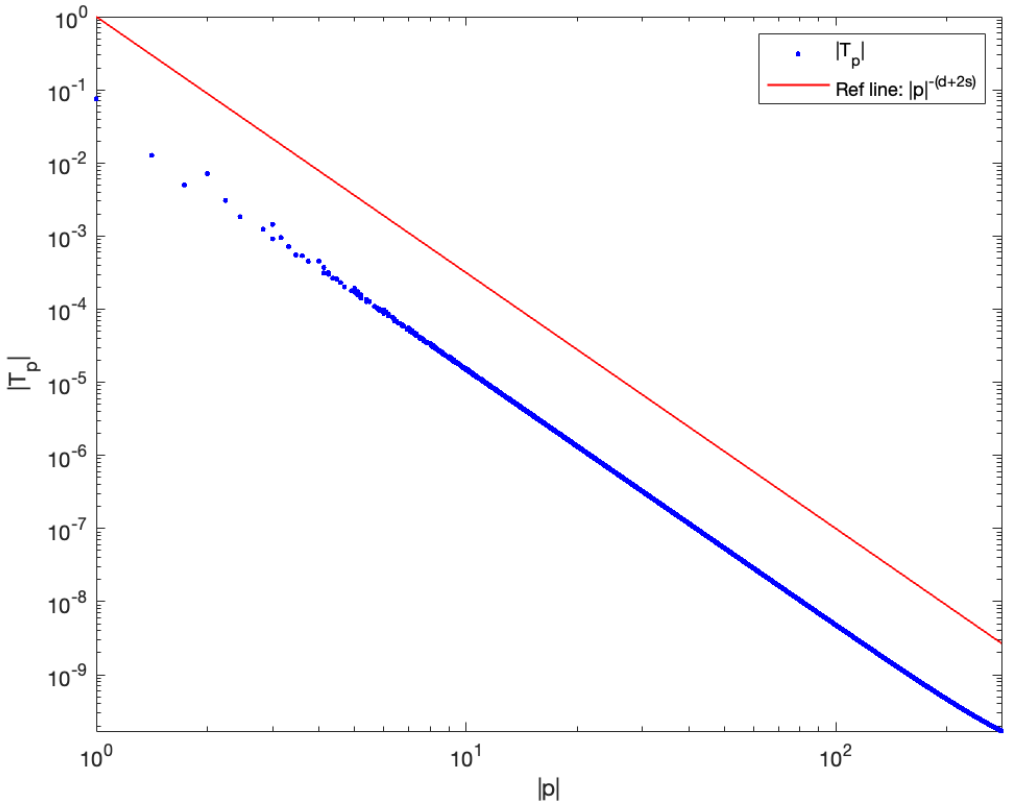}
}
\hspace{-15pt}
\subfigure[3D, $s=0.5$]{
\includegraphics[width=0.33\linewidth]{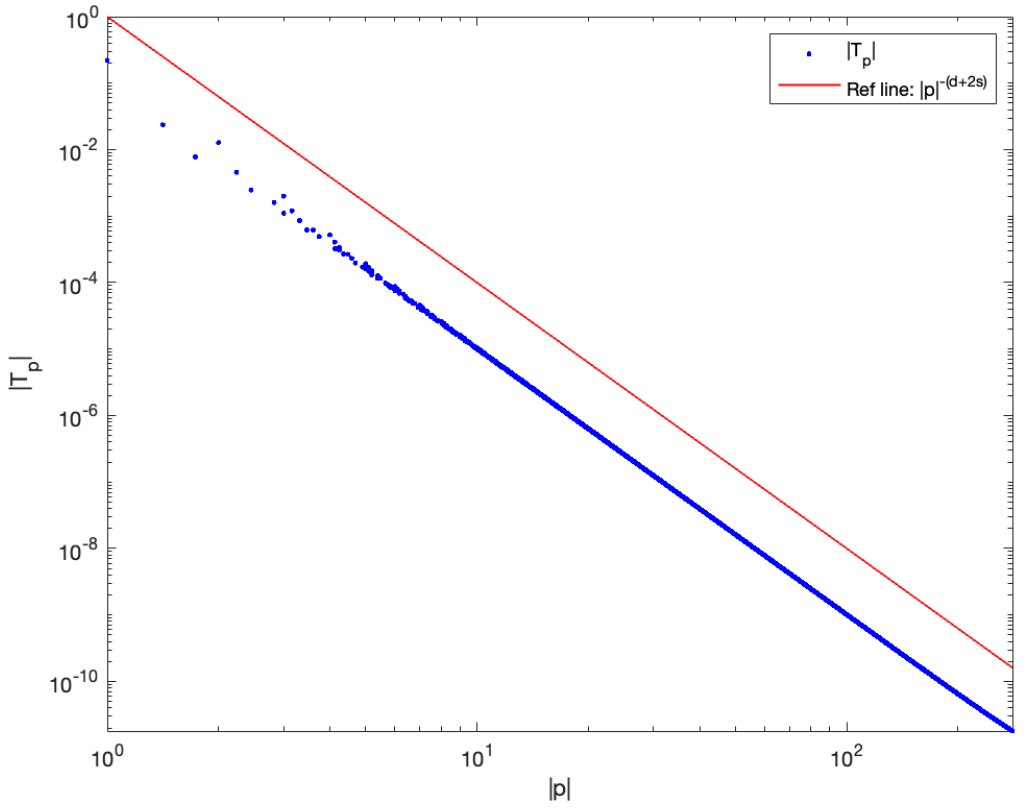}
}
\hspace{-15pt}
\subfigure[3D, $s=0.75$]{
\includegraphics[width=0.33\linewidth]{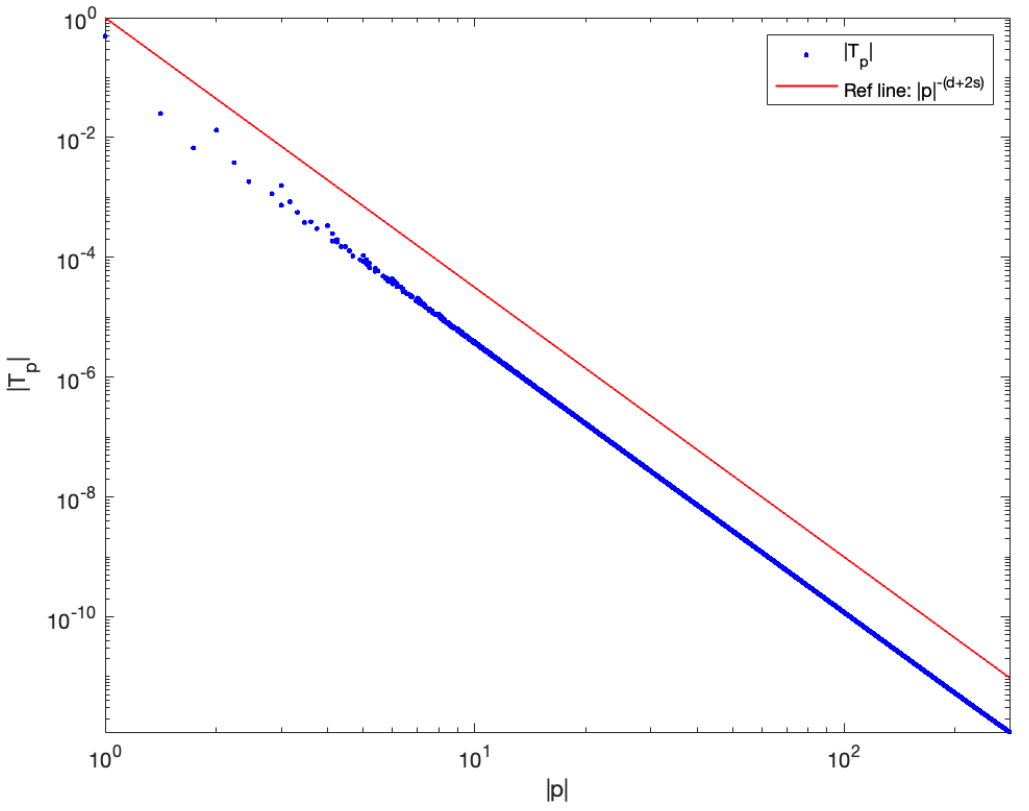}
}
\caption{The decay of $T_{\vec{p}}$ as $|\vec{p}| \to \infty$
for $s = 0.25$, 0.50, and 0.75. The reference line
is $T_{\vec{p}} = |\vec{p}|^{-(d+2 s)}$.}
\label{fig:T-decay-1}
\end{figure}

\subsection{The FFT approach}
\label{SEC:T-FFT}

We describe this approach in 2D. For any given integer $M \ge  2N_{\text{FD}}+1$, define a uniform grid
for $(-\pi,\pi) \times (-\pi,\pi)$ as
\begin{equation}
\label{uniform-grid-2}
(\xi_j,\eta_k) = \pi (\frac{2 j}{M}-1, \frac{2 k}{M}-1), \quad j, k = 0, 1, ..., M.
\end{equation}
Then, using the composite trapezoidal rule we have
\begin{align}
T_{p,q} & = \frac{1}{(2\pi)^2} \sum_{j=0}^{M-1} \sum_{k=0}^{M-1}
\int_{\xi_{j}}^{\xi_{j+1}}\int_{\eta_{k}}^{\eta_{k+1}}
 {\psi}(\xi,\eta) \cdot e^{i p \xi + i q\eta}d\xi d \eta
\notag \\
& \approx \frac{(-1)^{p+q}}{M^2} \sum_{j=0}^{M-1} \sum_{k=0}^{M-1}
  {\psi}(\xi_{j},\eta_{k}) \cdot e^{\frac{i  p j 2 \pi}{M}+\frac{iq k 2 \pi}{M}} ,
  \quad 0 \le p, q \le 2 N_{\text{FD}}
\label{T-FFT}
\end{align}
which can be computed with FFT. The number of flops required for this (in $d$-dimensions) is
\begin{equation}
\label{T-FFT-cost}
\mathcal{O}(M^d \log (M^d)) ,
\end{equation}
where $M \ge 2 N_{\text{FD}} + 1$.

To see how accurate this approach is, we apply it to the 1D case and compare the obtained approximation
with the analytical expression (\ref{T-1D}). The error is listed in Table~\ref{table:approximate-T-1}
for $M = 2^{10}$ and $2^{14}$.
It can be seen that this FFT approach performs much more accurate for large $s$ than small ones.
The difficulty with small $s$ comes from the fact that $\psi(\vec{\xi})$ has a cusp at the origin.

Filon's  \cite{Filon-1928} approach for highly oscillatory integrals and Richardson's extrapolation have
been considered for approximating $T$ in \cite{HS-2024-GoFD}. Some other numerical integral schemes,
such as composite Simpson's and Boole's rules, can also be used (to improve convergence order for smooth
integrants). Unfortunately, their convergence behavior for small $s$ is not much different from those
with the composite trapezoidal rule used here due to the low regularity of the function $\psi(\vec{\xi})$
at the origin for small $s$.

\begin{table}[htb]
\begin{center}
\caption{The error in approximating $T$ in one dimension using the FFT, non-uniform FFT, and
modified spectral approaches. The error is calculated as the maximum norm of the difference between
the approximation and analytical expression (\ref{T-1D}). $N_{\text{FD}} = 81$ is used.}
\begin{tabular}{|c|c|c|c|c|c|c|}\hline \hline
 & & \multicolumn{5}{|c|}{$s$} \\ \cline{3-7}
Approximation & $M$ & 0.1 & 0.25 & 0.5 & 0.75 & 0.9 \\ \hline
FFT & $2^{10}$ & 2.477e-04 & 3.247e-05 & 1.050e-06 & 2.609e-08 & 1.713e-09 \\  \cline{2-7}
        & $2^{14}$ & 8.812e-06 & 4.970e-07 & 3.902e-09 & 2.337e-11 & 6.516e-13\\ \hline
nuFFT & $2^{10}$ & 1.486e-06 & 1.798e-06 & 2.543e-06 & 3.597e-06 & 4.428e-06 \\ \cline{2-7}
            & $2^{14}$ & 5.735e-09 & 7.024e-09 & 9.934e-09 & 1.4055e-08 & 1.730e-08 \\ \hline
Modified Spectral & $2^{10}$ & 7.550e-07 & 2.962e-07 & 1.050e-06 & 2.792e-06 & 4.723e-06 \\ \cline{2-7}
                    & $2^{14}$ & 7.387e-08 & 2.457e-09 & 3.902e-09 & 1.037e-08 & 1.755e-08 \\
\hline \hline
\end{tabular}
\label{table:approximate-T-1}
\end{center}
\end{table}

\begin{table}[htb]
\begin{center}
\caption{CPU times (in seconds) taken to compute the matrix $T$ in 2D and 3D for various values of $M$.
$N_{\text{FD}} = 81$ and $s = 0.5$ are used except for the case with 3D and $M = 2^8$ where $N_{\text{FD}} = 71$
is used. For the modified spectral approximation, $n_G = 64$ is used. The computation was performed on a Macbook Pro
with M1 Max and 32 GB memory.}
\begin{tabular}{|c|c|c|c|c|c|c|}\hline \hline
 & \multicolumn{3}{|c|}{2D} & \multicolumn{3}{|c|}{3D}\\ \cline{2-7}
Approximation & $M = 2^{10}$ & $M = 2^{12}$ & $M = 2^{14}$ & $M = 2^8$ & $M = 2^{9}$ & $M = 2^{10}$\\ \hline
FFT & 0.033 & 0.082 & 3.879 & 0.103 & 1.895 & 38.09  \\ \hline
nuFFT & 0.0738 & 0.503 & 8.602 & 6.369 & 49.94 & 419.24 \\ \hline
Modified Spectral & 0.144 & 0.344 & 4.519 & 1.406 & 4.241 & 67.21 \\
\hline \hline
\end{tabular}
\label{table:approximate-T-2}
\end{center}
\end{table}

\subsection{The non-uniform FFT approach}
\label{SEC:T-nuFFT}

A strategy to deal with the low regularity of $\psi(\vec{\xi})$ at the origin is to use non-uniform sampling points.
For example, we use
\begin{equation}
\label{uniform-grid-3}
(\xi_j,\eta_k) = \pi \left ( \left (\frac{2 j}{M}-1\right )^2 \text{sign}\left (\frac{2 j}{M}-1\right ) ,
\left (\frac{2 k}{M}-1\right )^2 \text{sign}\left (\frac{2 k}{M}-1\right )\right ), \quad j, k = 0, 1, ..., M
\end{equation}
which clusters at the origin.
Then, we apply the composite trapezoidal rule and obtain
\begin{align}
T_{p,q} \approx \frac{1}{(2\pi)^2} \sum_{j=0}^{M-1} \sum_{k=0}^{M-1}
 \frac{(\xi_{j+1}-\xi_{j-1}) (\eta_{k+1}-\eta_{k-1})}{4}   {\psi}(\xi_{j},\eta_{k}) \cdot e^{i p \xi_j + i q \eta_k} ,
  \quad 0 \le p, q \le 2 N_{\text{FD}}
\label{T-nuFFT}
\end{align}
which can be computed using the non-uniform FFT \cite{Barnett-2021,Barnett-2019,Dutt-Rokhlin-1993}.
In our computation, we use \textbf{finufft} developed by Barnett et al. \cite{Barnett-2021,Barnett-2019}.
Matlab's function \textbf{nufft.m} can also be used
for this purpose but is slower than \textbf{finufft}.

This approach is applied to the 1D case and the error is listed in Table~\ref{table:approximate-T-1}.
It can be seen that this approach improves the accuracy significantly for small $s$ than the FFT approach
although it is inferior to the latter for large $s$. Moreover, the accuracy of this nuFFT approximation
is almost the same for all $s$ in (0, 1).

The CPU time is reported in Table~\ref{table:approximate-T-2}. It can be seen that, unfortunately,
the non-uniform FFT approximation is much more expensive than the FFT approximation.
This is especially true for large $M$ or for 3D.

\subsection{A spectral approximation}
\label{SEC:T-spectral}

This approach is the spectral approximation of
Zhou and Zhang \cite{Zhang2023} for the fractional Laplacian.
It can be interpreted as replacing the integrand and the domain in (\ref{T-2}) by spherically symmetric ones, i.e.,
\[
\psi(\vec{\xi}) \equiv \left ( 4 \sum_{j=1}^d \sin^2\left (\frac{\xi_j}{2}\right ) \right )^s \Longrightarrow |\vec{\xi}|^{2 s},\qquad
(-\pi,\pi)^d \Longrightarrow B_R(0),
\]
where $|\vec{\xi}|^{2 s}$ has the same behavior as $\psi(\vec{\xi})$ near the origin and
$R$ is chosen such that the ball $B_R(0)$ has the same volume as the cube $(-\pi,\pi)^d$, i.e.,
\begin{align}
\label{R-1}
R = 2 \sqrt{\pi} \left ( \Gamma\left (\frac{d}{2}+1\right ) \right )^{\frac{1}{d}} .
\end{align}
Then, we obtain
\begin{align}
\label{T-3}
\tilde{T}_{\vec{p}}  =  \frac{1}{(2\pi)^d} \iint_{B_R(0)}  |\vec{\xi}|^{2s} e^{i \vec{p}\cdot \vec{\xi}} d \vec{\xi},
\end{align}
The difference between (\ref{T-2}) and (\ref{T-3}) is
\begin{equation}
\label{T-3-1}
T_{\vec{p}}  - \tilde{T}_{\vec{p}}  = \frac{1}{(2\pi)^d}  \iint_{(-\pi,\pi)^d} \left ( \psi(\vec{\xi})-|\vec{\xi}|^{2 s} \right )
e^{i \vec{p} \cdot  \vec{\xi}}  d \vec{\xi}
+ \frac{1}{(2\pi)^d} \left (\iint_{(-\pi,\pi)^d\setminus B_R(0)} - \iint_{B_R(0)\setminus (-\pi,\pi)^d} \right )
|\vec{\xi}|^{2 s} e^{i \vec{p} \cdot  \vec{\xi}}  d \vec{\xi}.
\notag
\end{equation}
The first and second terms on the right-hand side correspond to the differences caused by replacing the integrand and
integration domain, respectively.
Using the formula of the Fourier transform of radial functions, we can rewrite (\ref{T-3}) into
\begin{align}
\label{T-4}
\tilde{T}_{\vec{p}}  & =  \frac{1}{(2\pi)^{\frac{d}{2}} |\vec{p}|^{\frac{d}{2}-1}} \int_0^R r^{2s + \frac{d}{2}} J_{\frac{d}{2}-1}(|\vec{p}| r) d r,
\end{align}
where $J_{\frac{d}{2}-1}(\cdot)$ is a Bessel function. From Prudnikov et al. \cite[2.12.4 (3)]{Prudnikov-1988} (with $\beta = 1$), we can rewrite this into
\begin{align}
\label{T-5}
\tilde{T}_{\vec{p}} = \frac{R^{d + 2s}}{(2 \pi )^{\frac{d}{2}} 2^{\frac{d}{2}} (s+\frac{d}{2}) \Gamma(\frac{d}{2})}
\mbox{}_1F_2(2s+\frac{d}{2}; 2s+\frac{d}{2}+1,\frac{d}{2}; - \frac{R^2 |\vec{p}|^2}{4}),
\end{align}
where $\mbox{}_1F_2(\cdot)$ is a hypergeometric function \cite{Andrews-1999}.
Hypergeometric functions are supported in many packages, including Matlab (\textbf{hypergeom.m}).
However, it can be extremely slow to compute $\tilde{T}_{\vec{p}}$ through (\ref{T-5}) using hypergeometric functions
especially when $N_{\text{FD}}$ is large.

Here, we propose a new and fast way to compute $\tilde{T}_{\vec{p}}$. When $\vec{p} = \vec{0}$, from (\ref{T-3}) we have
\begin{align}
\tilde{T}_{\vec{0}} & =  \frac{1}{(2\pi)^d} \iint_{B_R(0)}  |\vec{\xi}|^{2s} d \vec{\xi}
 = \frac{1}{(2\pi)^d} \int_0^R \frac{d \pi^{\frac{d}{2}} r^{d-1+2s}}{\Gamma(\frac{d}{2}+1)}  d r
= \frac{2 R^{d+2s}}{(d+2s) 2^d \pi^{\frac{d}{2}} \Gamma(\frac{d}{2})} .
\label{T-6}
\end{align}

When $\vec{p} \neq \vec{0}$, by changing the integral variable in (\ref{T-4}) we get
\begin{align}
\label{T-7}
\tilde{T}_{\vec{p}} = \frac{1}{(2 \pi)^{\frac{d}{2}} |\vec{p}|^{2s+d}} \int_0^{R |\vec{p}|} r^{2s+\frac{d}{2}} J_{\frac{d}{2}-1}(r) d r.
\end{align}
Recall that we need to compute $\tilde{T}_{\vec{p}}$ for $\vec{p} = p$ in 1D for $0\le p \le N_{\text{FD}}$,
$\vec{p} = (p,q)$ for $0 \le p, q \le N_{\text{FD}}$ in 2D, and $\vec{p} = (p,q,r)$ for $0 \le p, q, r \le N_{\text{FD}}$ in 3D.
Let $n = (N_{\text{FD}}+1)^d$. These points are sorted (and renamed if necessary) as
\[
|\vec{p}_0| \le |\vec{p}_1| \le \cdots \le |\vec{p}_{n}| .
\]
In actual computation, we can remove repeated points to save time.
Then, we apply a Gaussian quadrature (of $n_G$ points) to the integrals
\[
\int_{R|\vec{p}_{j-1}|}^{R|\vec{p}_{j}|} r^{2s+\frac{d}{2}} J_{\frac{d}{2}-1}(r) d r, \quad j = 1, ..., n.
\]
Finally, we obtain
\begin{align}
\label{T-8}
\tilde{T}_{\vec{p}_k} = \frac{1}{(2 \pi)^{\frac{d}{2}} |\vec{p}_k|^{2s+d}} \sum_{j=1}^{k}
\int_{R|\vec{p}_{j-1}|}^{R|\vec{p}_{j}|} r^{2s+\frac{d}{2}} J_{\frac{d}{2}-1}(r) d r, \quad k = 1, ..., n.
\end{align}
The total number of flops required to compute $\tilde{T}$ is
\begin{equation}
\label{T-spectral-cost}
\mathcal{O}(n_G\, (N_{\text{FD}}+1)^d),
\end{equation}
which is linearly proportional to the total number of nodes of the uniform grid $\mathcal{T}_{\text{FD}}$. This cost
is smaller than or comparable with that of the FFT approach, (\ref{T-FFT-cost}).
Since the cost of computing this spectral approximation is independent of $M$, we do not list the CPU time for this approach
in Table~\ref{table:approximate-T-2}. Nevertheless, we can get an idea from the CPU time for the modified spectral approach
(cf. Section~\ref{SEC:T-m-spectral}) which employs the current approach for computing one of its integrals.

It has been shown in \cite{Zhang2023} that the stiffness matrix (\ref{T-3}) can lead to accurate approximations to BVP (\ref{BVP-1})
(also cf. Section~\ref{SEC:numerics}).
It is worth emphasizing that (\ref{T-3}) is different from (and not close to) the stiffness matrix (\ref{T-2}). In the following, we show that
(\ref{T-3}) has asymptotic rates different  from (\ref{T-decay-0} for the stiffness matrix (\ref{T-2}).

\begin{pro}
\label{pro:T-1}
The stiffness matrix $\tilde{T}$ defined in (\ref{T-3}) has the decay
\begin{align}
\label{pro-T-1-1}
\tilde{T}_{\vec{p}} = \mathcal{O}(\frac{1}{|\vec{p}|^{\frac{d+1}{2}}}), \quad \text{ as } |\vec{p}| \to \infty .
\end{align}
\end{pro}

\begin{proof}
It is known that the Bessel function $J_{\frac{d}{2}-1}$ can be expressed as
\[
J_{\frac{d}{2}-1}(r) = \sqrt{\frac{2}{\pi r}} \cos(r - \frac{(d-1)\pi}{4}) + \mathcal{O}(r^{-\frac{3}{2}}) .
\]
Thus, as $|\vec{p}| \to \infty$ we have
\begin{align*}
& \int_0^{R |\vec{p}|} r^{2s+\frac{d}{2}} J_{\frac{d}{2}-1}(r) d r
\\
& =  \int_0^1 r^{2s+\frac{d}{2}} J_{\frac{d}{2}-1}(r) d r
+ \sqrt{\frac{2}{\pi}} \int_1^{R |\vec{p}|} r^{2s+\frac{d-1}{2}} \left [ \cos(r - \frac{(d-1)\pi}{4}) + \mathcal{O}(r^{-1})\right ] d r
\\
& = \int_0^1 r^{2s+\frac{d}{2}} J_{\frac{d}{2}-1}(r) d r + \sqrt{\frac{2}{\pi}} \int_1^{R |\vec{p}|} r^{2s+\frac{d-1}{2}} \cos(r - \frac{(d-1)\pi}{4}) d r
+ \mathcal{O}(|\vec{p}|^{2 s + \frac{d-1}{2}})
\\
& = \int_0^1 r^{2s+\frac{d}{2}} J_{\frac{d}{2}-1}(r) d r  + \mathcal{O}(|\vec{p}|^{2 s + \frac{d-1}{2}})
+ \sqrt{\frac{2}{\pi}} r^{2s+\frac{d-1}{2}} \sin(r - \frac{(d-1)\pi}{4})\left. \frac{}{}\right |_{1}^{R |\vec{p}|}
\\
& \quad
- (2s + \frac{d-1}{2}) \sqrt{\frac{2}{\pi}} \int_1^{R |\vec{p}|} r^{2s+\frac{d-3}{2}} \sin(r - \frac{(d-1)\pi}{4}) d r
\\
& = \int_0^1 r^{2s+\frac{d}{2}} J_{\frac{d}{2}-1}(r) d r + \mathcal{O}(|\vec{p}|^{2 s + \frac{d-1}{2}})  + \mathcal{O}(|\vec{p}|^{2 s + \frac{d-1}{2}})
+ \mathcal{O}(|\vec{p}|^{2 s + \frac{d-1}{2}})
\\
& = \mathcal{O}(|\vec{p}|^{2 s + \frac{d-1}{2}}) .
\end{align*}
The asymptotic (\ref{pro-T-1-1}) follows this and (\ref{T-7}).
\end{proof}

\begin{pro}
\label{pro:T-2}
In 1D, the stiffness matrix $\tilde{T}$ defined in (\ref{T-3}) has the decay
\begin{align}
\label{pro-T-2-1}
\tilde{T}_{p} = \mathcal{O}(\frac{1}{|{p}|^{1+\min(1,2s)}}), \quad \text{ as } |{p}| \to \infty .
\end{align}
\end{pro}

\begin{proof}
From (\ref{T-3}), we have
\[
\tilde{T}_p = \frac{1}{2\pi} \int_{-\pi}^{\pi} |\xi|^{2 s} e^{i p \xi} d \xi = \frac{1}{\pi} \int_0^{\pi} \xi^{2 s} \cos(p \xi) d \xi
= \pi^{2 s} \int_0^1 \xi^{2 s} \cos( p \xi \pi) d \xi .
\]
From Gradshteyn and Ryzhik \cite[Page 436, 3.761 (6)]{Gradshteyn-2015}, we have
\[
\tilde{T}_p = \frac{\pi^{2s}}{2 (1+2s)} \big ( M(1+2s, 2+2s, i p \pi) + M(1+2s, 2+2s, - i p \pi) \big ),
\]
where $M(a,b,z)$ is the Kummer's confluent hypergeometric function.


As $p \to \infty$, the Kummer's function has the asymptotic
\begin{align*}
& M(1+2s, 2+2s,i p \pi)
\\
& = \frac{e^{i p \pi} (i p \pi)^{-1}}{\Gamma(1+2 s)} \left [ 1 + \frac{(-2s) (1)}{1!} (i p \pi)^{-1} + \mathcal{O}(p^{-2}) \right ]
+ \frac{e^{\pi i (1+2s)} (i p \pi)^{-(1+2s)}}{\Gamma(1)} \left [ 1 + \mathcal{O}(p^{-1})\right ]
\\
& = \frac{e^{i p \pi} (i p \pi)^{-1}}{\Gamma(1+2 s)} + \frac{ 2 s e^{i p \pi}}{(p \pi)^2 \Gamma(1+2 s)}
+ \frac{e^{\pi i (1+2s)} (i p \pi)^{-(1+2s)}}{\Gamma(1)}  + \mathcal{O}(p^{-(2+2s)}) .
\end{align*}
Noticing $i = e^{i \pi/2}$, we get
\[
M(1+2s, 2+2s,i p \pi) \sim
\frac{e^{i (p \pi-\frac{\pi}{2})} }{p \pi \Gamma(1+2 s)} + \frac{ 2 s e^{i p \pi}}{(p \pi)^2 \Gamma(1+2 s)}
+ \frac{e^{ i (1+2s) \frac{\pi}{2}}}{(p \pi)^{1+2s} } + \mathcal{O}(p^{-(2+2s)})
\]
and
\begin{align*}
& M(1+2s, 2+2s,i p \pi) + M(1+2s, 2+2s,-i p \pi)
\notag \\
& =  \frac{2 \cos(x\pi-\frac{\pi}{2}) }{p \pi \Gamma(1+2 s)} + \frac{ 4 s \cos(p \pi)}{(p \pi)^2 \Gamma(1+2 s)}
+ \frac{2 \cos ((1+2s)\pi/2) }{(p \pi)^{1+2s} }  + \mathcal{O}(p^{-(2+2s)}).
\end{align*}
Combining the above results, we get
\begin{align}
\label{1D-T-1}
\tilde{T}_p  = \frac{\pi^{2s}}{2 (1+2s)} \left [
\frac{2 \sin(p \pi) }{p \pi \Gamma(1+2 s)} + \frac{ 4 s \cos(p \pi)}{(p \pi)^2 \Gamma(1+2 s)}
+ \frac{2 \cos ((1+2s)\pi/2) }{(p \pi)^{1+2s} } \right ]  + \mathcal{O}(p^{-(2+2s)}).
\end{align}
Recall that $p$ is a positive integer and $\sin(p \pi) = 0$. Thus,
\begin{align*}
\tilde{T}_p & = \frac{\pi^{2s}}{2 (1+2s)} \left [
 \frac{ 4 s \cos(p \pi)}{(p \pi)^2 \Gamma(1+2 s)}
+ \frac{2 \cos ((1+2s)\pi/2) }{(p \pi)^{1+2s} } \right ]  + \mathcal{O}(p^{-(2+2s)})
\\
& =  \mathcal{O}(p^{-2}) +  \mathcal{O}(p^{-(1+2s)})
= \mathcal{O}(p^{-(1+\min(1,2s))}) .
\end{align*}
\end{proof}

Numerical results in Fig.~\ref{fig:T-decay-2} verify the decay rates of $\tilde{T}_{\vec{p}}$
as in Propositions~\ref{pro:T-1} and \ref{pro:T-2}.
Moreover, Fig.~\ref{fig:T-decay-2} shows different patterns of the distribution of $\tilde{T}$
than those in Fig.~\ref{fig:T-decay-1} for the stiffness matrix $T$.
Specifically, the entries of $T$ stay closely along the decay line whereas
the entries of $\tilde{T}$ spread out more below the decay line.

\begin{figure}[ht!]
\centering
\subfigure[1D, $s=0.25$]{
\includegraphics[width=0.33\linewidth]{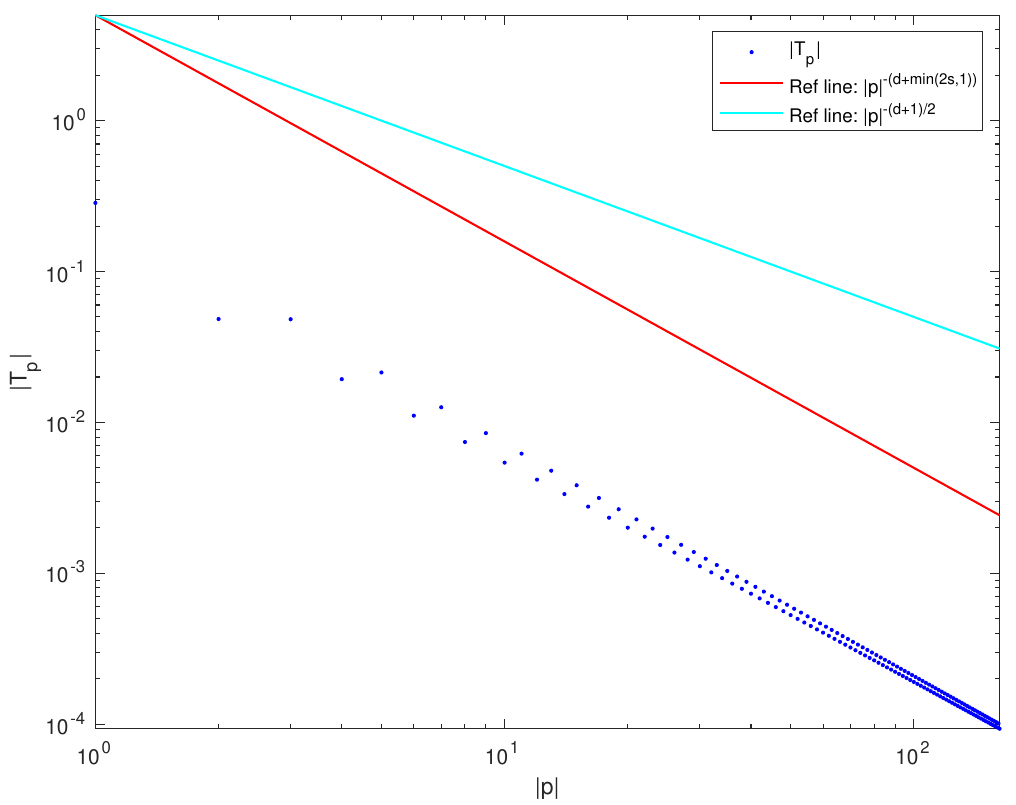}
}
\hspace{-15pt}
\subfigure[1D, $s=0.5$]{
\includegraphics[width=0.33\linewidth]{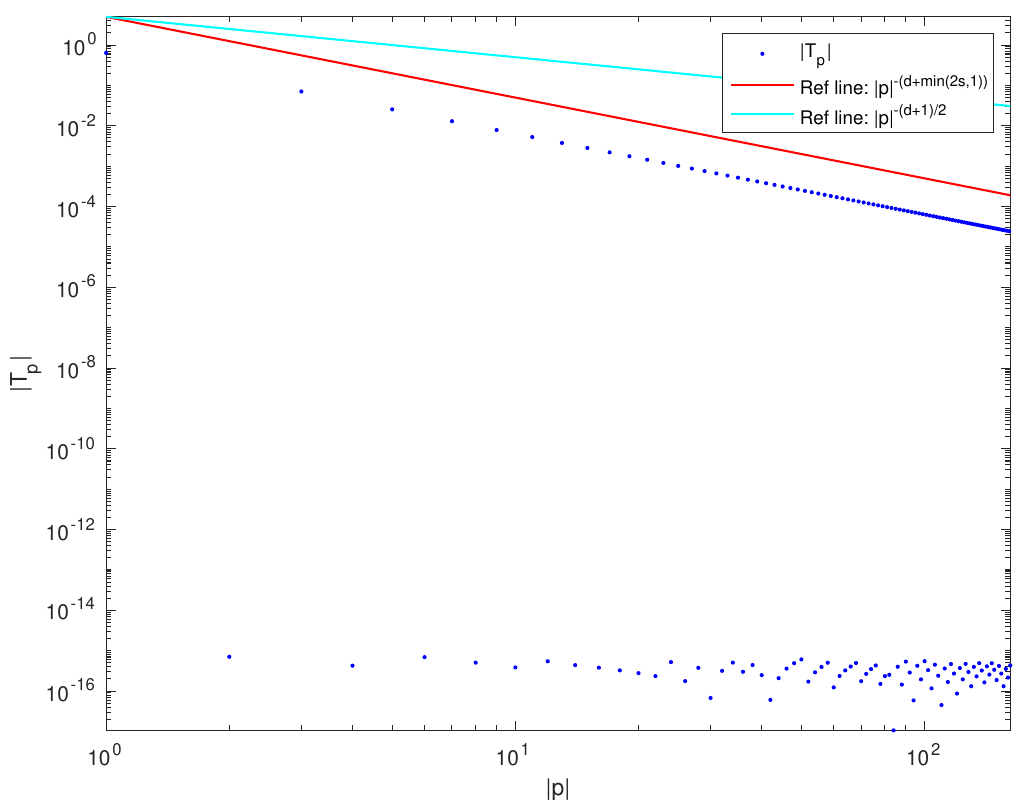}
}
\hspace{-15pt}
\subfigure[1D, $s=0.75$]{
\includegraphics[width=0.33\linewidth]{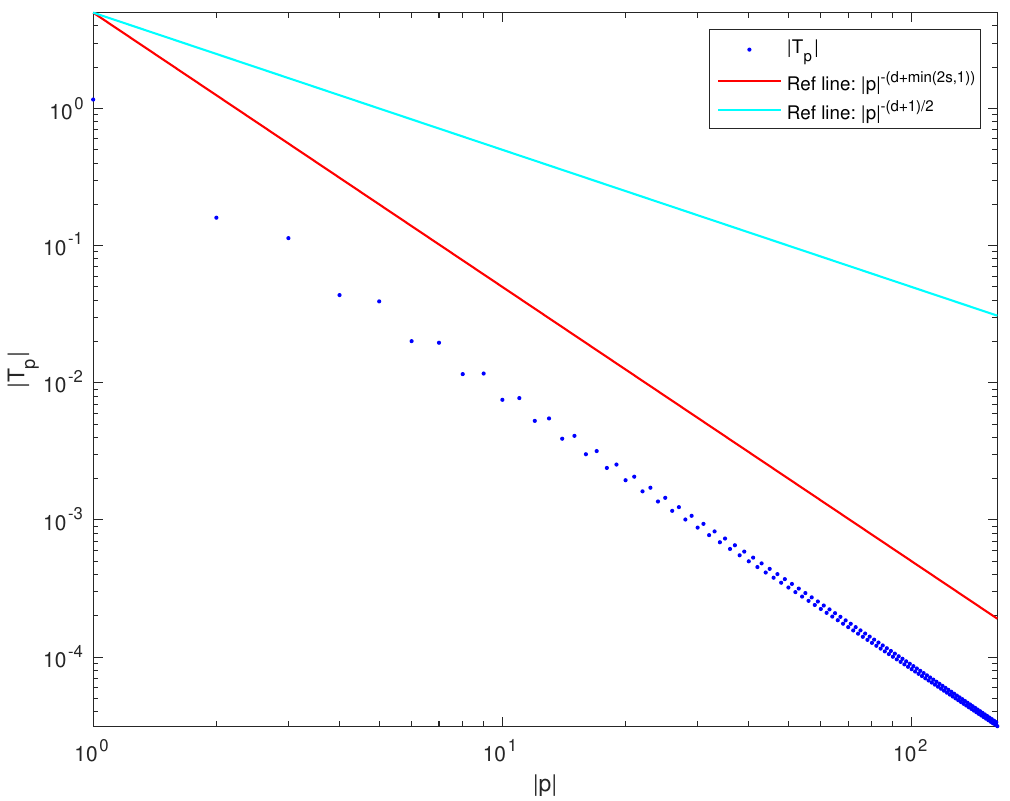}
}
\\
\subfigure[2D, $s=0.25$]{
\includegraphics[width=0.33\linewidth]{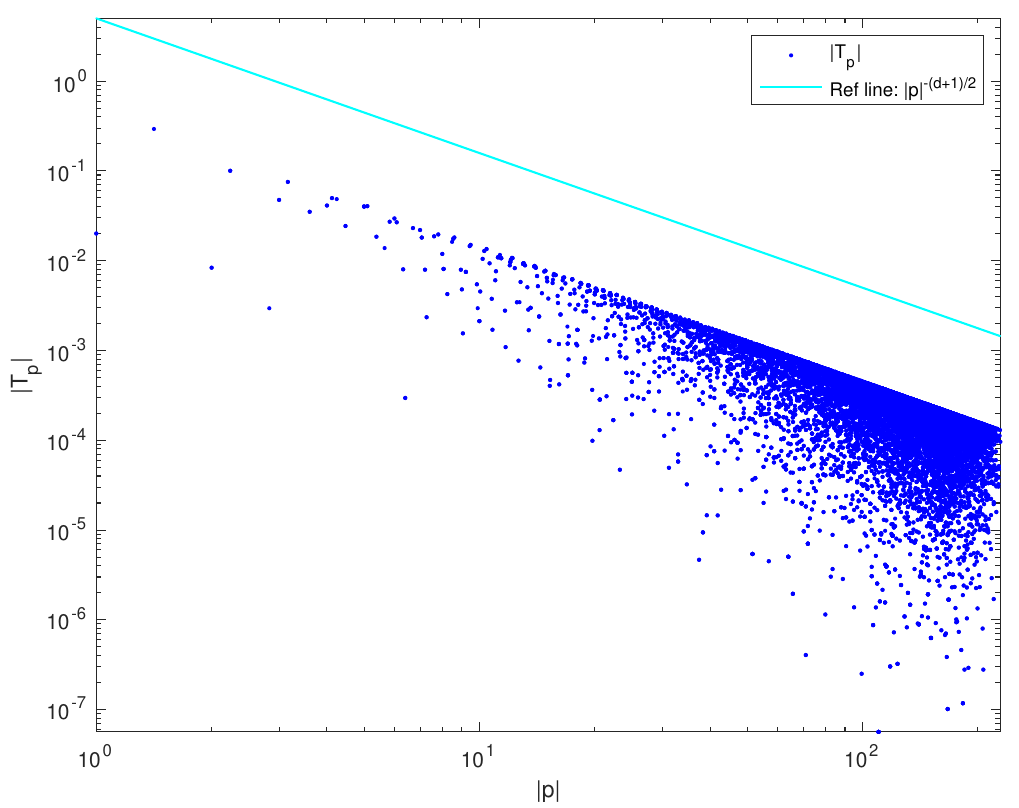}
}
\hspace{-15pt}
\subfigure[2D, $s=0.5$]{
\includegraphics[width=0.33\linewidth]{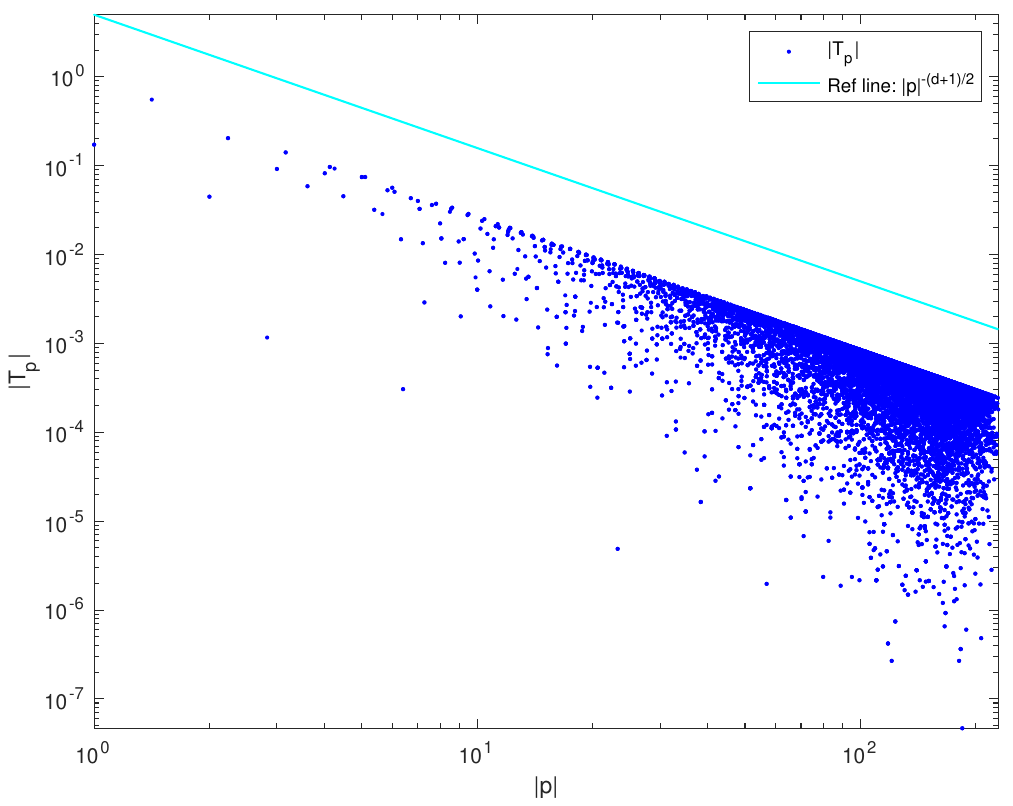}
}
\hspace{-15pt}
\subfigure[2D, $s=0.75$]{
\includegraphics[width=0.33\linewidth]{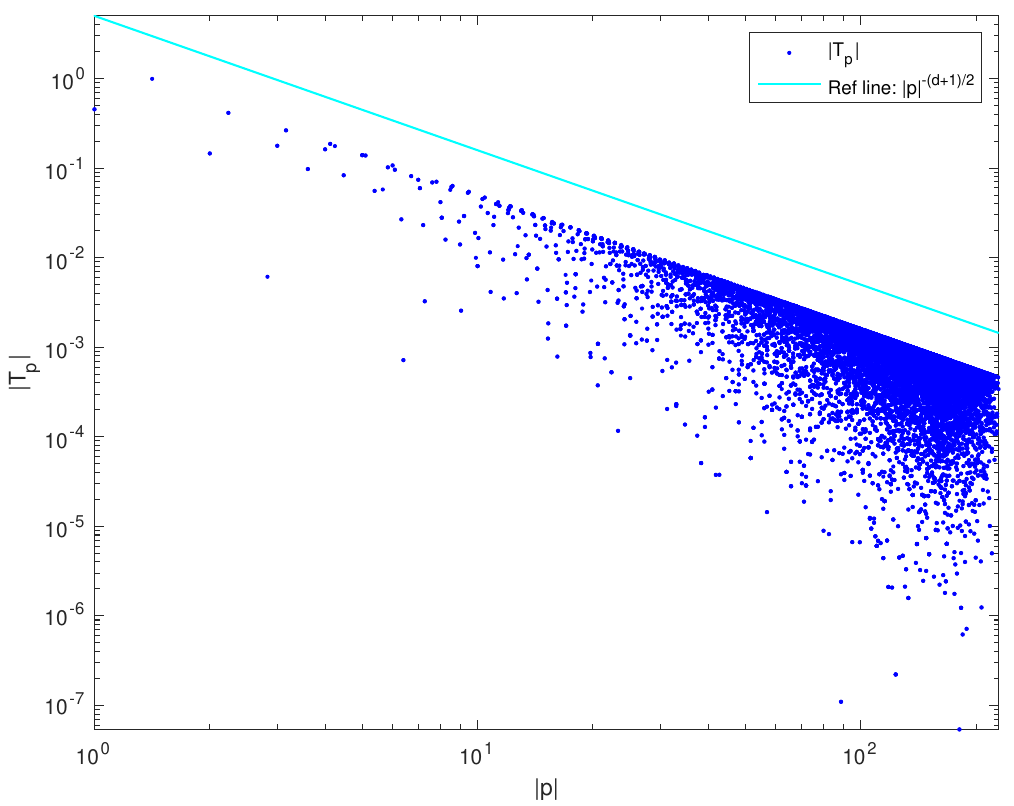}
}
\\
\subfigure[3D, $s=0.25$]{
\includegraphics[width=0.33\linewidth]{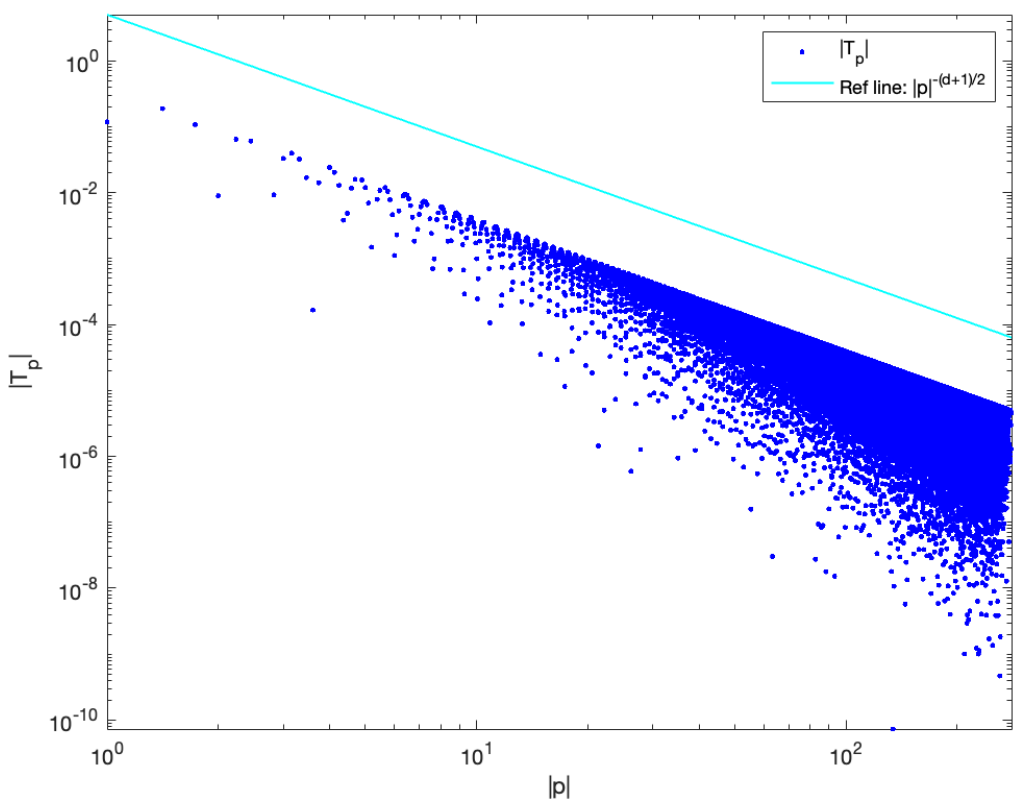}
}
\hspace{-15pt}
\subfigure[3D, $s=0.5$]{
\includegraphics[width=0.33\linewidth]{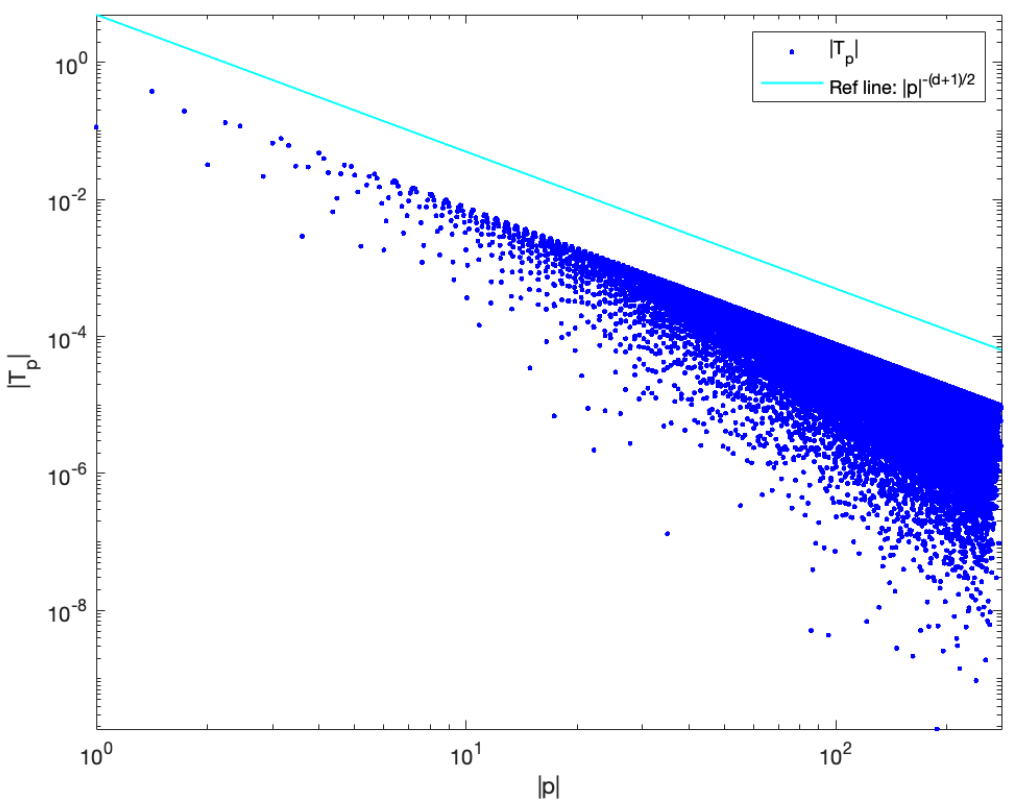}
}
\hspace{-15pt}
\subfigure[3D, $s=0.75$]{
\includegraphics[width=0.33\linewidth]{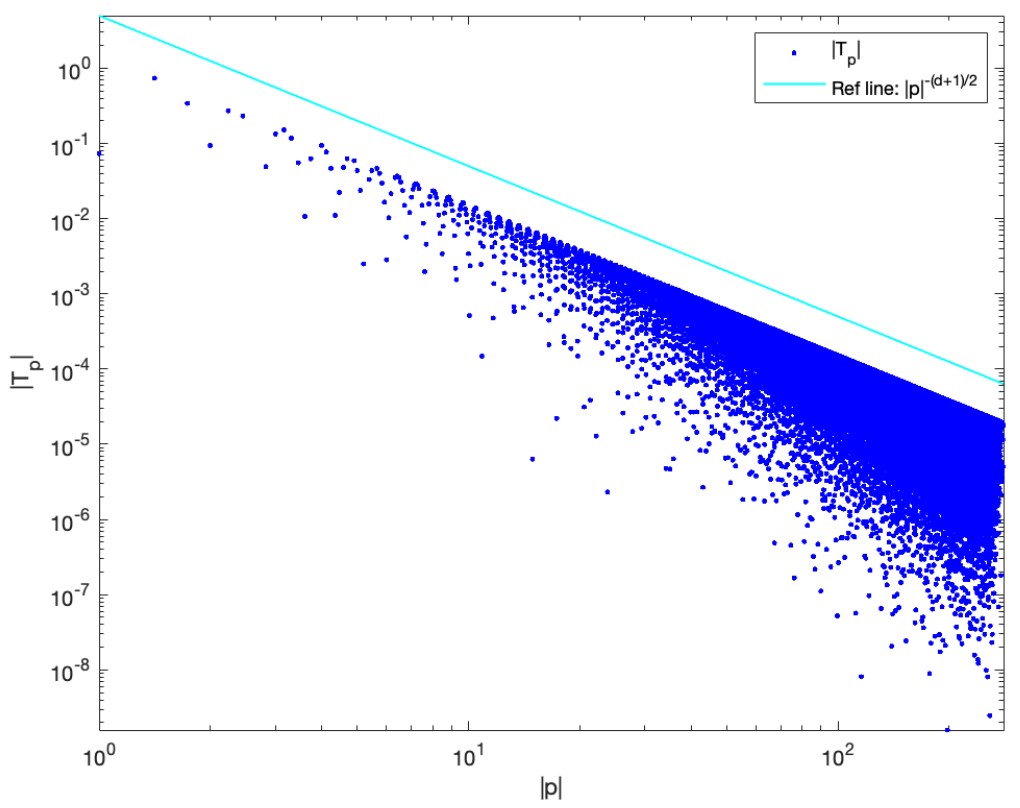}
}
\caption{The decay of $\tilde{T}_{\vec{p}}$ as $|\vec{p}| \to \infty$
for $s = 0.25$, 0.50, and 0.75. The reference lines
are  $T_{\vec{p}} = |\vec{p}|^{-(d+2 s)}$ (for 1D) and $T_{\vec{p}} = |\vec{p}|^{-(d+1)/2}$ (for 2D and 3D).}
\label{fig:T-decay-2}
\end{figure}

\subsection{A modified spectral approximation}
\label{SEC:T-m-spectral}

This approximation is a modification of the spectral approximation discussed in the previous subsection.
We rewrite (\ref{T-2}) into
\begin{align}
\label{T-9}
& T_{\vec{p}} = \frac{1}{(2\pi)^d} \iint_{(-\pi,\pi)^d}  \left ( \psi(\vec{\xi})-|\vec{\xi}|^{2 s} \right ) e^{i \vec{p} \cdot  \vec{\xi}}  d \vec{\xi}
+ \frac{1}{(2\pi)^d} \iint_{(-\pi,\pi)^d}  |\vec{\xi}|^{2 s} e^{i \vec{p} \cdot  \vec{\xi}}  d \vec{\xi}.
\end{align}
We expect that the low regularity issue can be avoided in the first integral on the right-hand side since $\psi(\vec{\xi})$ and
$|\vec{\xi}|^{2 s}$ have similar behavior near the origin. Moreover, we would like to use the algorithm (for spherically symmetric
integrals) in the previous subsection for the second integral. Thus,
we approximate the second integral by replacing the cube by a ball with the radius given in (\ref{R-1}). We get
\begin{align}
\label{T-10}
& \tilde{\tilde{T}}_{\vec{p}} = \frac{1}{(2\pi)^d} \iint_{(-\pi,\pi)^d}  \left ( \psi(\vec{\xi})-|\vec{\xi}|^{2 s} \right ) e^{i \vec{p} \cdot  \vec{\xi}}  d \vec{\xi}
+ \frac{1}{(2\pi)^d} \iint_{B_R(0)}  |\vec{\xi}|^{2 s} e^{i \vec{p} \cdot  \vec{\xi}}  d \vec{\xi}.
\end{align}
Notice that the second term on the right-hand side of the above equation
is actually the spectral approximation (\ref{T-3}) discussed in the previous subsection.
Thus, the current and previous approximations differ in the first term.
We apply the FFT approach of Section~\ref{SEC:T-FFT} to the first integral and the spectral approach of Section~\ref{SEC:T-spectral}
to the second integral on the right-hand side.
The cost of the current approach is essentially the addition of those for the FFT and spectral approaches.
The difference between (\ref{T-2}) and (\ref{T-10}) is
\begin{equation}
\label{T-10-1}
T_{\vec{p}}  - \tilde{\tilde{T}}_{\vec{p}}  =
\frac{1}{(2\pi)^d} \left (\iint_{(-\pi,\pi)^d\setminus B_R(0)} - \iint_{B_R(0)\setminus (-\pi,\pi)^d} \right )
|\vec{\xi}|^{2 s} e^{i \vec{p} \cdot  \vec{\xi}}  d \vec{\xi}.
\end{equation}
In 1D, $B_R(0) = (-\pi, \pi)$. Thus, the right-hand side of the above equation is zero and
$\tilde{\tilde{T}}_{\vec{p}}$ is the same as $T_{\vec{p}}$ in 1D.
As a result, we can compare this modified spectral approximation with the analytical expression (\ref{T-1D})
to see how well the low regularity of the integrand at the origin is treated.
From Tables~\ref{table:approximate-T-1} and \ref{table:approximate-T-2},
we can see that the current approach produces a level of accuracy comparable with the non-uniform FFT approximation
but is much faster than the latter.

It should be pointed out that $\tilde{\tilde{T}}_{\vec{p}}$ is different from $T_{\vec{p}}$ and $\tilde{T}_{\vec{p}}$ in multi-dimensions.
The distributions of its entries are plotted in Fig.~\ref{fig:T-decay-3}. One can see that they are slightly different from
those in Fig.~\ref{fig:T-decay-2} for $\tilde{T}_{\vec{p}}$ but have the same asymptotic decay rate $|\vec{p}|^{-(d+1)/2}$ in 2D and 3D.

\begin{figure}[ht!]
\centering
\subfigure[2D, $s=0.25$]{
\includegraphics[width=0.33\linewidth]{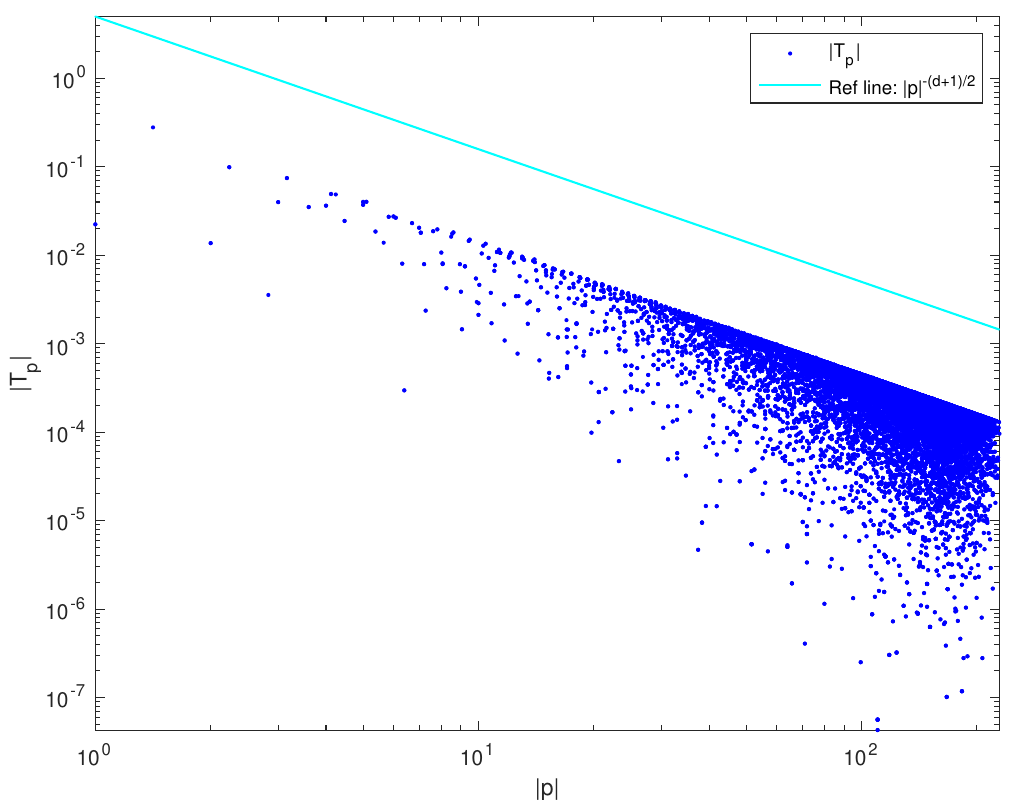}
}
\hspace{-15pt}
\subfigure[2D, $s=0.5$]{
\includegraphics[width=0.33\linewidth]{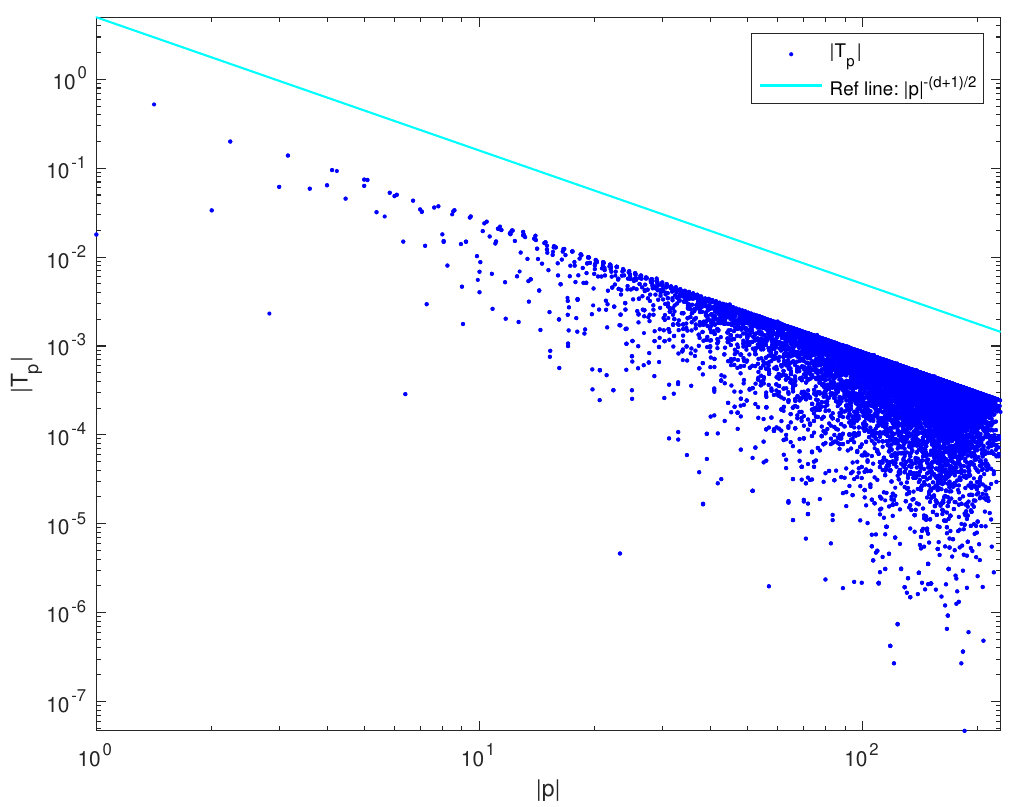}
}
\hspace{-15pt}
\subfigure[2D, $s=0.75$]{
\includegraphics[width=0.33\linewidth]{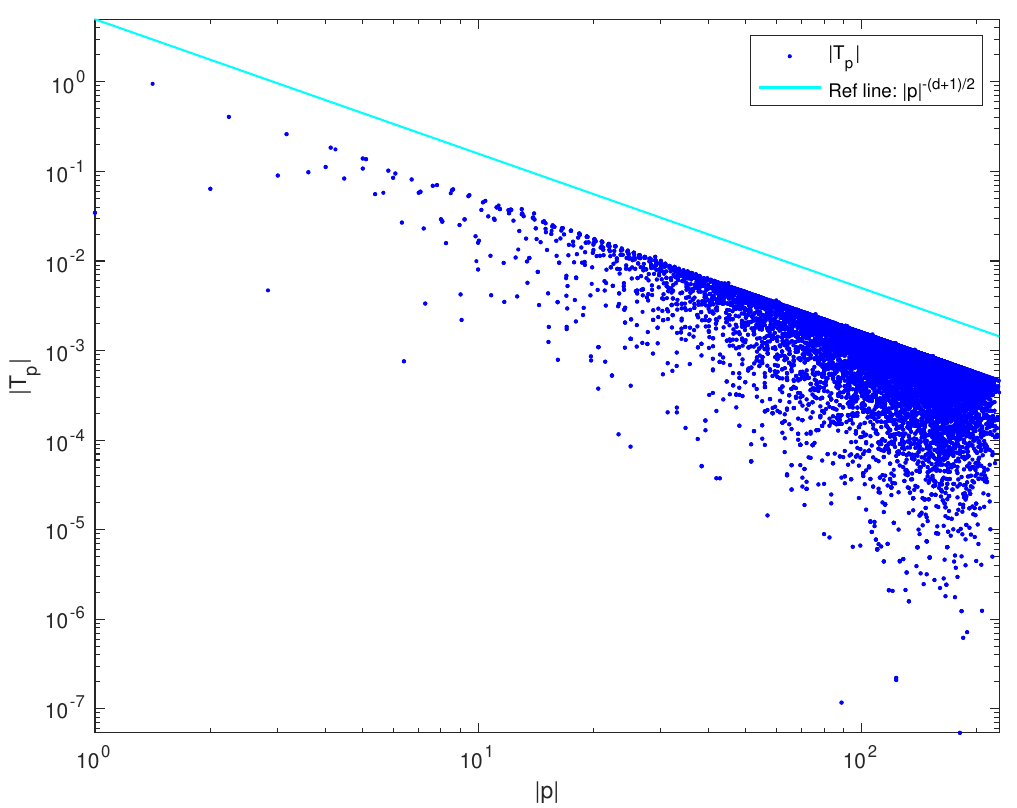}
}
\\
\subfigure[3D, $s=0.25$]{
\includegraphics[width=0.33\linewidth]{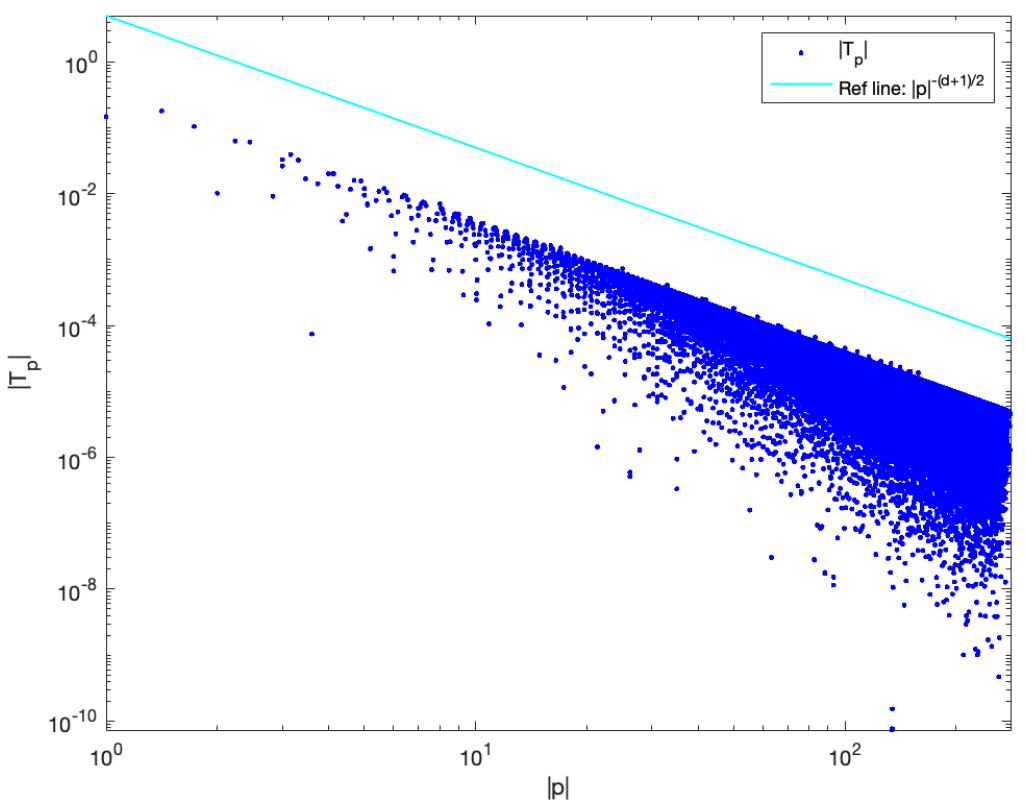}
}
\hspace{-15pt}
\subfigure[3D, $s=0.5$]{
\includegraphics[width=0.33\linewidth]{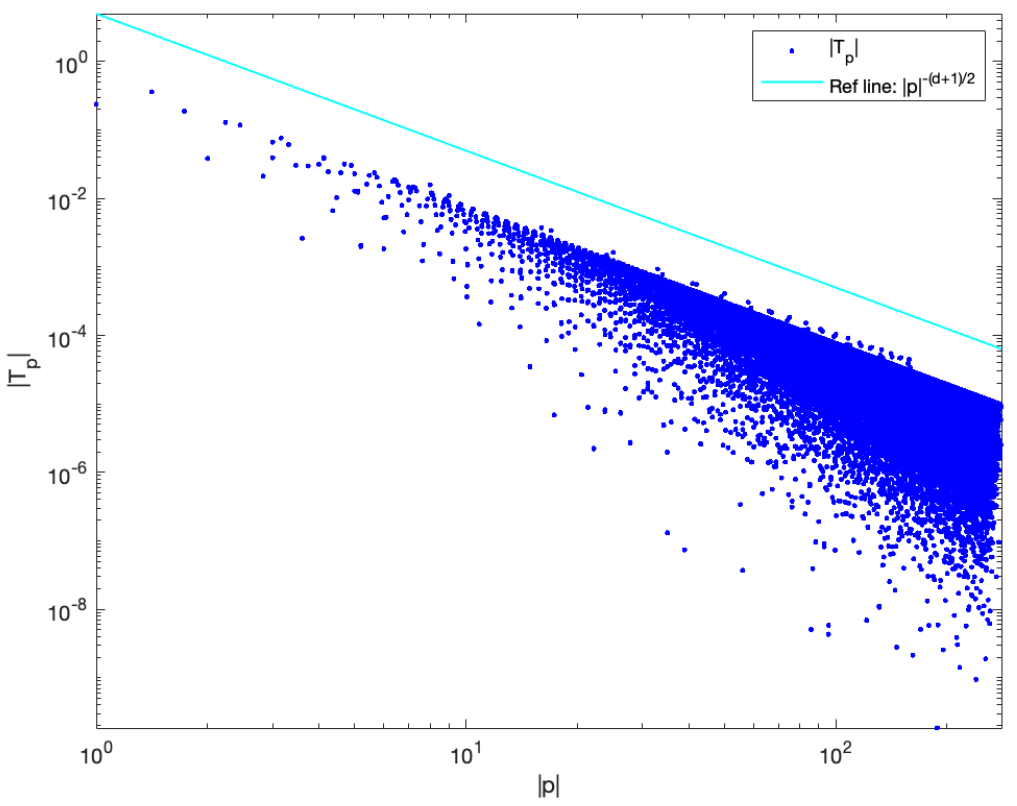}
}
\hspace{-15pt}
\subfigure[3D, $s=0.75$]{
\includegraphics[width=0.33\linewidth]{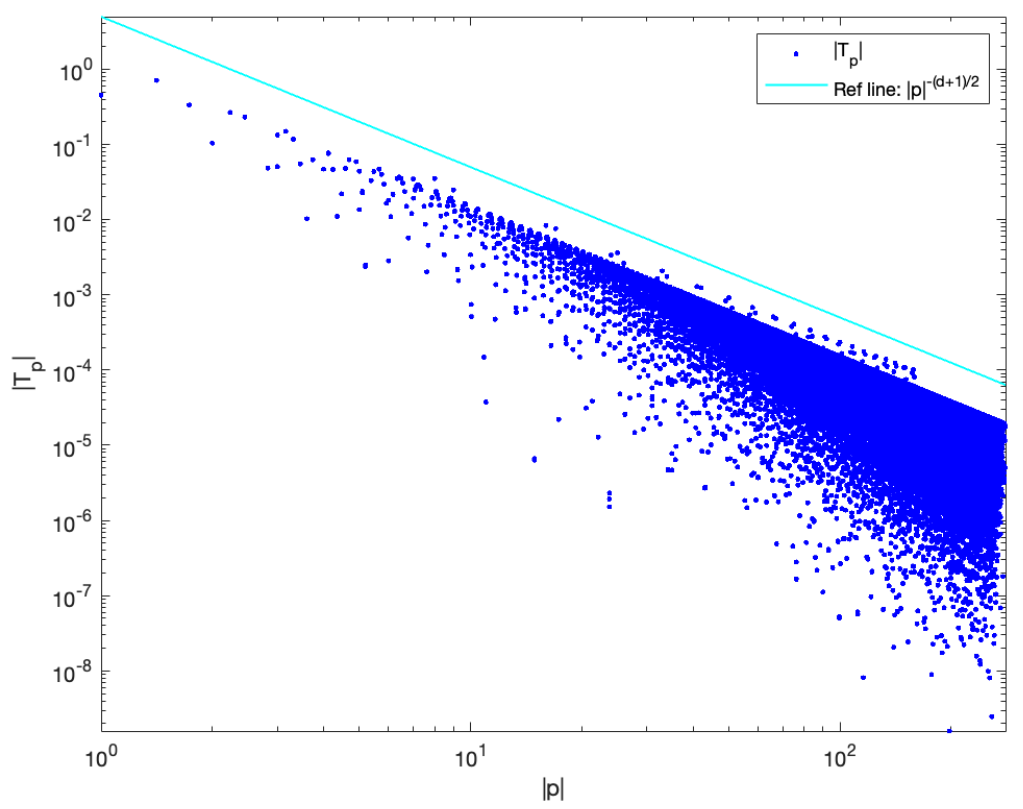}
}
\caption{The decay of $\tilde{\tilde{T}}_{\vec{p}}$ as $|\vec{p}| \to \infty$
for $s = 0.25$, 0.50, and 0.75. The reference line is
$T_{\vec{p}} = |\vec{p}|^{-(d+1)/2}$ (for 2D and 3D).}
\label{fig:T-decay-3}
\end{figure}

\subsection{Summary of approximations of the stiffness matrix $T$}

Properties of the four approximations for the stiffness matrix $T$ are listed in Table~\ref{table:approximate-T-3}.
The FFT and non-uniform FFT approximations are based on the same stiffness matrix $T$ defined in (\ref{T-2}).
The FFT approximation has good accuracy for large $s$ (say, $s\ge 0.5$) but suffers from the low regularity of the integrand at the origin for small $s$.
The non-uniform FFT approximation provides good accuracy for all $s \in (0,1)$ (is less accurate than FFT for large $s$) but is expensive.
The spectral and modified spectral approximations are fast and do not suffer from the low regularity of the integrand for small $s$
but are based on different stiffness matrices, (\ref{T-3}) and (\ref{T-10}), respectively. These matrices
have slower asymptotic decay rates than that of the stiffness matrix (\ref{T-2}) except in 1D where  (\ref{T-2})
and (\ref{T-10}) are the same.

\begin{table}[htb]
\begin{center}
\caption{Summary of properties of the four approximations of the stiffness matrix. The preconditioning and solution
of BVPs are tested with GoFD (cf. the linear system (\ref{GoFD-2})). The accuracy of the FFT approximation is
not good for small $s$. All four approximations lead to expected convergence order of GoFD for BVPs.}
\begin{tabular}{|c|c|c|c|c|c|c|c|}\hline \hline
 & & How fast & \multicolumn{2}{|c|}{Asymptotic} & \multicolumn{2}{|c|}{Precond effective} & Work \\ \cline{4-7}
Approximation & $T$ & to compute & $d=1$ & $d\ge 2$ & sparse & circulant &  for BVPs \\ \hline \hline
FFT & (\ref{T-2}) & fast & \multicolumn{2}{|c|}{$\frac{1}{|\vec{p}|^{d+2s}}$} & yes & yes & yes \\ \hline
non-uniform FFT & (\ref{T-2}) & slow & \multicolumn{2}{|c|}{$\frac{1}{|\vec{p}|^{d+2s}}$} & yes & yes & yes \\ \hline
Spectral & (\ref{T-3}) & fastest & $\frac{1}{|p|^{\min(2,1+2s)}}$ & $\frac{1}{|\vec{p}|^{(d+1)/2}}$ & no & no & yes \\ \hline
Mod. Spectral & (\ref{T-10}) & fast & $\frac{1}{|p|^{1+2s}}$ & $\frac{1}{|\vec{p}|^{(d+1)/2}}$ & no & yes & yes \\
\hline \hline
\end{tabular}
\label{table:approximate-T-3}
\end{center}
\end{table}

\section{Preconditioning}
\label{SEC:preconditioning}

Recall that the linear algebraic system (\ref{GoFD-2}) can be solved efficiently with an iterative method although the coefficient matrix
is dense. This is due mainly to the fact that $I_h^{\text{FD}}$ is sparse and the multiplication of $A_{\text{FD}}$ with vectors
can be carried out efficiently using FFT. We consider preconditioning for the iterative solution of the linear system (\ref{GoFD-2})
to further improve the computational efficiency.

A sparse preconditioner for solving (\ref{GoFD-2}) has been suggested in \cite{HS-2024-GoFD}.
First,  a sparsity pattern based on the FD discretization of the Laplacian is chosen. For example, 9-point and 27-point patterns
can be taken in 2D and 3D, respectively.
Then, a sparse matrix using the entries of $A_{\text{FD}}$ at the positions specified by the pattern can be formed.
Denote these matrices by $A_{\text{FD}}^{(9)}$ and $A_{\text{FD}}^{(27)}$, respectively.
Next, define
\begin{equation}
\label{Ah-4}
A_h^{(9)} =  (I_{h}^{\text{FD}} )^T A_{\text{FD}}^{(9)} I_{h}^{\text{FD}} ,
\qquad A_h^{(27)} = (I_{h}^{\text{FD}} )^T A_{\text{FD}}^{(27)} I_{h}^{\text{FD}}  .
\end{equation}
Finally, the sparse preconditioner for (\ref{GoFD-2}) is obtained using the modified incomplete Cholesky decomposition
of $A_h^{(9)}$ or $A_h^{(27)}$ with dropoff threshold $10^{-3}$. Notice that the so-obtained preconditioner is sparse
and can be computed and applied economically.

Next, we consider a different preconditioner that is based on circulant preconditioners for Toeplitz systems
(see, e.g., \cite{Chan-1996,Chan1988}). In principle, circulant preconditioners for Toeplitz systems can be applied directly
to $A_{\text{FD}}$ since $A_{\text{FD}}$ is a block Toeplitz
matrix with Toeplitz blocks. Here, we present a simple and explicit description based on (\ref{AFD-2}).
Recall that constructing a preconditioner for $A_{\text{FD}}$ means finding a way to solve $u_{m,n}$'s
for given $A_{\text{FD}} \vec{u}_{\text{FD}}$. Unfortunately, we cannot do this directly from
(\ref{AFD-2}) because the Fourier transforms there use
$4 N_{\text{FD}}$ sample points in each axial direction but only $2 N_{\text{FD}}$ points in each axial direction
are used for $u$. To avoid this difficulty, we reduce the number of sampling points to $2N_{\text{FD}}$ and define
an approximation for $A_{\text{FD}}$ as
\begin{align}
(\tilde{A}_{\text{FD}}\vec{u}_{\text{FD}})_{(j,k)}
& = \frac{1}{(2 N_{\text{FD}})^2} \sum_{p=-N_{\text{FD}}}^{N_{\text{FD}}-1} \sum_{q=-N_{\text{FD}}}^{N_{\text{FD}}-1} \check{T}_{p,q} \check{u}_{p,q}
(-1)^{p+N_{\text{FD}} + q+N_{\text{FD}}}
\notag \\
& \qquad \qquad \qquad \qquad \cdot
e^{\frac{i 2 \pi (p+N_{\text{FD}})(j+N_{\text{FD}})}{2N_{\text{FD}}} + \frac{i 2 \pi (q+2N_{\text{FD}})(k+N_{\text{FD}})}{2N_{\text{FD}}}} ,
\quad - N_{\text{FD}} \le j, k \le N_{\text{FD}}
\label{AFD-4}
\end{align}
where
\begin{align}
& \check{T}_{p,q} = \sum_{m=-N_{\text{FD}}}^{N_{\text{FD}}-1} \sum_{n=-N_{\text{FD}}}^{N_{\text{FD}}-1} T_{m,n}
e^{-\frac{i 2 \pi (m+N_{\text{FD}})(p+N_{\text{FD}})}{2N_{\text{FD}}}
-\frac{i 2 \pi (n+N_{\text{FD}})(q+N_{\text{FD}})}{2N_{\text{FD}}}},
\label{T-11}
\\
& \check{u}_{p,q}  = \sum_{m = - N_{\text{FD}}}^{ N_{\text{FD}}-1} \sum_{n=- N_{\text{FD}}}^{ N_{\text{FD}}-1}
u_{m,n} e^{-\frac{i 2 \pi (m+N_{\text{FD}})(p+N_{\text{FD}})}{2N_{\text{FD}}}
-\frac{ i 2 \pi (n+N_{\text{FD}})(q+N_{\text{FD}})}{2N_{\text{FD}}}} .
\label{uhat-1}
\end{align}
We hope that $\tilde{A}_{\text{FD}}$ is a good approximation of $A_{\text{FD}}$. Most importantly, using the discrete Fourier transform and its inverse
we can find the inverse of $\tilde{A}_{\text{FD}}$ as
\begin{align}
(\tilde{A}_{\text{FD}}^{-1}\vec{v}_{\text{FD}})_{(j,k)}
& = \frac{1}{(2 N_{\text{FD}})^2} \sum_{p=-N_{\text{FD}}}^{N_{\text{FD}}-1} \sum_{q=-N_{\text{FD}}}^{N_{\text{FD}}-1} \check{T}_{p,q}^{-1} \check{v}_{p,q}
(-1)^{p+N_{\text{FD}} + q+N_{\text{FD}}}
\notag \\
& \qquad \qquad \qquad \qquad \cdot
e^{\frac{i 2 \pi (p+N_{\text{FD}})(j+N_{\text{FD}})}{2N_{\text{FD}}} + \frac{i 2 \pi (q+2N_{\text{FD}})(k+N_{\text{FD}})}{2N_{\text{FD}}}} ,
\quad - N_{\text{FD}} \le j, k \le N_{\text{FD}}
\label{AFD-5}
\end{align}
where $\check{T}$ is given in (\ref{T-11}) and $\check{v}$ is the discrete Fourier transform of $v_h$ (which has a similar expression as (\ref{uhat-1})).
The computation of $\tilde{A}_{\text{FD}}^{-1}\vec{v}_{\text{FD}}$ requires three FFT/inverse FFT,
totaling $\mathcal{O}(N_{\text{FD}}^{d} \log (N_{\text{FD}}^{d}))$ flops (in $d$-dimensions).

Once we have obtained the preconditioner for $A_{\text{FD}}$, we can form the preconditioner for the coefficient matrix
of (\ref{GoFD-2}). Notice that
\begin{equation}
\label{Ah-inverse-1}
((I_{h}^{\text{FD}})^T A_{\text{FD}} I_{h}^{\text{FD}})^{-1}
= (I_{h}^{\text{FD}})^{+} A_{\text{FD}}^{-1} ((I_{h}^{\text{FD}})^{+})^T,
\end{equation}
where $(I_{h}^{\text{FD}})^{+}$ is the Moore-Penrose pseudo-inverse of $I_{h}^{\text{FD}}$,
\begin{equation}
\label{Ih-pseudo-inverse}
(I_{h}^{\text{FD}})^{+} = (I_{h}^{\text{FD}})^T \Big ( (I_{h}^{\text{FD}})^T I_{h}^{\text{FD}} \Big )^{-1}.
\end{equation}
The final circulant preconditioner for the system (\ref{GoFD-2})
is obtained by replacing $A_{\text{FD}}^{-1}$ in (\ref{Ah-inverse-1}) with $\tilde{A}_{\text{FD}}^{-1}$ and
$(I_{h}^{\text{FD}})^T I_{h}^{\text{FD}}$ in (\ref{Ih-pseudo-inverse}) with
its modified incomplete Cholesky decomposition with dropoff threshold $10^{-3}$.

Convergence histories of PCG for solving (\ref{GoFD-2}) are plotted in Fig.~\ref{fig:precond-1}.
It can be seen that for the FFT approximation of the stiffness matrix,
both the sparse and circulant preconditioners perform well, significantly reducing the number of PCG iterations
to reach a prescribed tolerance, with the latter doing slightly better than the former.
For the spectral approximation, both preconditioners do not work, requiring more or much more iterations
than CG without preconditioning. Recall that the spectral approximation uses a different stiffness matrix,
(\ref{T-3}), which has the different integrand and domain from (\ref{T-2}).
For the modified spectral approximation, the stiffness matrix is given in (\ref{T-10}), which differs
from (\ref{T-2}) only by the integration domain. In this case, the circulant preconditioner works well as it does for the FFT approximation.
This actually is the main motivation to introduce the modified spectral approximation.
Unfortunately, the sparse preconditioner still does not work for this case.

At this point, it remains unclear why both sparse and circulant preconditioners do not work for the spectral approximation.
Understanding this and developing effective preconditioners for the spectral approximation
is an interesting topic for future research.

\begin{figure}[ht!]
\centering
\subfigure[2D, FFT appr.]{
\includegraphics[width=0.33\linewidth]{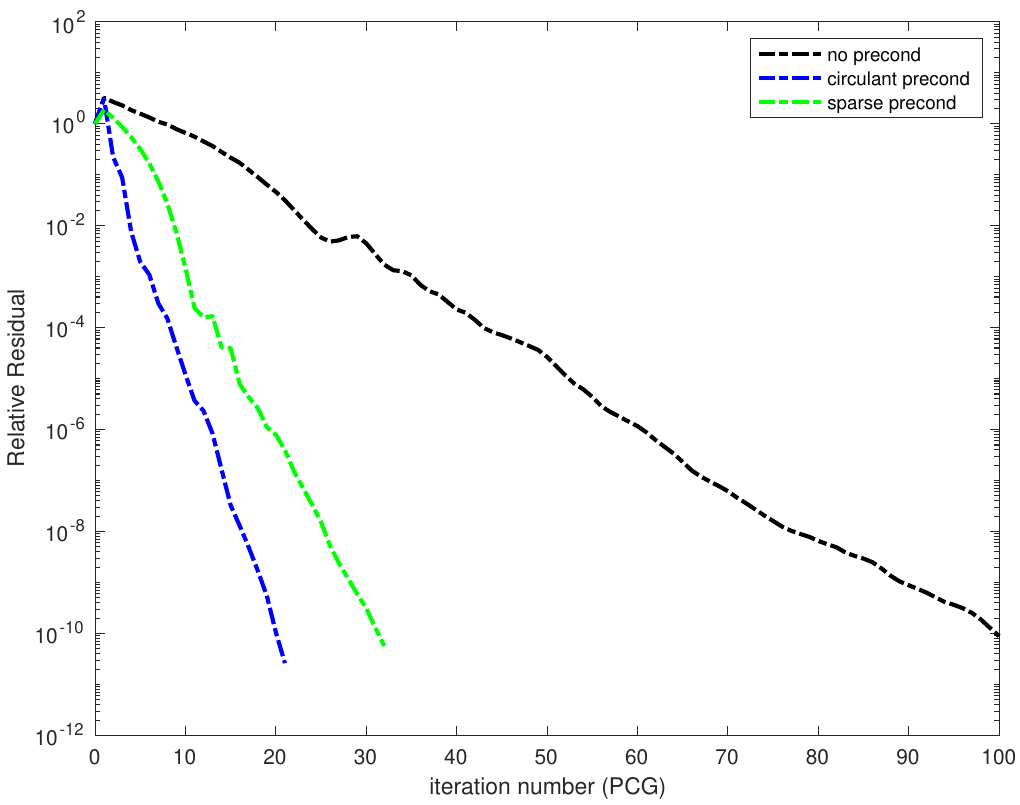}
}
\hspace{-15pt}
\subfigure[2D, Spectral appr.]{
\includegraphics[width=0.33\linewidth]{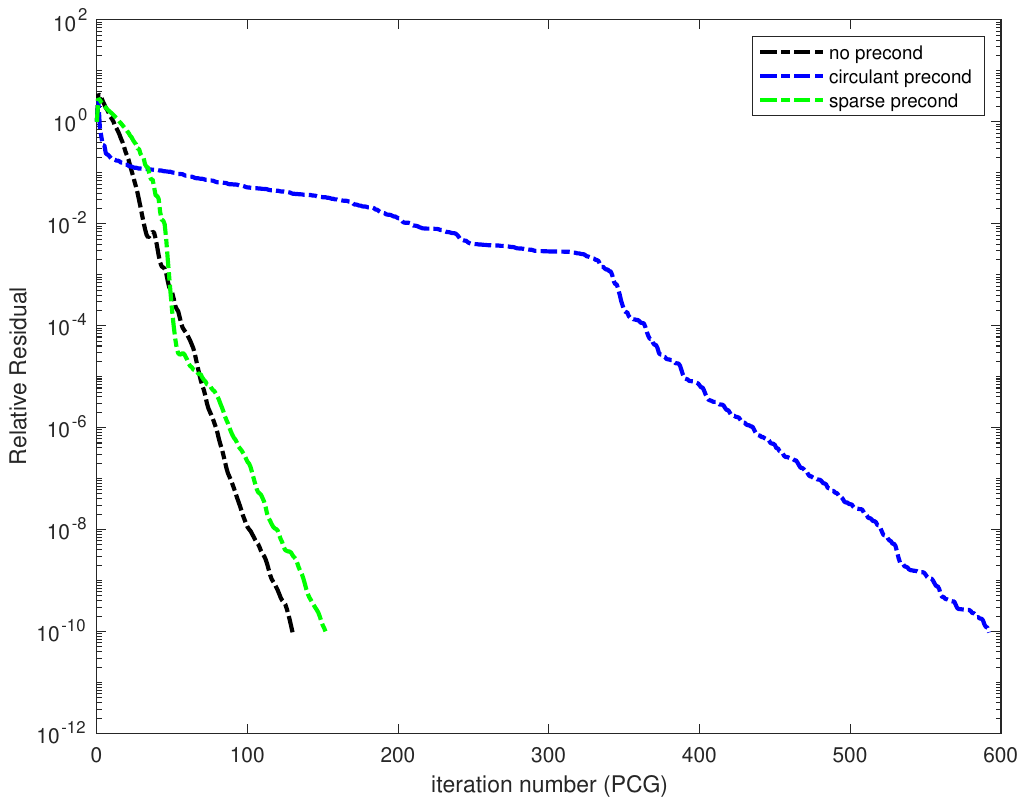}
}
\hspace{-15pt}
\subfigure[2D, Mod. spectral appr.]{
\includegraphics[width=0.33\linewidth]{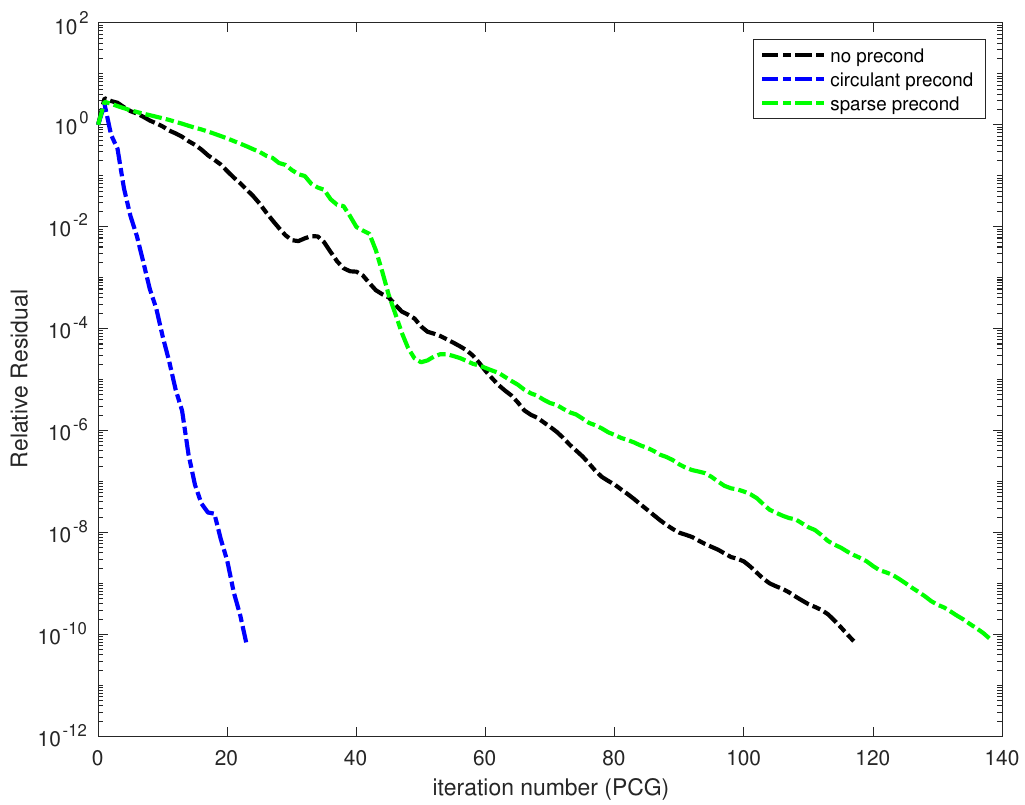}
}
\\
\subfigure[3D, FFT appr.]{
\includegraphics[width=0.33\linewidth]{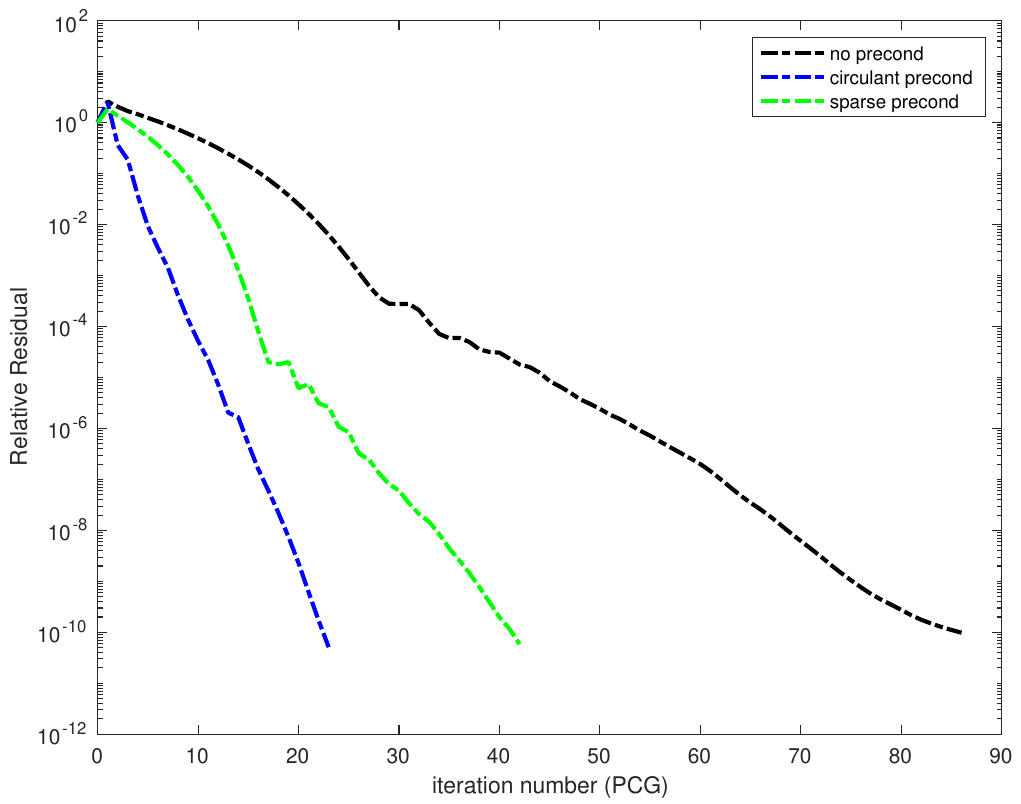}
}
\hspace{-15pt}
\subfigure[3D, Spectral appr.]{
\includegraphics[width=0.33\linewidth]{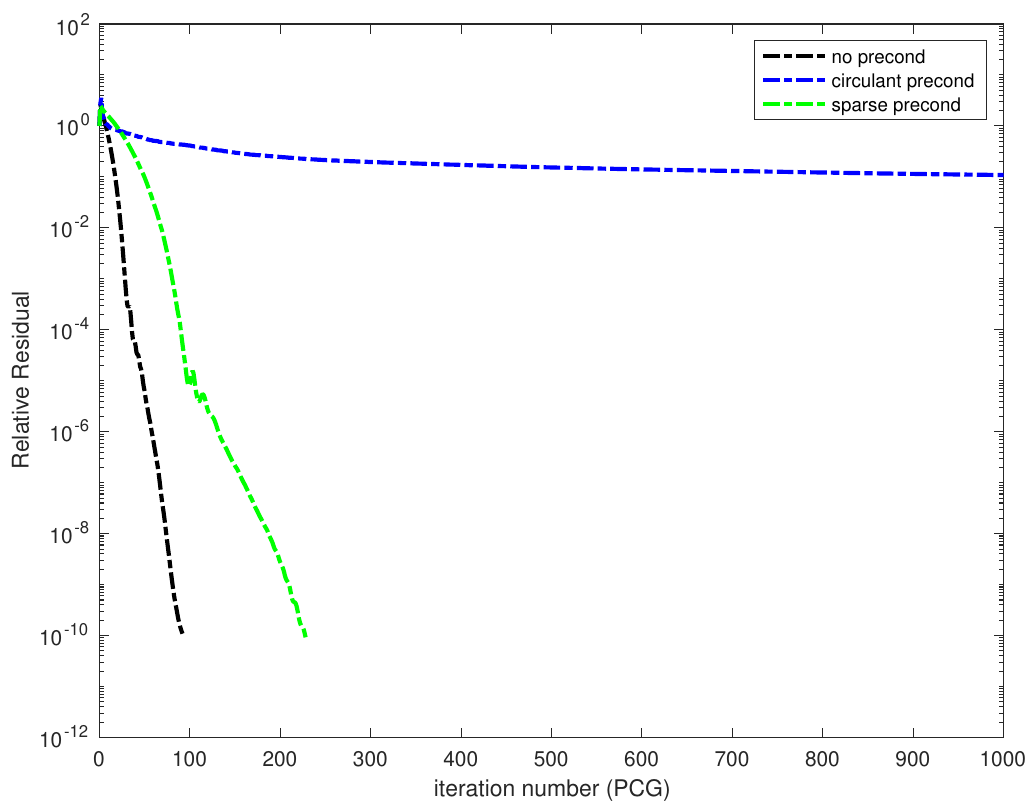}
}
\hspace{-15pt}
\subfigure[3D, Mod. spectral appr.]{
\includegraphics[width=0.33\linewidth]{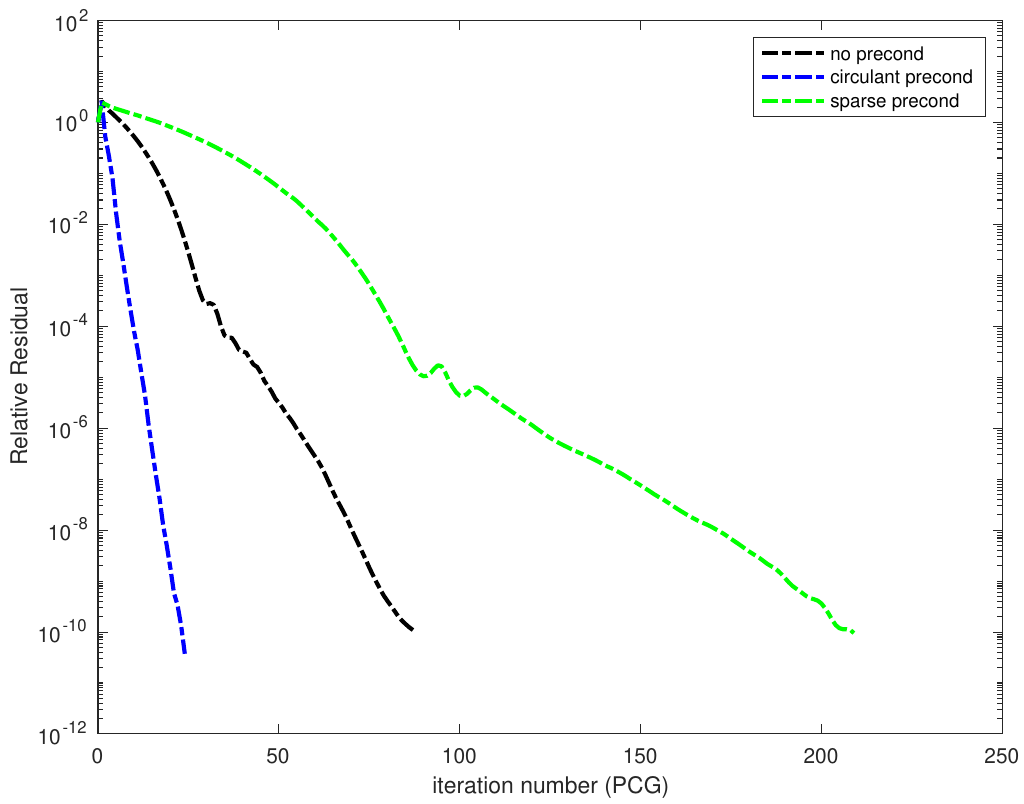}
}
\caption{Example (\ref{main-example}) with $s = 0.75$ and in 2D and 3D.
Convergence histories of PCG for solving (\ref{GoFD-2}) with $tol = 10^{-10}$ for three different approximations
of the stiffness matrix $T$. $M = 2^{14}$, $n_G = 64$, $N = 11,930$ (for mesh $\mathcal{T}_h$), and
$N_{\text{FD}} = 75$ (for mesh $\mathcal{T}_{\text{FD}}$) are used in 2D and
$M = 2^{10}$, $n_G = 64$, $N = 922,447$ (for mesh $\mathcal{T}_h$), and
$N_{\text{FD}} = 110$ (for mesh $\mathcal{T}_{\text{FD}}$) are used in 3D. }
\label{fig:precond-1}
\end{figure}

\section{Numerical examples}
\label{SEC:numerics}

In this section we present numerical results obtained for Example~(\ref{main-example}) in 2D and 3D
to demonstrate that GoFD with all four approximations for the stiffness matrix
produces numerical solutions for BVP (\ref{BVP-1}) with expected convergence order.
In our computation, the linear system (\ref{GoFD-2}) is solved using PCG with the circulant preconditioner
except for the spectral approximation when CG without preconditioning is used.
We use $n_G = 64$ (in the spectral and modified spectral approximations)
and $M=2^{14}$ (for 2D) and $M=2^{10}$ (for 3D) in the FFT and modified spectral approximations.
Results obtained with the non-uniform FFT approximation are omitted here
since they are almost identical to those obtained with the FFT approximation
and these two approximations are based on the same stiffness matrix (\ref{T-2}).

The $L^2$ norm of the error obtained with GoFD with quasi-uniform meshes is plotted as a function of $N$ (the number of elements
of mesh $\mathcal{T}_h$) in Figs.~\ref{fig:GoFD_Err-2d-1} and \ref{fig:GoFD_Err-3d-1} for 2D and 3D, respectively.
The results are almost identical for all three approximations for the stiffness matrix.
Moreover, the error shows a expected convergence rate,
$\mathcal{O}(\bar{h}^{\min(1, 0.5+s)})$, where $\bar{h} = N^{-1/d}$
is the average element diameter.
These results show that GoFD works well for solving BVP (\ref{BVP-1}) with all four approximations of the stiffness matrix.

Next, we consider adaptive meshes.
It is known (see, e.g., \cite{Acosta201701,Ros-Oton-2014})
that the solution of (\ref{BVP-1}) has low regularity near the boundary of $\Omega$,
\[
u(\vec{x}) \sim \text{dist}^s(\vec{x}, \partial \Omega),\quad \text{ for } \vec{x} \text{ close to } \partial \Omega
\]
where $\text{dist}(\vec{x}, \partial \Omega)$ denotes the distance from $\vec{x}$ to the boundary of $\Omega$.
Mesh adaptation is useful to improve accuracy and convergence order.
Here, we use the so-called moving mesh PDE (MMPDE) method \cite{HR11} to generate an adaptive mesh based on
the function $v(\vec{x}) = \text{dist}^s(\vec{x}, \partial \Omega)$. Once an adaptive mesh is generated, BVP (\ref{BVP-1})
is solved using GoFD on the generated mesh. No iteration between mesh generation and BVP solution is taken.
To save space, we do not give the description of the MMPDE method here. The interested reader is referred to
\cite[Section~4]{HS-2024-GoFD} and \cite{HR11} for the detail of the method.

Typical adaptive meshes and computed solutions in 2D are shown in Fig.~\ref{fig:GoFD-mesh-solution-2d}. It can be seen that
mesh concentration is higher near the boundary. To save space, we present the error in Fig.~\ref{fig:GoFD_Err-2d-2}
only for GoFD with the modified spectral approximation of the stiffness matrix.
It is clear that the error converges at a second-order rate $\mathcal{O}(\bar{h}^2)$.
It is also interesting to observe that for $s = 0.25$ and $0.5$, the error has a slower convergence rate before it reaches
the second-order rate for finer meshes.

We present 3D adaptive mesh results in Fig.~\ref{fig:GoFD_Err-3d-2}. The error is smaller than those
in Fig.~\ref{fig:GoFD_Err-3d-1} for quasi-uniform meshes but has not reached second order
especially for $s = 0.25$ and $0.5$ for the considered range of the mesh size.
The computation with meshes of larger $N$ runs out memory. For those cases, the minimum height $a_h$
of the adaptive mesh is very small, requiring a very large $N_{\text{FD}}$ (cf. (\ref{hFD-2})) and thus a large amount
of memory. This reflects a drawback of GoFD with 3D adaptive meshes.

\begin{figure}[ht!]
\centering
\subfigure[FFT appr., $s=0.25$]{
\includegraphics[width=0.33\linewidth]{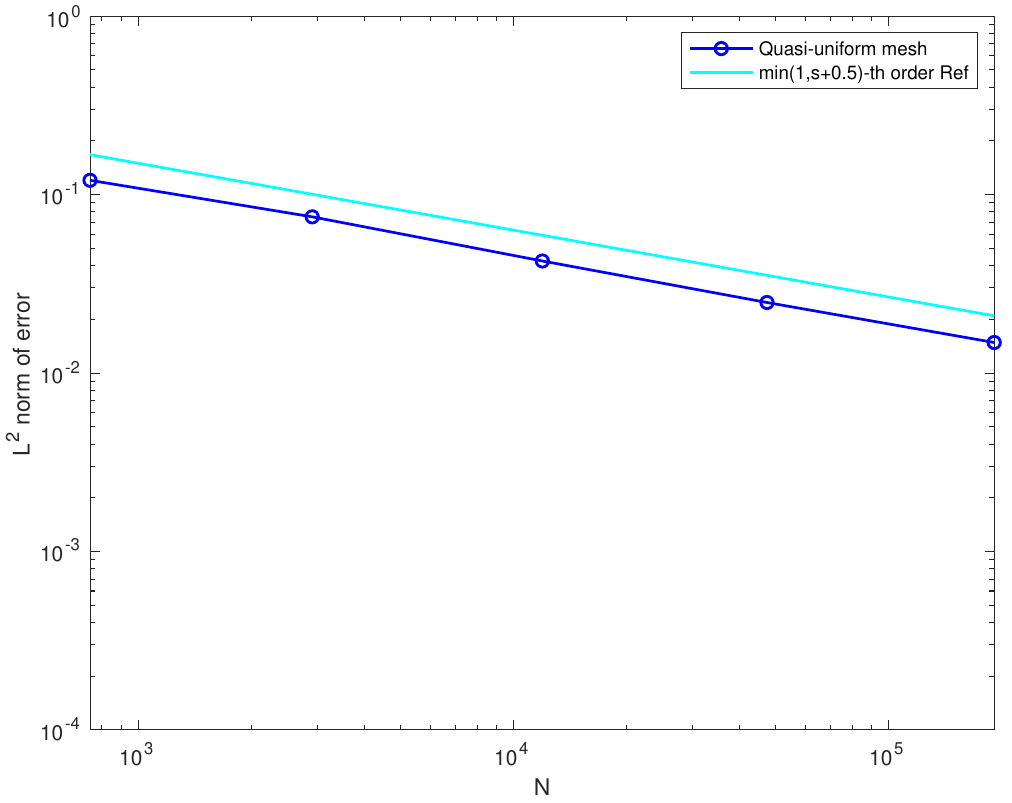}
}
\hspace{-15pt}
\subfigure[FFT appr., $s=0.50$]{
\includegraphics[width=0.33\linewidth]{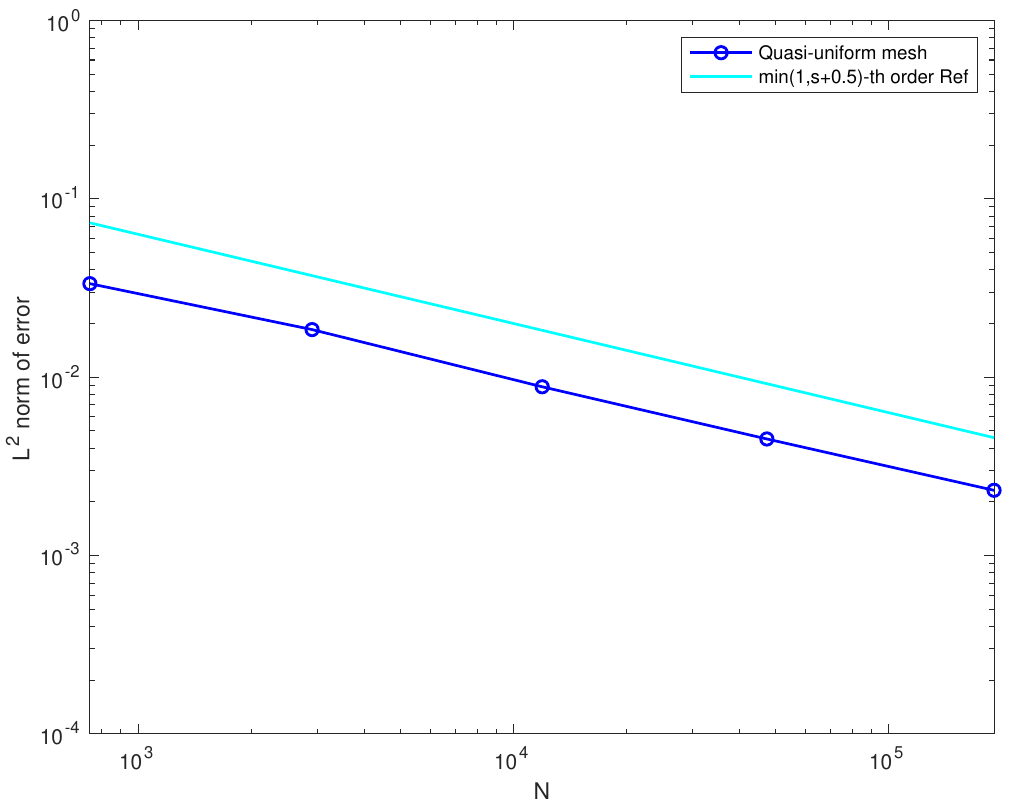}
}
\hspace{-15pt}
\subfigure[FFT appr., $s=0.75$]{
\includegraphics[width=0.33\linewidth]{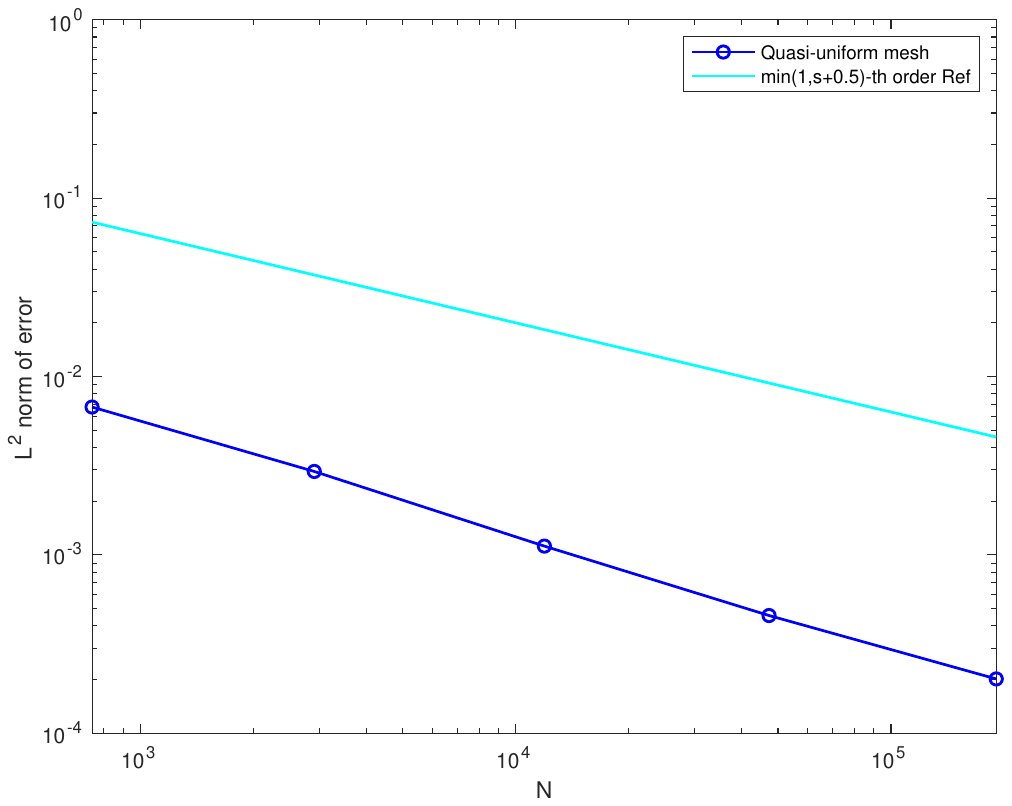}
}
\\
\subfigure[Spectral appr., $s=0.25$]{
\includegraphics[width=0.33\linewidth]{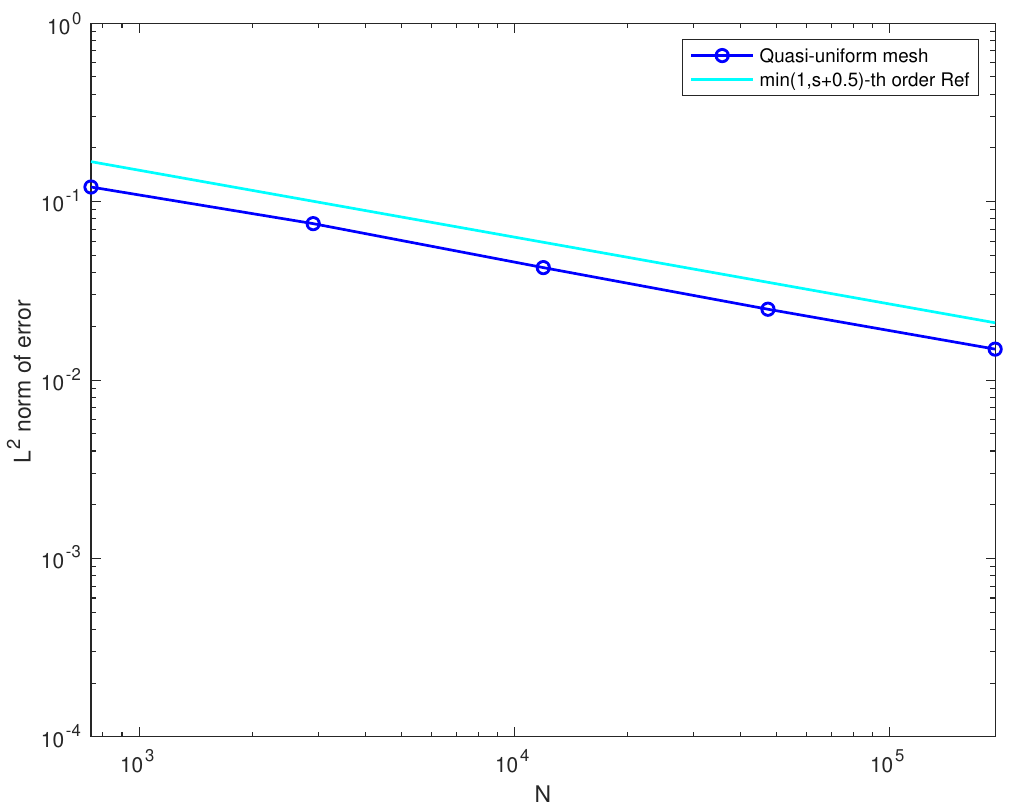}
}
\hspace{-15pt}
\subfigure[Spectral appr., $s=0.50$]{
\includegraphics[width=0.33\linewidth]{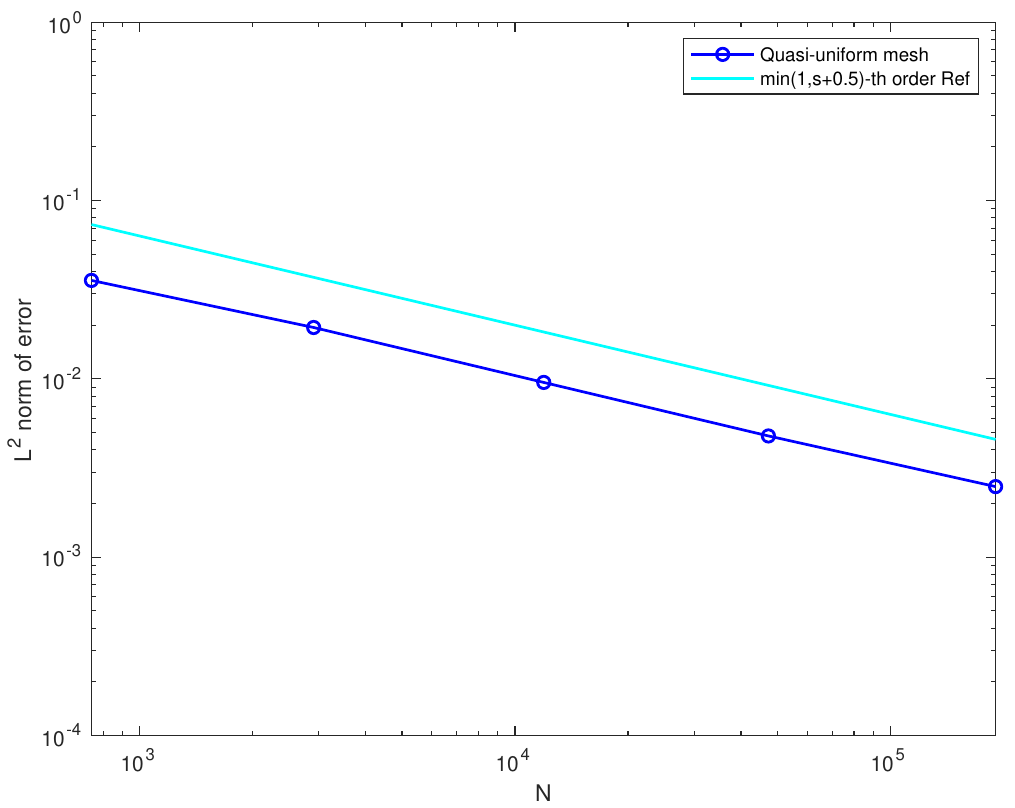}
}
\hspace{-15pt}
\subfigure[Spectral appr., $s=0.75$]{
\includegraphics[width=0.33\linewidth]{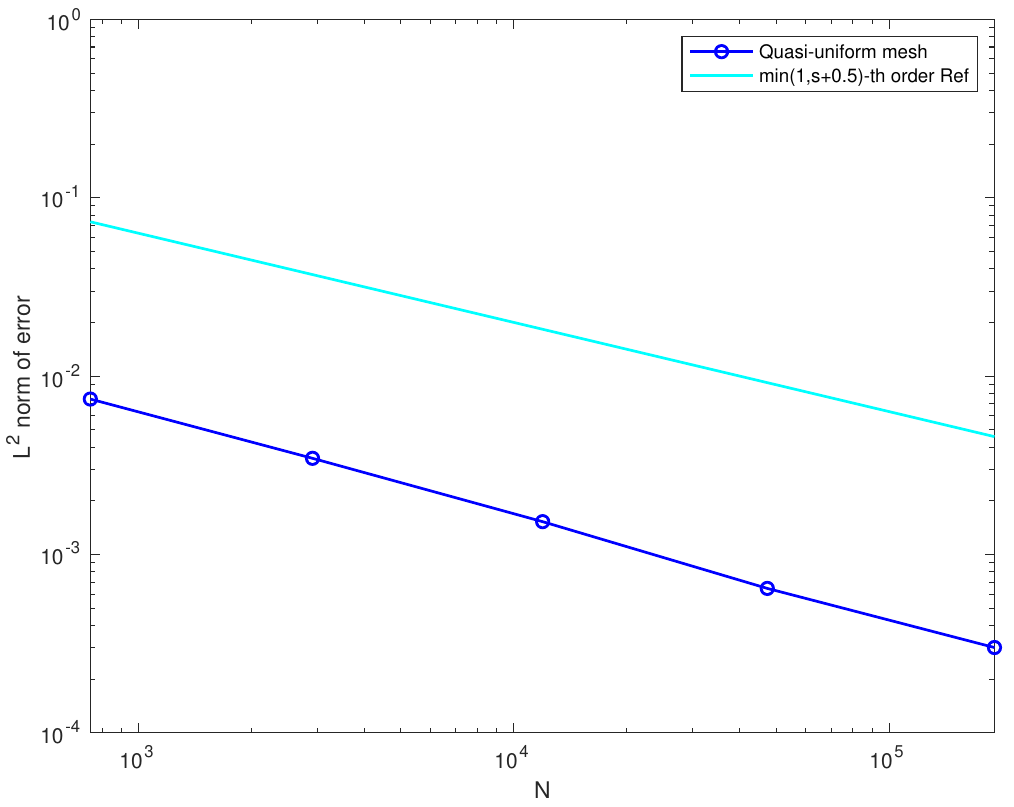}
}
\\
\subfigure[Mod. spectral appr., $s=0.25$]{
\includegraphics[width=0.33\linewidth]{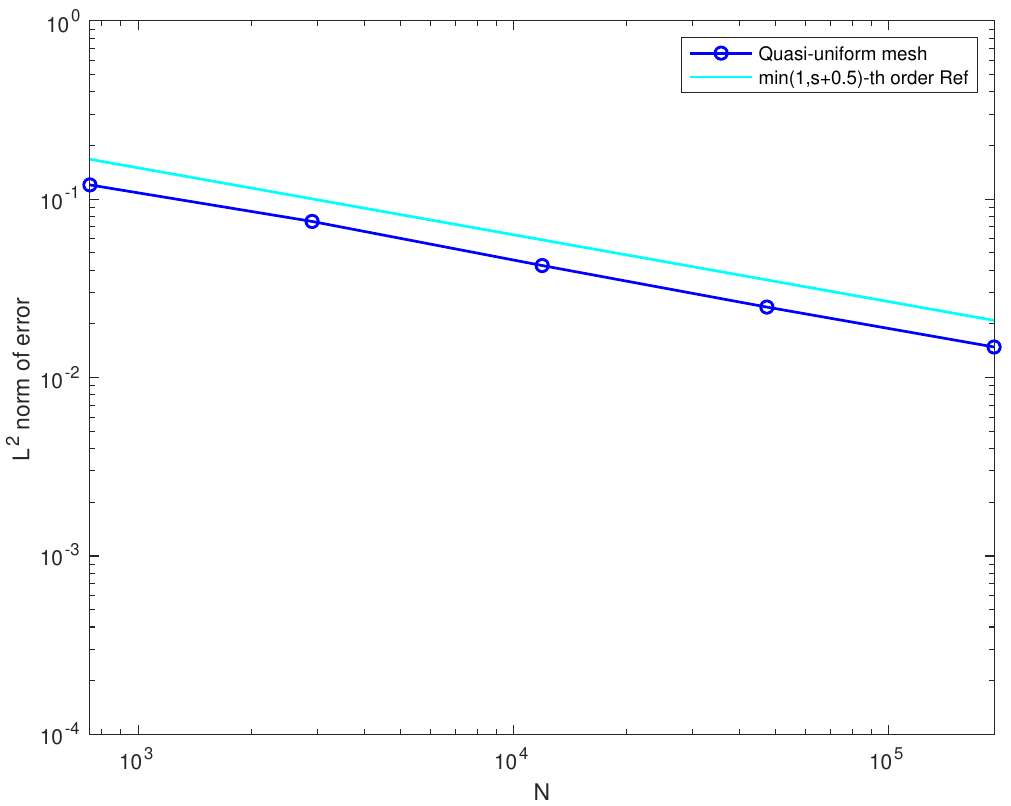}
}
\hspace{-15pt}
\subfigure[Mod. spectral appr., $s=0.50$]{
\includegraphics[width=0.33\linewidth]{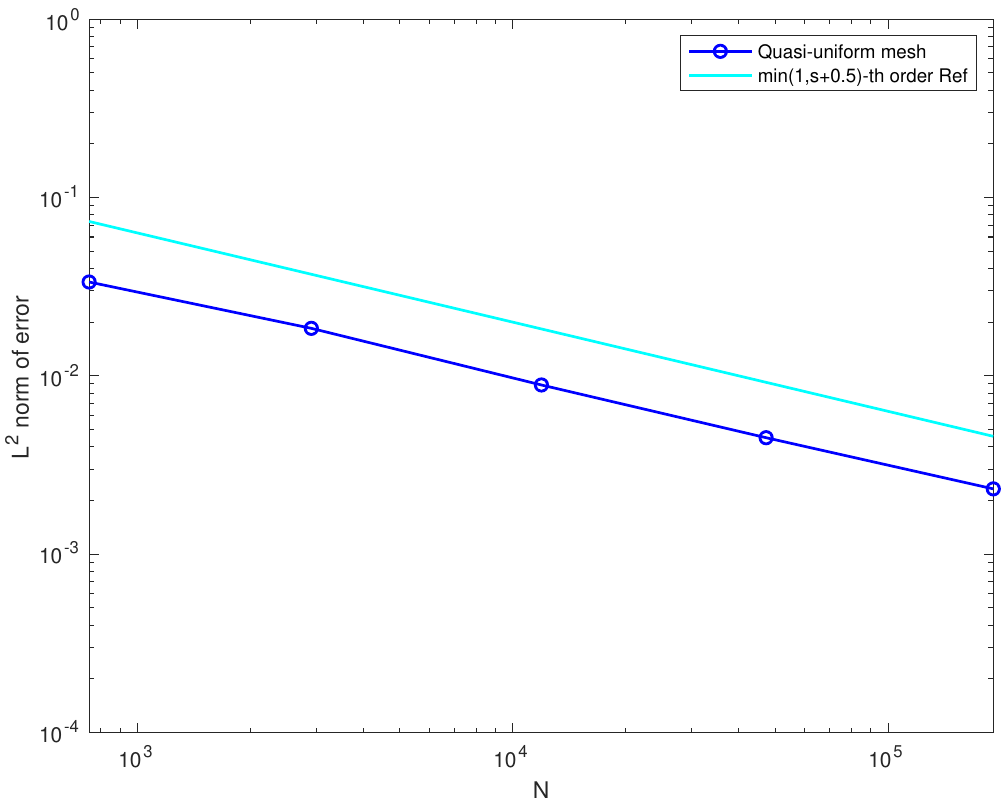}
}
\hspace{-15pt}
\subfigure[Mod. spectral appr., $s=0.75$]{
\includegraphics[width=0.33\linewidth]{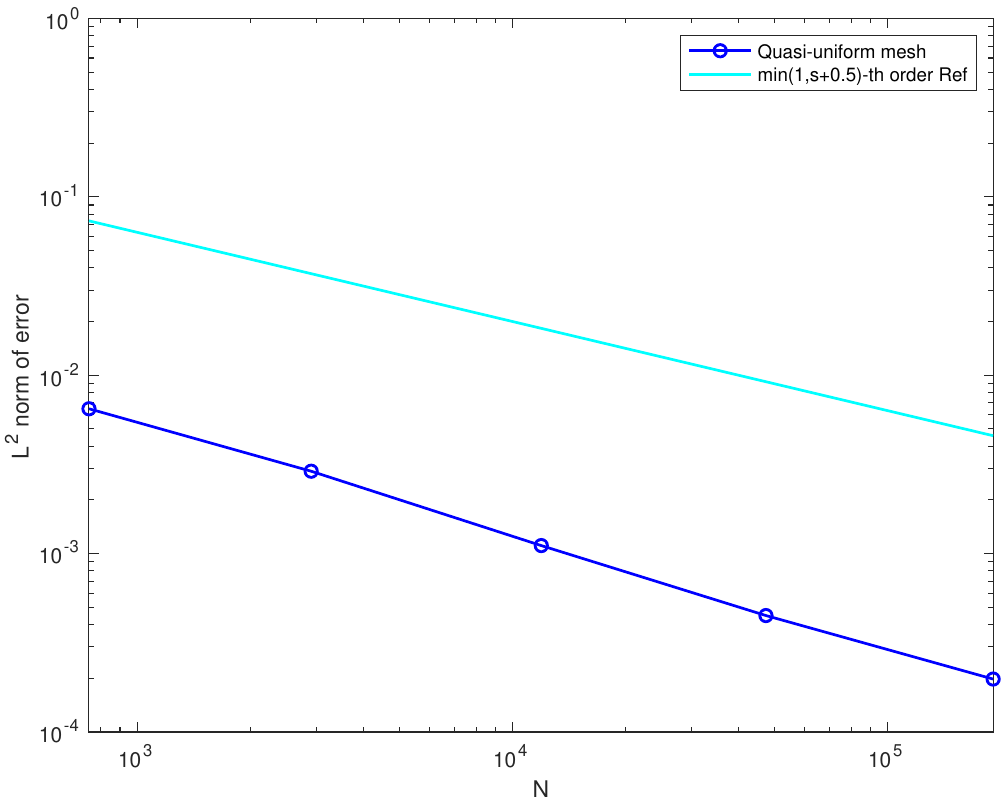}
}
\caption{Example (\ref{main-example}) in 2D. The $L^2$ norm of the error is plotted as a function $N$ for GoFD with various approximations
of the stiffness matrix.}
\label{fig:GoFD_Err-2d-1}
\end{figure}

\begin{figure}[ht!]
\centering
\subfigure[FFT appr., $s=0.25$]{
\includegraphics[width=0.33\linewidth]{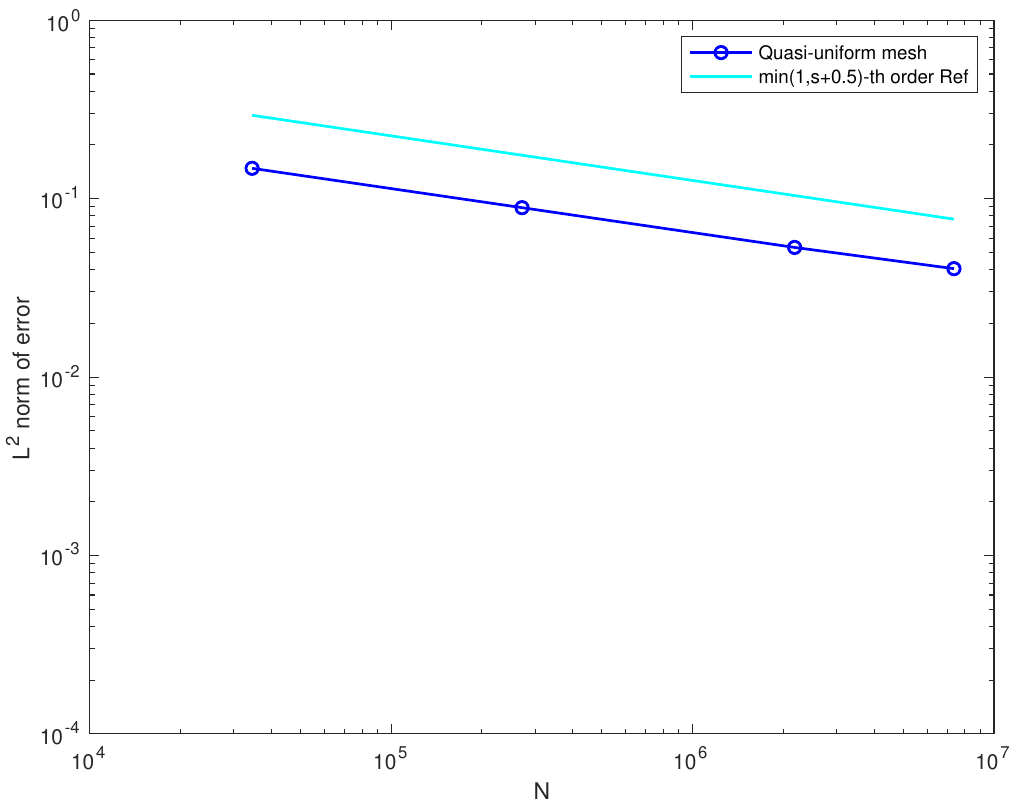}
}
\hspace{-15pt}
\subfigure[FFT appr., $s=0.50$]{
\includegraphics[width=0.33\linewidth]{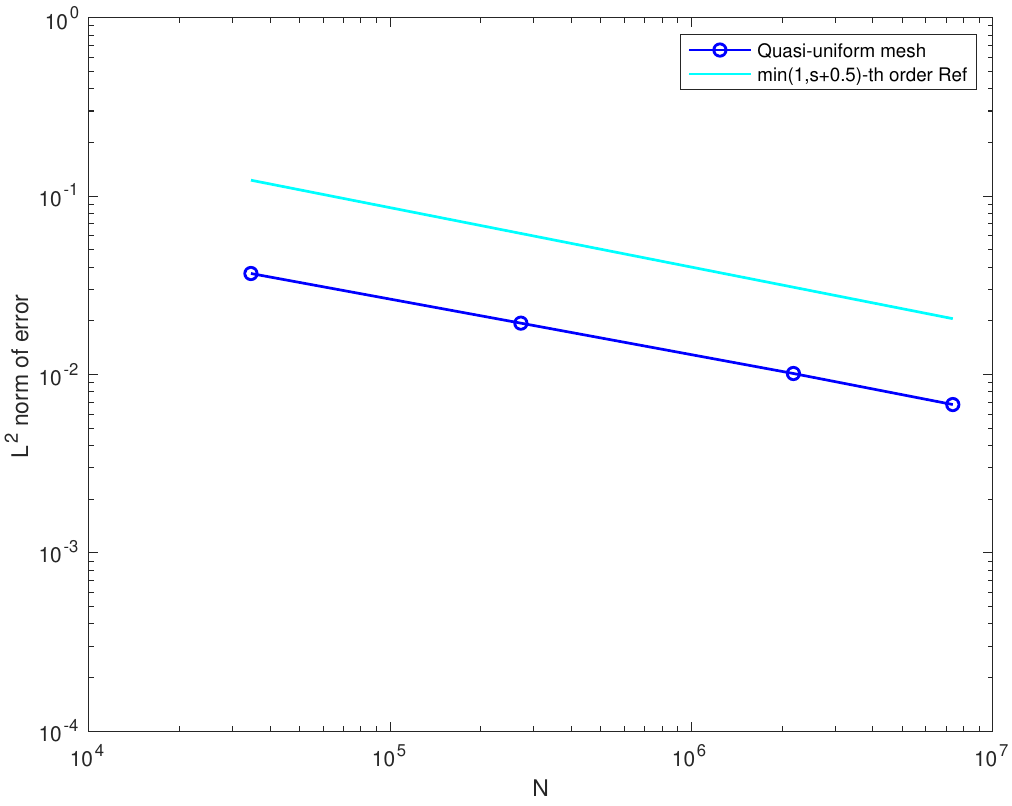}
}
\hspace{-15pt}
\subfigure[FFT appr., $s=0.75$]{
\includegraphics[width=0.33\linewidth]{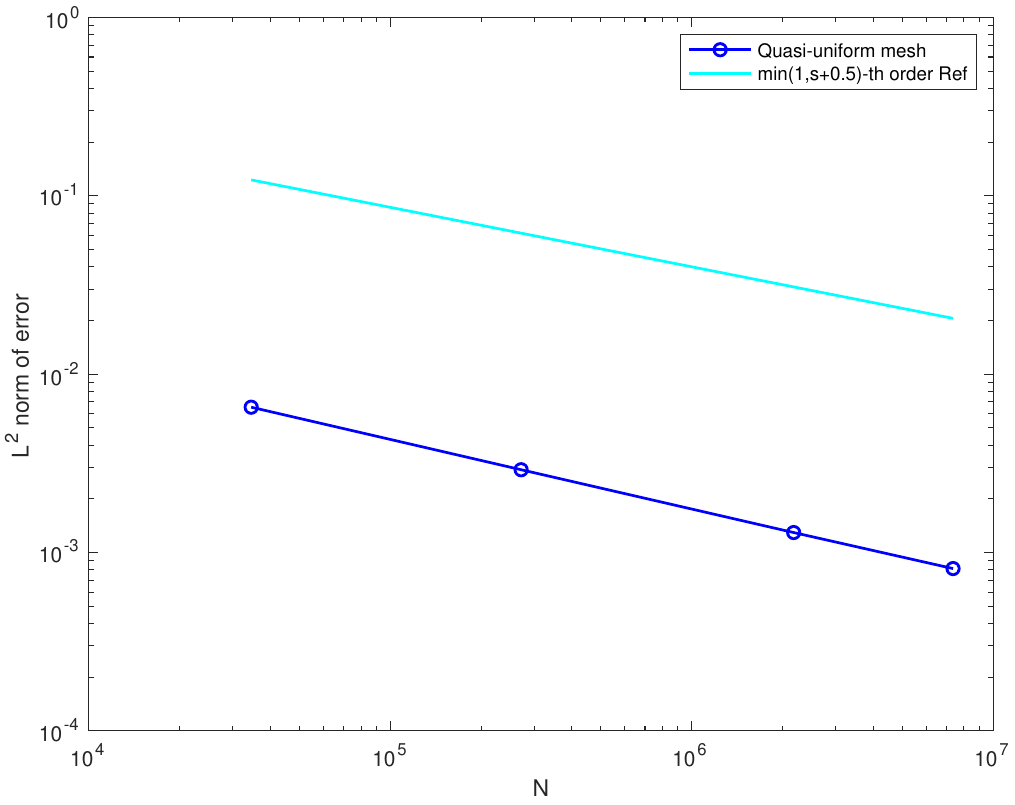}
}
\\
\subfigure[Spectral appr., $s=0.25$]{
\includegraphics[width=0.33\linewidth]{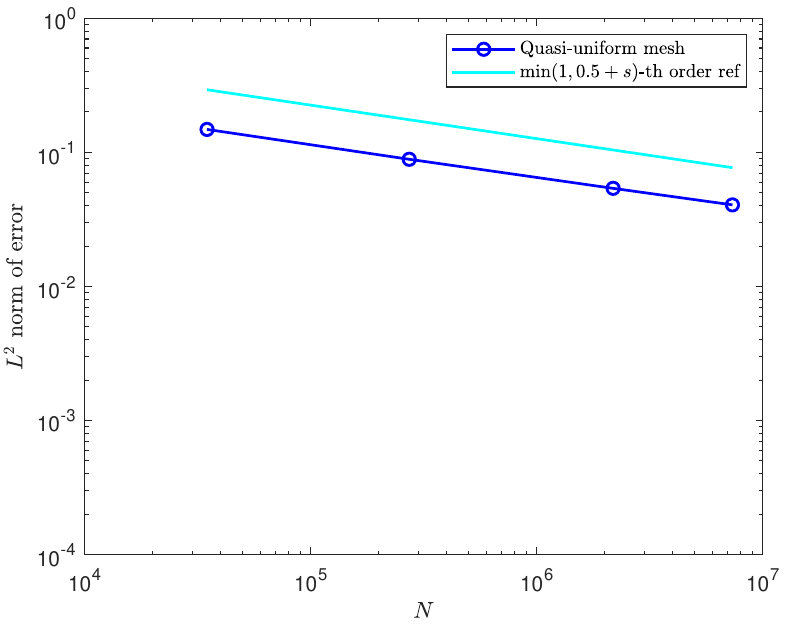}
}
\hspace{-15pt}
\subfigure[Spectral appr., $s=0.50$]{
\includegraphics[width=0.33\linewidth]{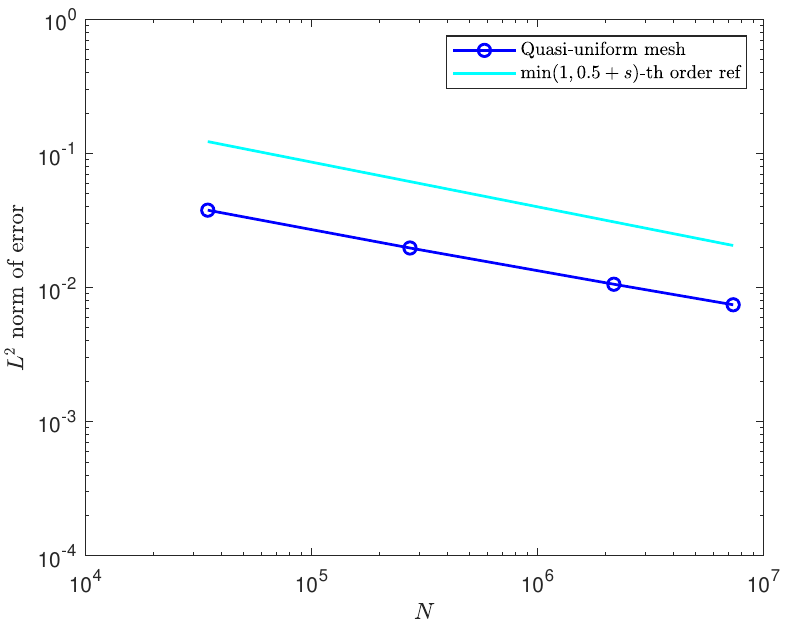}
}
\hspace{-15pt}
\subfigure[Spectral appr., $s=0.75$]{
\includegraphics[width=0.33\linewidth]{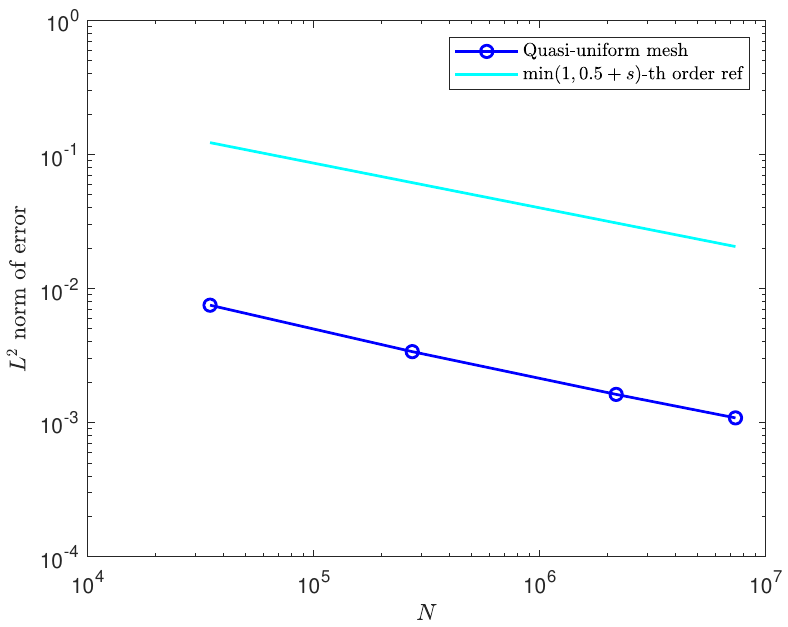}
}
\\
\subfigure[Mod. spectral appr., $s=0.25$]{
\includegraphics[width=0.33\linewidth]{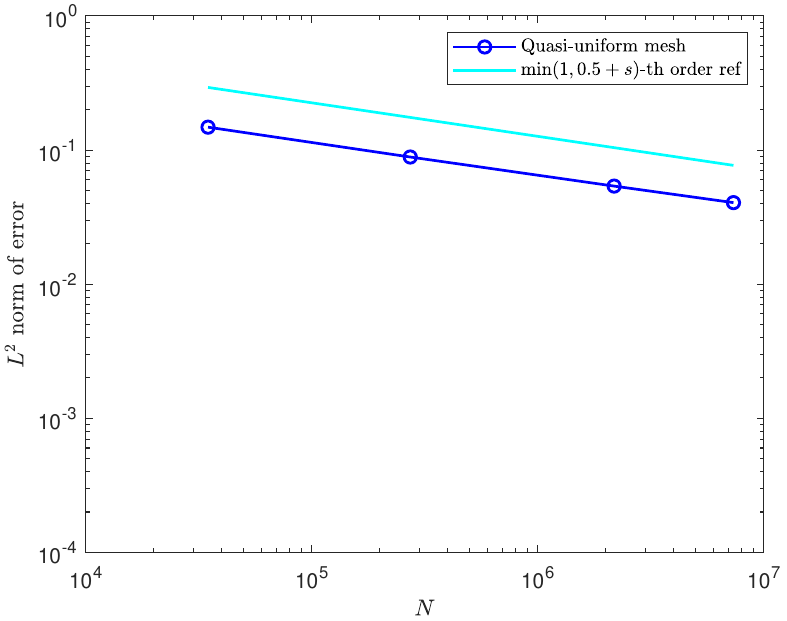}
}
\hspace{-15pt}
\subfigure[Mod. spectral appr., $s=0.50$]{
\includegraphics[width=0.33\linewidth]{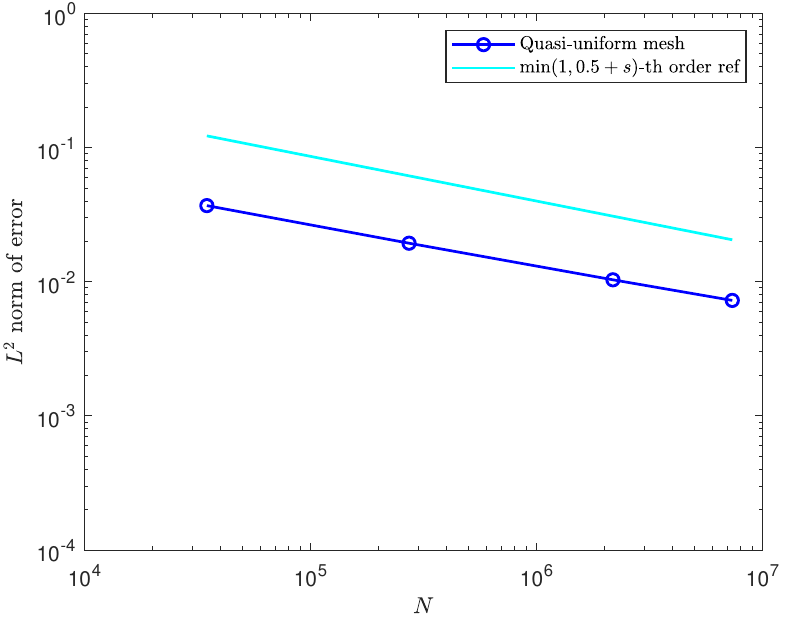}
}
\hspace{-15pt}
\subfigure[Mod. spectral appr., $s=0.75$]{
\includegraphics[width=0.33\linewidth]{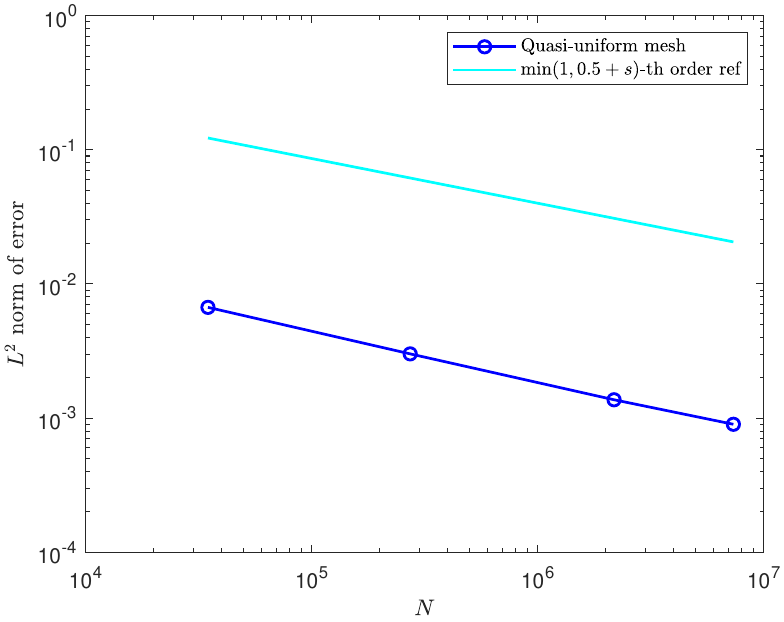}
}
\caption{Example (\ref{main-example}) in 3D. The $L^2$ norm of the error is plotted as a function $N$ for GoFD with various approximations
of the stiffness matrix.}
\label{fig:GoFD_Err-3d-1}
\end{figure}

\begin{figure}[ht!]
\centering
\subfigure[$s=0.25$]{
\includegraphics[width=0.33\linewidth]{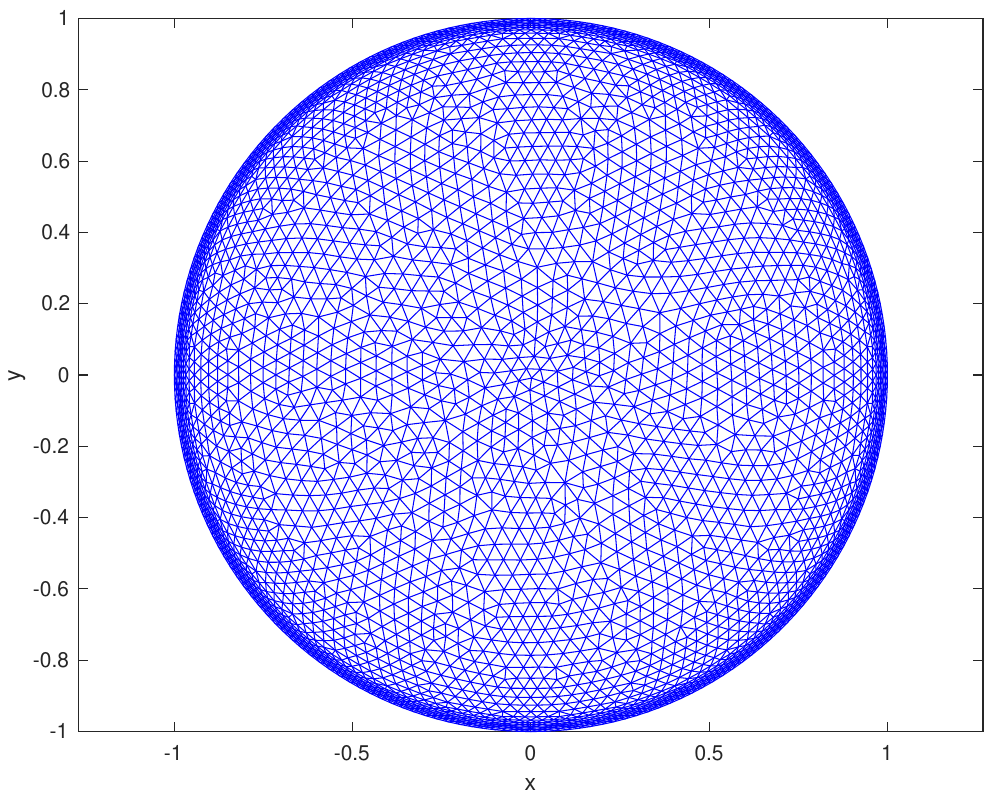}
}
\hspace{-15pt}
\subfigure[$s=0.50$]{
\includegraphics[width=0.33\linewidth]{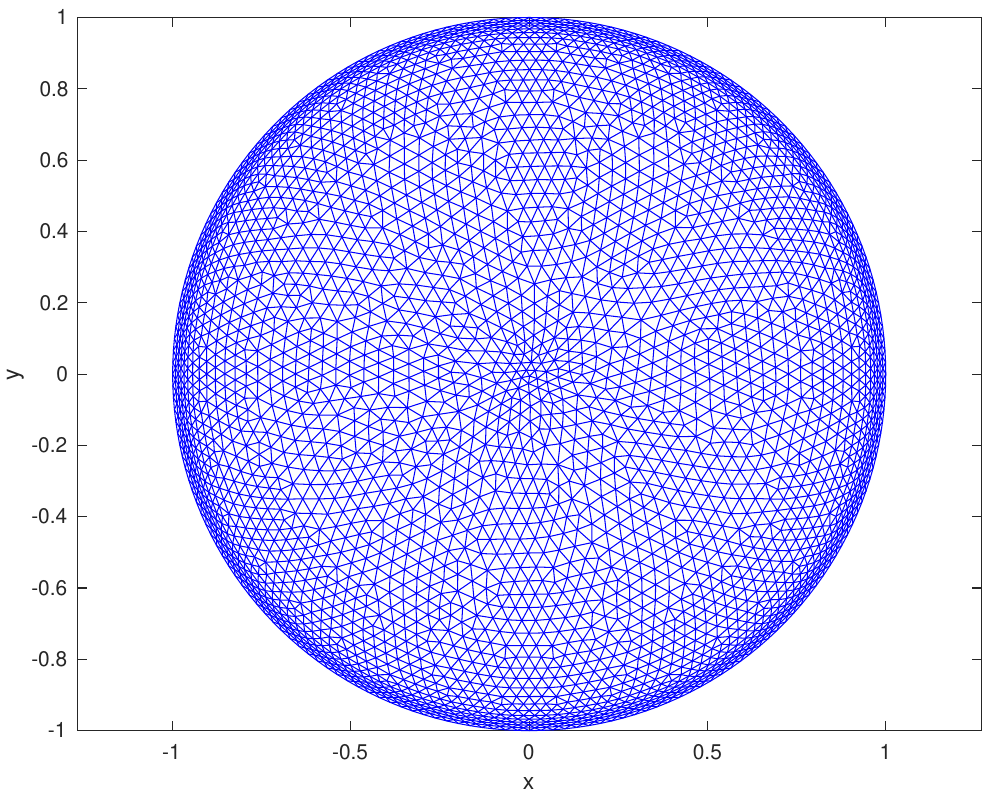}
}
\hspace{-15pt}
\subfigure[$s=0.75$]{
\includegraphics[width=0.33\linewidth]{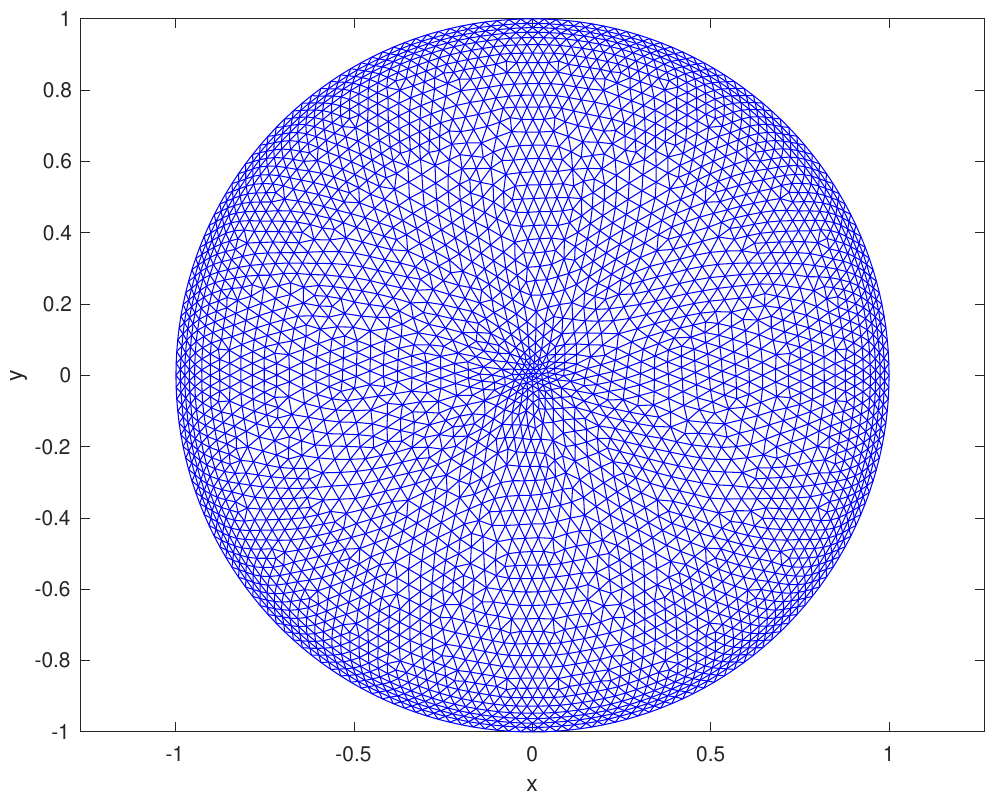}
}
\\
\subfigure[$s=0.25$]{
\includegraphics[width=0.33\linewidth]{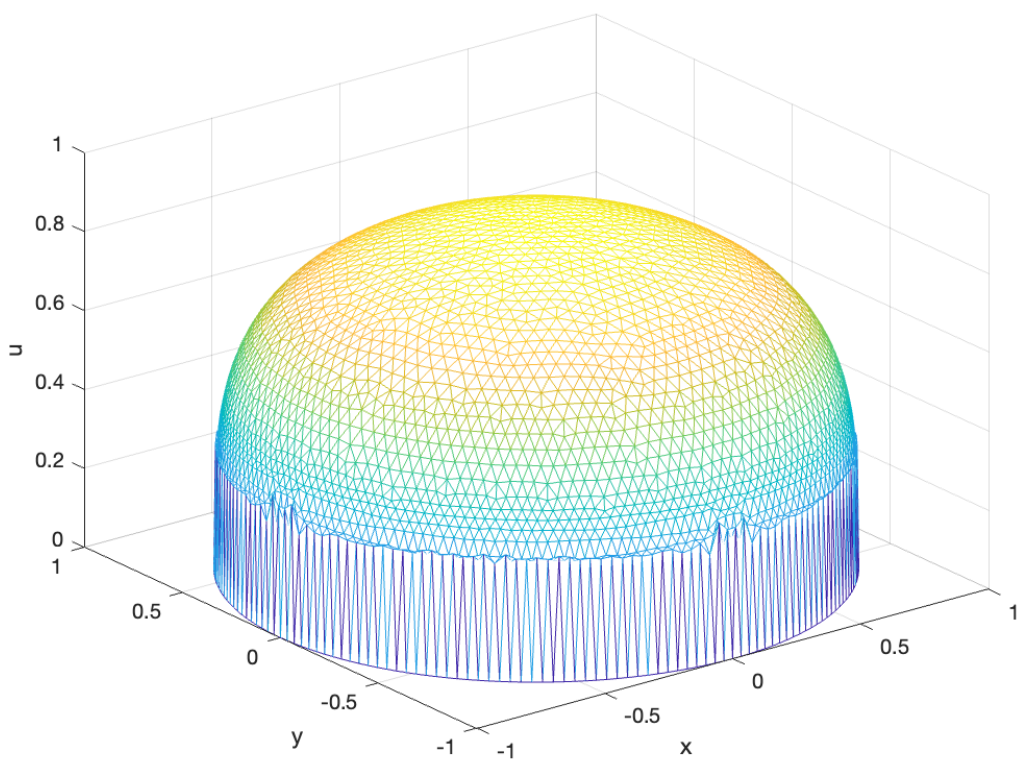}
}
\hspace{-15pt}
\subfigure[$s=0.50$]{
\includegraphics[width=0.33\linewidth]{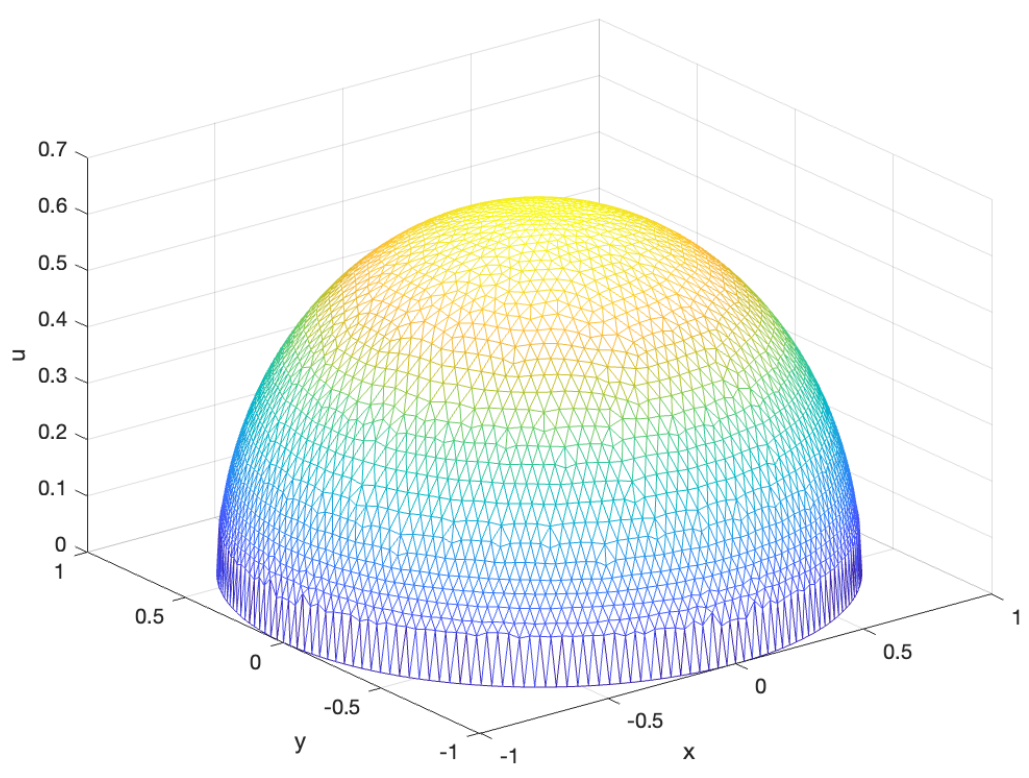}
}
\hspace{-15pt}
\subfigure[$s=0.75$]{
\includegraphics[width=0.33\linewidth]{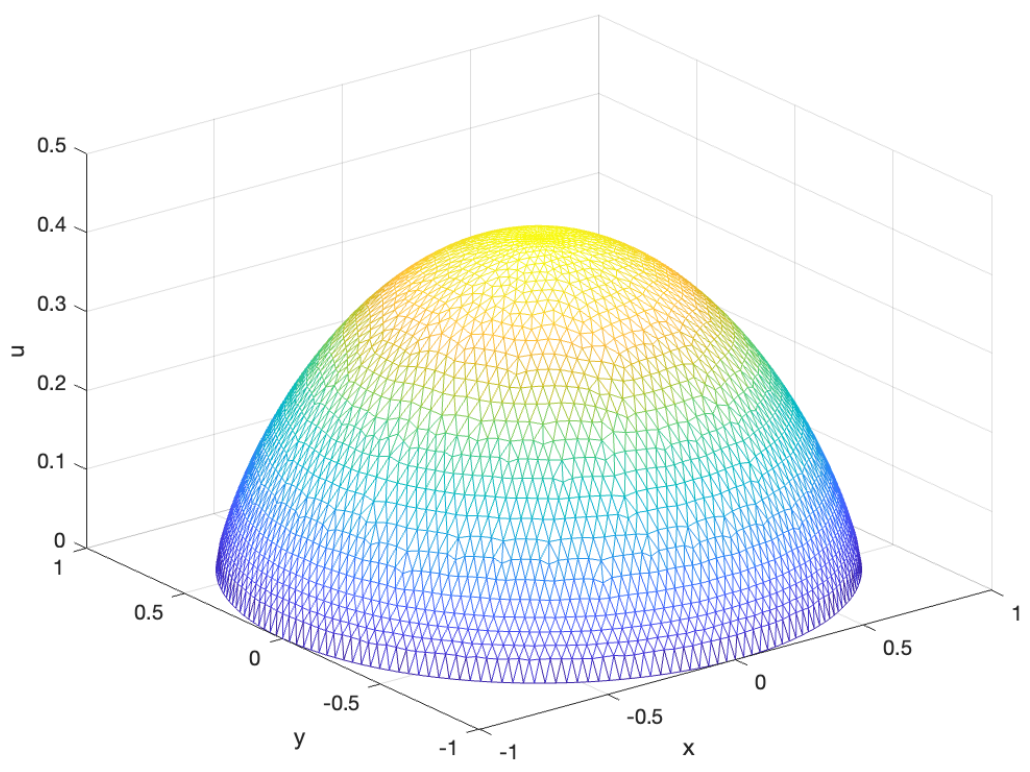}
}
\caption{Example (\ref{main-example}) in 2D. Adaptive meshes and computed solutions obtained with
GoFD with the modified spectral approximation.}
\label{fig:GoFD-mesh-solution-2d}
\end{figure}

\begin{figure}[ht!]
\centering
\subfigure[$s=0.25$]{
\includegraphics[width=0.33\linewidth]{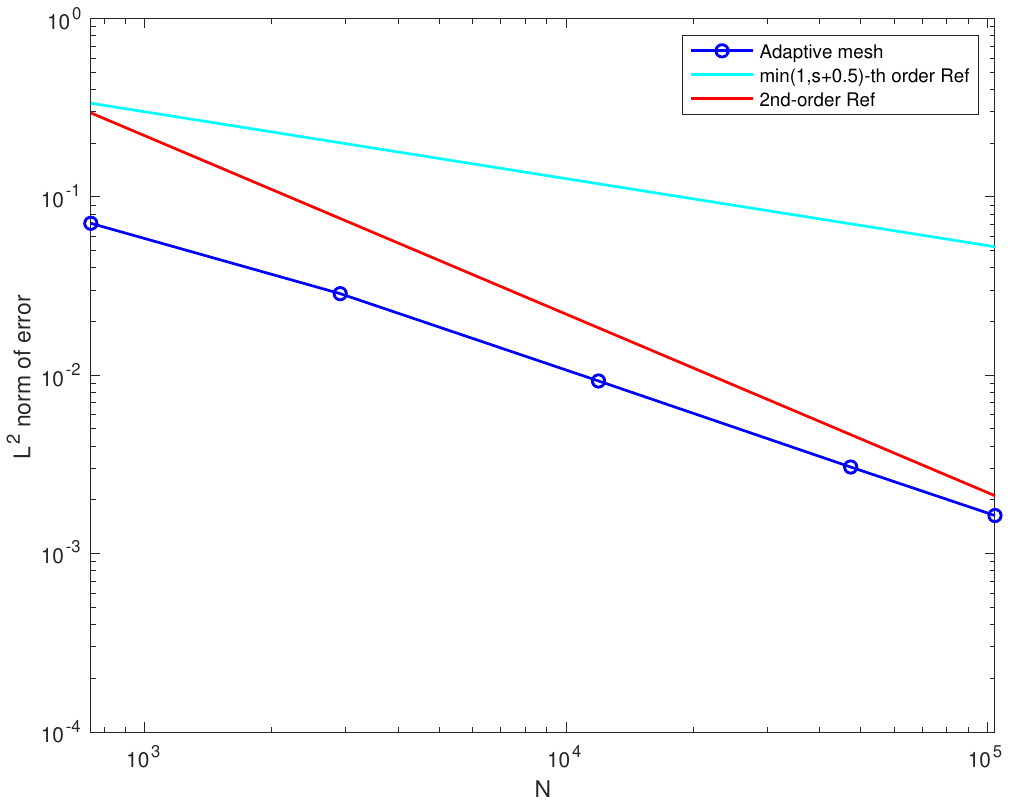}
}
\hspace{-15pt}
\subfigure[$s=0.50$]{
\includegraphics[width=0.33\linewidth]{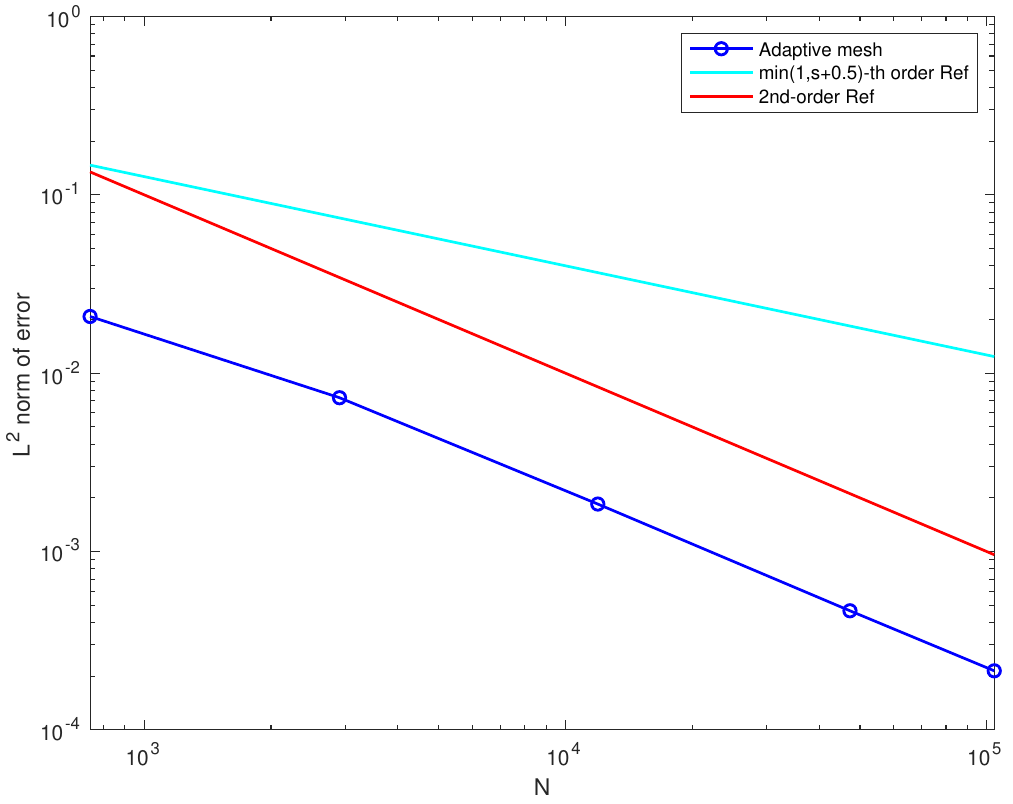}
}
\hspace{-15pt}
\subfigure[$s=0.75$]{
\includegraphics[width=0.33\linewidth]{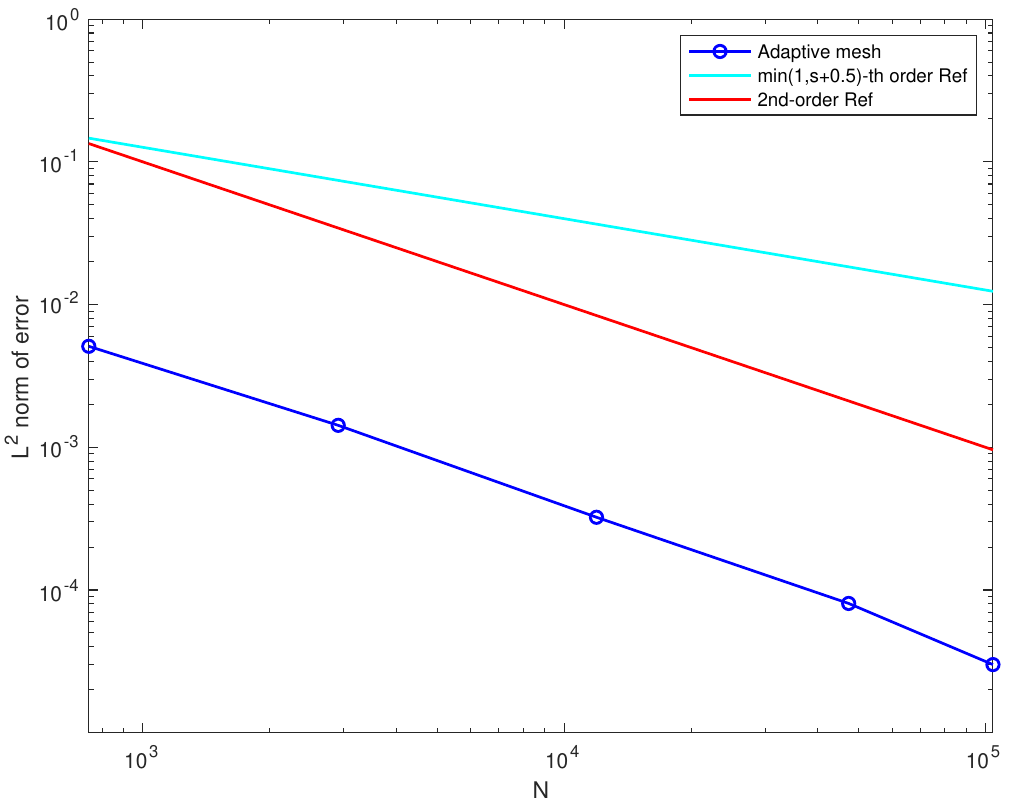}
}
\caption{Example (\ref{main-example}) in 2D. The $L^2$ norm of the error is plotted as a function $N$
for GoFD with the modified spectral approximation for the stiffness matrix and adaptive meshes.}
\label{fig:GoFD_Err-2d-2}
\end{figure}

%
%
%

\begin{figure}[ht!]
\centering
\subfigure[$s=0.25$]{
\includegraphics[width=0.33\linewidth]{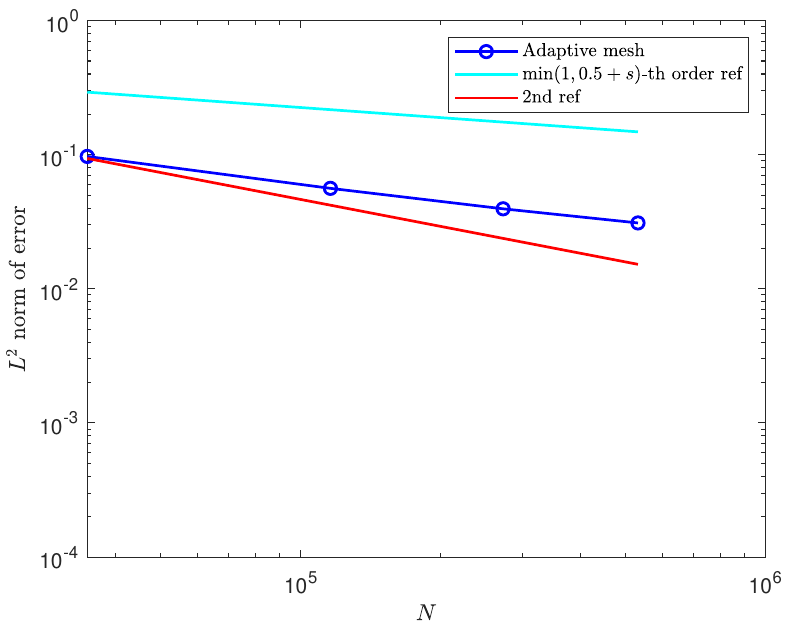}
}
\hspace{-15pt}
\subfigure[$s=0.50$]{
\includegraphics[width=0.33\linewidth]{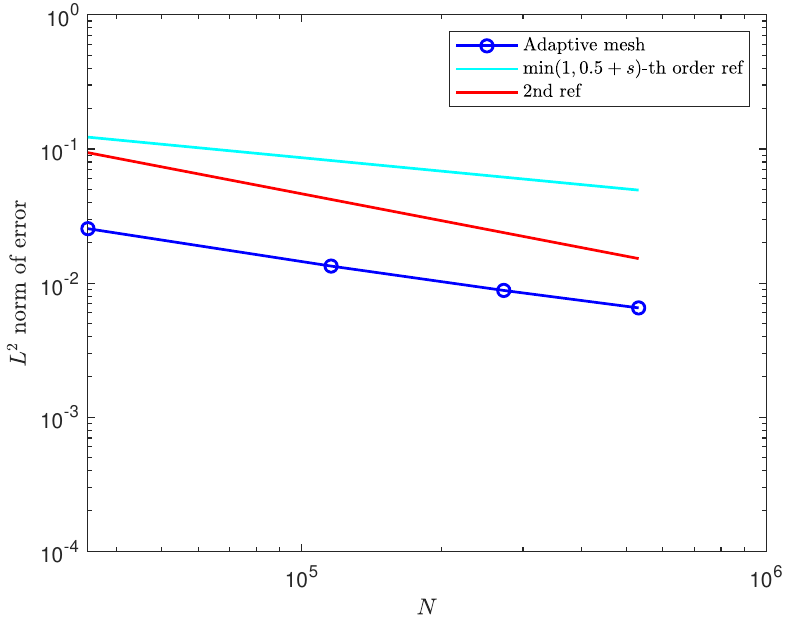}
}
\hspace{-15pt}
\subfigure[$s=0.75$]{
\includegraphics[width=0.33\linewidth]{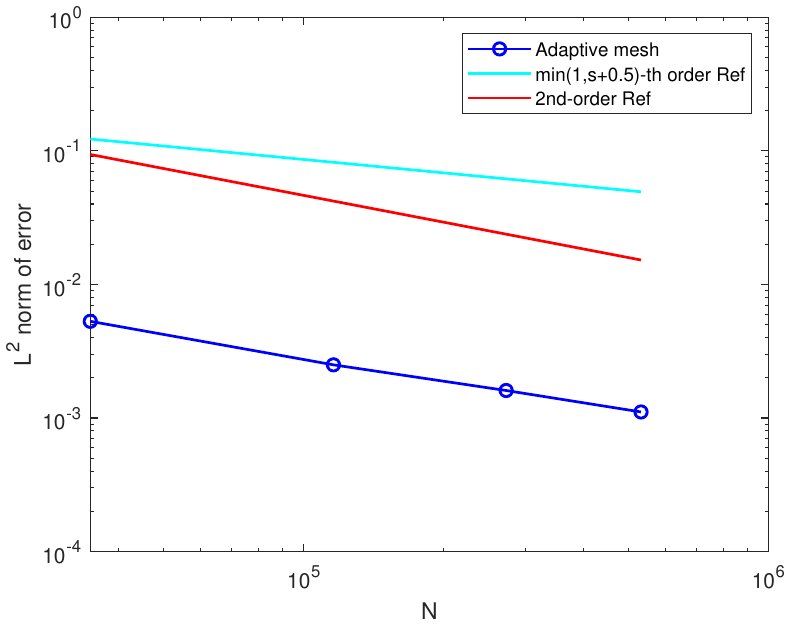}
}
\caption{Example (\ref{main-example}) in 3D. The $L^2$ norm of the error is plotted as a function $N$
for GoFD with the modified spectral approximation for the stiffness matrix and adaptive meshes.}
\label{fig:GoFD_Err-3d-2}
\end{figure}

%

\section{Conclusions}
\label{SEC:conclusions}

In the previous sections we have presented an analysis on the effect of the accuracy in approximating
the stiffness matrix $T$ on the accuracy in the FD solution of BVP (\ref{BVP-1}).
The analysis and numerics show that this effect can be significant, requiring accurate and reasonably
economic approximations to the stiffness matrix. Four approaches for approximating $T$ have been discussed
and shown to work well with GoFD for the numerical solution of BVP (\ref{BVP-1}).
Their properties are summarized in Table~\ref{table:approximate-T-3}.
This study shows that the FFT approach is a good choice for large $s$ (say $s \ge 0.5$) since it is of high
accuracy, has fast asymptotic decay rates away from the diagonal line, works well with both sparse and circulant
preconditioners, and is reasonably economic to compute.
For small $s$, the modified spectral approximation is a good choice since
it leads to high accuracy, works with well with the circulant preconditioner, and is reasonably economic to compute
although it does not have an asymptotic decay  as fast as those for the FFT approximation.

\vspace{20pt}

\section*{Acknowledgments}
J. Shen was supported in part by the National Natural Science Foundation of China through grant
[12101509] and W. Huang was supported in part by the Simons Foundation through grant MP-TSM-00002397.

\bibliographystyle{abbrv}

\end{document}